\documentclass{amsart}
\usepackage[utf8]{inputenc}
\usepackage{amsfonts}
\usepackage{amssymb}
\usepackage{amsmath}
\usepackage{amsthm}
\usepackage{amscd}
\usepackage{amstext}
\usepackage{appendix}
\usepackage{bm}
\usepackage{caption}
\usepackage{enumitem}
\usepackage{comment}
\usepackage{enumitem}
\usepackage{fancyhdr}
\usepackage{hyperref}
\usepackage{graphicx}
\usepackage{lipsum}
\usepackage{mathabx}
\usepackage{mathrsfs}
\usepackage{mathtools}
\usepackage{romannum}
\usepackage[section]{placeins}
\usepackage{stmaryrd}
\usepackage[switch, modulo]{lineno}
\usepackage{tikz}
\usepackage{tikz-cd}
\usepackage{wasysym}
\usepackage{xcolor}
\hypersetup{
    colorlinks,
    linkcolor={black},
    citecolor={black},
    urlcolor={black}
}

\theoremstyle{plain}
    \newtheorem{theorem}{Theorem}[section]
    \newtheorem{lem}[theorem]{Lemma}
    \newtheorem{prop}[theorem]{Proposition}
    \newtheorem{cor}[theorem]{Corollary}
    \newtheorem{rem}[theorem]{Remark}
\theoremstyle{definition}
    \newtheorem{defn}[theorem]{Definition}
    \newtheorem{exmp}[theorem]{Example}

\title{Non-Hausdorff Separation Axioms}
\author{Tianyi Zhou}
\date{}
        
\begin{document}

\clearpage
\thispagestyle{empty}

\begin{abstract}
This note is an introductory survey of non-Hausdorff separation axioms. The main focus is to study properties that are between $T_0$ and $T_1$, properties between $T_1$ and Hausdorff and how the $T_0$-quotient change them and the relation between them. 
\end{abstract}

\maketitle

\vspace{0.5cm}
\begin{center}
\large\bfseries{Introduction}
\end{center}
\vspace{0.5cm}

The purpose of this note is to provide an introductory summary of previous study on topics related to separation axioms in non-Hausdorff spaces. Proofs of the original results will be included but some may be revised if, from the author's perspective, details or explanations are missing. After an introduction of notations and concepts, a review of preorders, weak and strong topologies, and box products will be included. The main contents can be loosely divided to four main parts:

\begin{itemize}
    \item In \textbf{Section 1}, we introduce the specialization preorder. This turns any topological space into a preorder set, which becomes non-trivial only for non-$T_1$ spaces. In these cases, the specialization preorder (or order, for $T_0$ spaces) has strong ties with the applications in which such topologies arise. In \textbf{Section 2} and \textbf{Section 3}, we study $T_0$ and $T_1$ spaces from the viewpoint of the specialization preorder and will discuss several less known equivalent forms of these axioms (see \cite{21}, \cite{22}, \cite{35}).\\

    \item In \textbf{Section 4}, we will look at $R_0$ and $R_1$ spaces, which arise when the specialization preorder is a non-identity equivalence relation (see \cite{18}, \cite{20}, \cite{29}, \cite{31}, \cite{35}). These two properties, in the strict sense, are not separation properties since topologies satisfying them need not distinguish different points. These are ``non-separation" version of $T_1$ and Hausdorff axioms and are satisfied by respectable spaces such as regular spaces (which include all completely regular spaces, and particularly all pseudo-metric spaces). \textbf{Section 5} studies the connection between $R_1$ spaces and compactness. Since $R_1$ spaces are considered almost Hausdorff, it is no surprise that many but not all results known to hold for compact or locally compact Hausdorff spaces can be extended to this setting.\\

    \item In \textbf{Section 6}, we introduce \textbf{locally closed} sets, and define the \textbf{Skula topology} associated with a given topology (see \cite{19}), which emerges as a useful tool for analyzing the original topology. In \textbf{Section 7}, we study $T_D$ spaces. This is an intermediate class between $T_0$ and $T_1$. A $T_D$ space can be defined as one where all singletons are locally closed, or one where all derived sets are closed, but there are several more reformulations (see \cite{12}, \cite{14}, \cite{15}). \textbf{Section 8} focuses on \textbf{sober spaces}. This notion is defined in terms of irreducible spaces and (together with $T_D$ spaces to a lesser extent) are meaningful in fields such as algebra, order theory and logic, where the involved topologies typically have non-closed points (see \cite{15}, \cite{24}, \cite{25}). Sobriety forms the foundation of theory of locales, also known as point-free topology. Hausdorff spaces are sober for trivial reasons while sober spaces are $T_0$ by definition. On the other hand, $T_1$ and sobriety are incomparable, hence ``$T_1$ $+$ sober" spaces form a distinctive class of $T_1$ spaces. However, there is an intimate relationship between $T_D$ and sobriety. In \textbf{Section 9}, similar to the construction of the \textbf{Skula topology}, we introduce the construction of the \textbf{sobrification}, which can embed every $T_0$ space as a dense subspace to a sober space.\\

    \item Starting from \textbf{Section 10}, we study properties between $T_1$ and Hausdorff. The first one is \textbf{locally Hausdorff}. This class finds its way into the study of non-Hausdorff manifolds, modern algebraic geometry, as well as groupoid actions in connection to $C^*$-algebras. We, however, choose to study this class separately due to its distinctive behavior from the other classes that will be studied in \textbf{Section 11}. We will close \textbf{Section 10} by presenting a method of constructing locally Hausdorff spaces by gluing, in a similar vein as forming a manifold out of coordinate patches (see \cite{13}). In \textbf{Section 11}, we study in less detail several classes which lie between $T_1$ and Hausdorff (see, for instance, \cite{16} \cite{26}, \cite{27}, \cite{30}, \cite{32}, \cite{36}, \cite{37}). They relate to Hausdorff spaces in varying degrees. Some of them generalize well-known facts in a Hausdorff space, such as the fact that a compact subset in a Hausdorff space is closed. Others are concerned with convergent sequences and uniqueness of limits.
    
\end{itemize}

\section{Preliminaries}

\noindent
Facts mentioned here will not be used very frequently. However, weak and strong topologies that arise from a family of mappings are important in various portions of general topology and thus deserve to be studied thoroughly. Our standard references on classical general topology are \cite{2} and \cite{10}.\\

\noindent
We shall occasionally used the language of category theory; several constructions in the sequel are best constructed as adjoint functors. Category theory is now ubiquitous in all of mathematics. We will not define relevant concepts here, since there are many excellent accounts on this subject. We suggest \cite{4} for a concise summary of basic concepts, and recommend \cite{6} for a more in-depth look. It will suffice for readers to be acquainted with the term \textit{functor}, \textit{full category}, \textit{natural transformation}, \textit{adjoint functors} and \textit{(co-) reflexive category}. We denote by \textbf{Top} the category whose objects are topological spaces and morphisms are continuous mappings. In general, given $x,y$ two objects in a category $\mathcal{A}$, we use $\operatorname{Hom}_{\mathcal{A}}(x, y)$ to denote all $\mathcal{A}$-morphisms from $x$ to $y$. We use $\operatorname{id}_{\mathcal{A}}$ to denote the identity functor.

\subsection{Notations}

We use $\omega$ to denote the set of non-negative integers, also the least infinite ordinal, and use $\mathbb{N}$ to denote the set of positive integers. We use $\omega_1$ to denote the least uncountable ordinal. \\

\noindent
We often use $(X, \tau)$ to denote a topological space and, for any $x\in X$, we use $\mathcal{N}_x$ to denote the set of all neighborhoods of $x$ and $\mathcal{U}(x)$ to denote the set of all open neighborhoods of $x$. For any $A\subseteq X$, we use $\operatorname{int}_X(A)$ to denote the interior of $A$ in $X$, and $\operatorname{cl}_X(A)$ to denote its closure in $X$. If not specified or when $X$ is fixed, we simply use $A^{\circ}$ and $\overline{A}$ to denote the interior and closure of $A$ in $X$ respectively. For any subspace $Y\subseteq X$, we use $\tau\Big|_Y$ to denote the relative topology in $Y$. In this case, we have $\operatorname{int}_Y(A) = A^{\circ}\cap Y$ and $\operatorname{cl}_Y(A) = \overline{A}\cap Y$. We use $A'$ to denote the set of accumulation points, or the derived set of $A$. We also define:

$$
\mathcal{N}_A = \big\{U\subseteq X: A\subseteq U^{\circ} \big\}
$$
the set of neighborhoods of $A$, and use $\mathcal{U}(A)$ to denote the set of all open subsets that contain $A$.

\subsection{Preorders}

We will follow notations from pages {\cite[p7, p8]{2}}. Given a set $X$, we use $\Delta_X$ to denote the diagonal, namely $\Delta_X = \big\{(x, x):x\in X \big\}$. Given a subset $R\subseteq X\times X$, we define:

$$
R^{-1} = \big\{(y, x):(x, y)\in R \big\}
$$
We call $R$ \textbf{anti-symmetric} if $R\cap R^{-1}\subseteq \Delta_X$, and \textbf{asymmetric} if $R\cap R^{-1}=\emptyset$. We call $R$ \textbf{reflexive} if $\Delta_X\subseteq R$, and \textbf{transitive} if, given any $x, y, z\in X$, $(x, y)\in R$ and $(y, z)\in R$ imply $(x, z)\in R$. We call $R$ a \textbf{preorder} if $R$ is both reflexive and transitive. Note that $R$ is a preorder if and only if $R^{-1}$ is a preorder. We call $R$ a \textbf{partial order} if it is an anti-symmetric preorder. If $R$ is a preorder, then $R\cap R^{-1}$ is the equivalence relation induced by $R$. \\

\noindent
Given a set $X$ and $R\subseteq X\times X$, define:

$$
R\circ R = \big\{(x, z)\in X\times X:\exists\,y\in X\,\text{  such that  }\,(x, y)\in R\,,\, (y, z)\in R \big\}
$$
For any $n\in\omega$, we define $R^0=\Delta_X$, $R^1=R$ and $R^{n+1} = R\circ R^n$. Define:

$$
\hat{R} = \bigcup_{n\in\omega} R^n
$$
the \textbf{transitive reflexive closure} of $R$. Namely, $\hat{R}$ is the smallest preorder that contains $R$.

\begin{defn}\label{Definition 0.1}

Given a preorder $R$ defined on a set $X$, for any $A\subseteq X$, define the $R$-\textbf{saturation} of $A$ as the following:

$$
R[A] = \big\{y\in X:\exists\,x\in A\,\text{  such that  }\,(x, y)\in R \big\}
$$
We also call $A$ \textbf{upward-closed}, an \textbf{up-set} or an $R$-\textbf{saturated set} if $A=R[A]$. Similarly we can define $R^{-1}[A]$ the same way, and call $A$ \textbf{downward-closed} set, or a \textbf{down-set} if $A=R^{-1}[A]$. We use $\Sigma(R)$ to denote the family of all $R$-saturated subsets of $X$. The \textbf{principal up-set} and \textbf{principal down-set} associated with $x\in X$ are defined respectively:

$$
R[x] = \big\{y\in X: (x, y)\in R \big\} \hspace{1cm} R^{-1}[x] = \big\{y\in X: (x, y)\in R^{-1} \big\}
$$
As a result, for any $A\subseteq X$:

$$
R[A] = \bigcup_{a\in A}R[a] \hspace{1cm} R^{-1}[A] = \bigcup_{a\in A} R^{-1}[a]
$$
    
\end{defn}

\begin{defn}
Given two preorder sets $(X, R_X)$, $(Y,R_Y)$, we call a mapping $f:X\rightarrow Y$ $R_X$-$R_Y$ \textbf{monotone} if for any $x_1, x_2\in X$, $\big( f(x_1), f(x_2) \big) \in R_Y$ whenever $(x_1, x_2)\in R$. We denote \textbf{Preord} the category consisting preorder sets and monotone mappings. 
\end{defn}

\begin{prop}\label{Proposition 0.3}

Given a preorder set $(X, R)$ and $A\subseteq X$:

\begin{enumerate}[label = (\alph*)]

    \item $$R[A] = \bigcap\big\{ B\in \Sigma(R): A\subseteq B \big\}$$
    \item $$\Sigma(R^{-1}) = \big\{ A^c: A\in \Sigma(R) \big\}$$
    \item if we define:

    $$
    \langle\,A\,\rangle_R = \bigcup\big\{ B\in \Sigma(R): B\subseteq A \big\}
    $$
    then we have:

    $$
    \langle\, A\,\rangle_R = \bigcup\big\{ R[x]: R[x]\subseteq A \big\} = X\backslash R^{-1} \big[ A^c \big]
    $$
    
\end{enumerate}
    
\end{prop}

\begin{proof}

\begin{enumerate}[label = (\alph*)]

    \item For any $B\in \Sigma(R)$, if $A\subseteq B$, then $R[A]\subseteq R[B]=B$ and this shows the left inclusion. Meanwhile, we have $R[A]\in \Sigma(R)$ and $A\subseteq R[A]$, and hence the equation follows.

    \item Fix $A\in \Sigma(R)$ and $x\in A^c$. Suppose $(x, y)\in R^{-1}$ for some $y\in X$. If $y\in A$, we will have $(a, y)\in R$ for some $a\in A$, which implies $(a, x)\in R$ by the transitivity of $R$. In this case, $x\in A$ since $A\in \Sigma(R)$, which contradicts our assumption. Hence $A^c\in \Sigma(R^{-1})$. Conversely, for any $B\in \Sigma(R^{-1})$, by the same reasoning we can show that $B^c\in \Sigma(R)$.

    \item For any $B\in \Sigma(R)$, $B\subseteq A$ precisely when $R[b]\subseteq A$ for all $b\in B$. Then the first equation follows. By part $(a)$ and $(b)$, we have:

    $$
    R^{-1}\big[ A^c \big] = \bigcap\big\{ B\in \Sigma(R^{-1}): A^c\subseteq B \big\} = \bigcap\big\{ B^c \in \Sigma(R): B^c\subseteq A \big\} = \langle\, A\,\rangle_R
    $$
    
\end{enumerate}
    
\end{proof}

\subsection{Weak topologies}

In $(X, \tau_X)$, given $\big( Y_a, \tau_a \big)_{a\in A}$ a family of topological spaces and $\mathcal{F}_A = \big\{f_a:X\rightarrow Y_a \big\}_{a\in A}$ a family of mappings, the \textbf{weak topology}, or \textbf{initial topology} induced by $\mathcal{F}_A$ is the topology generated by the following subbase:

$$
\big\{ f_a^{-1}(V):a\in A\,,\, V\in \tau_a \big\}
$$
and is denoted by $\tau(\mathcal{F}_A)$. Clearly the weak topology induced by $\mathcal{F}_A$ must be the weakest topology among all topologies $\tau$ such that $f_a$ is $\tau$-$\tau_a$ continuous for all $a\in A$. $\tau(\mathcal{F}_A)$ is characterized as the unique topology $\tau$ that is defined on $X$ and satisfies the following: given any topological space $(Z, \tau_Z)$ and a mapping $g:Z\rightarrow X$, $g$ is $\tau_Z$-$\tau$ continuous if and only if $f_a\circ g$ is $\tau_Z$-$\tau_a$ continuous for all $a\in A$. 

\begin{exmp}

Given $(X_a, \tau_a)_{a\in A}$ a family of topological spaces, define $X=\prod_{a\in A}X_a$ and for each $a\in A$, let $\pi_a: X\rightarrow X_a$ be the projection on $X_a$. The \textbf{product topology} is the weak topology induced by the family $(\pi_a:a\in A)$. If not specified, we will assume that the topology on $\prod_{a\in A}X_a$ is the product topology.
    
\end{exmp}

\noindent
In the previous set-up, assume that for each $a\in A$, $\tau_a$ is the weak topology induced by the following family of functions:

$$
\mathcal{F}_{a, I_a}:\big\{g_{a, i}:Y_a \rightarrow (Z_{a, i}, \tau_{a, i})\big\}_{i\in I_a} 
$$
Define a new index set:

$$
\mathcal{C} = \bigcup_{a\in A}\{a\}\times I_a
$$
Then $\tau(\mathcal{F}_A)$ is the weak topology induced by:

$$
\big\{ g_{(a, i)}\circ f_a:X\rightarrow (Z_{a, i}, \tau_{a, i}) \big\}_{(a, i)\in \mathcal{C}}
$$
This is known as the \textbf{transitivity of weak topologies}.

\begin{defn}\label{Definition 0.3}

A property $\mathcal{P}$ in a topological space is a \textbf{initial property} if the following is true: given a set $X$ endowed with $\tau(\mathcal{F}_A)$ a weak topology induced by $\mathcal{F}_A = \big\{ f_a:X\rightarrow Y_a \big\}_{a\in A}$, whenever each $(Y_a, \tau_a)$ satisfies $\mathcal{P}$, $(X, \tau(\mathcal{F}_A))$ also satisfies $\mathcal{P}$. It follows that an initial property $\mathcal{P}$ is hereditary and productive. In addition, if $(\tau_a:a\in A)$ is a family of topologies in $X$ and $(X, \tau_a)$ satisfies $\mathcal{P}$ for all $a\in A$, then so does $(X, \tau_A)$ where $\tau_A$ is the supremum of $(\tau_a:a\in A)$.

\end{defn}

\subsection{Box topology}

Given a family of topological spaces $(X_a, \tau_a)_{a\in A}$, define $X=\prod_{a\in A}X_a$. The \textbf{box topology} is the topology generated by the following base:

\begin{equation}\label{e0}
\left\{ \bigcap_{a\in A}\pi_a^{-1}(U_a):a\in A\,,\, U_a\in \tau_a \right\}
\end{equation}
We use $\prod^{\text{box}}_{a\in A}X_a$ to denote the topological space $X$ equipped with the box topology, to differentiate it with the product topology. Observe that in $\prod^{\text{box}}_{a\in A}X_a$, each $\pi_a$ is not only continuous but also open. Given $B_a\subseteq X_a$ for each $a\in A$, we have, with respect to the box topology:

$$
\overline{\prod_{a\in A} B_a} = \prod_{a\in A}\overline{B_a}
$$
Fix $\bm{y} = (y_a)_{a\in A}\in X$ and $\alpha\in A$. Define $\iota_{\alpha}: X_{\alpha} \rightarrow X$ in the following way: for each $z\in X_{\alpha}$ and $a\in A$:

$$
\big( \iota_{\alpha}(z) \big)_a = 
\begin{cases}
z, \hspace{1cm} a=\alpha\\
y_a, \hspace{0.85cm} a\neq \alpha
\end{cases}
$$
Then in $\prod^{\text{box}}_{a\in A}X_a$, for each $\alpha\in A$, $\iota_{\alpha}$ is a homeomorphism onto its image. This implies that $X_a$ is homeomorphic to a subspace of $X$ for all $a\in A$. 

\subsection{Strong topology}

We will discuss the dual notion of weak topology. Given a family of  topological spaces $(X_a,\tau_a)_{a\in A}$ and mappings $\mathcal{F}_A = \big( f_a:X_a\rightarrow X \big)_{a\in A}$, the \textbf{strong topology} or \textbf{final topology} $\tau_X$ induced by $\mathcal{F}_A$ is defined by:

$$
\tau_X = \big\{W\subseteq X: f_a^{-1}(W)\in\tau_a \,\forall\,a\in A \big\}
$$
By definition it is the strongest one among all topologies $\tau$ defined on $X$ such that for each $a\in A$, $f_a$ is $\tau_a$-$\tau$ continuous. 

\begin{exmp}

Given a set $X$ and $(\tau_a)_{a\in A}$ a family of topologies defined on $X$, the infimum $\tau = \bigcap_{a\in A}\tau_a$ of all topologies in the strong topology induced by the family of identity mapping $\big( \operatorname{id}_a: (X, \tau_a)\rightarrow X \big)_{a\in A}$
    
\end{exmp}

\begin{exmp}\label{Example 0.5}

Given $\big( (X_{\alpha}, \tau_{\alpha}): \alpha\in \mathcal{A} \big)$ a family of topological spaces, let $X$ be the disjoint union of this family of topological spaces. As a set:

$$
X = \bigcup_{\alpha\in \mathcal{A}} \{\alpha\}\times X_{\alpha}
$$
For each $\alpha\in \mathcal{A}$, define:

$$
j_{\alpha}: X_{\alpha} \rightarrow X, \hspace{0.3cm} x\mapsto (\alpha, x)
$$
Then an arbitrary subset $B\subseteq X$ can be uniquely expressed as $\bigcup_{\alpha\in \mathcal{A}} j_{\alpha} (B_{\alpha})$ where $B_{\alpha} = j_{\alpha}^{-1}(B)$ for each $\alpha\in \mathcal{A}$. The \textbf{sum topology}, denoted by $\tau_X$, is defined as the strong topology induced by $(j_{\alpha})_{\alpha\in \mathcal{A}}$. That is:

$$
\tau_X = \left\{ \bigcup_{\alpha\in \mathcal{A}} j_{\alpha}(U_{\alpha}): \forall\,\alpha \in \mathcal{A}, \hspace{0.3cm} U_{\alpha} \in \tau_{\alpha} \right\}
$$
The resulting topological space $(X, \tau_X)$ is called the \textbf{sum of} $\big( (X_{\alpha}, \tau_{\alpha}): \alpha\in \mathcal{A} \big)$. The sum topology enjoys the following universal property. For any topological space $(Y, \tau_Y)$ and any family of continuous mappings $\big( f_{\alpha}:X_{\alpha} \rightarrow Y \big)_{\alpha\in \mathcal{A}}$, there is a unique mapping $f:X\rightarrow Y$ such that $f\circ j_{\alpha} = f_{\alpha}$. By definition of $\tau_X$, each $j_{\alpha}$ is open and closed. Hence each $X_{\alpha}$, when viewed as a subset of $X$, is also clopen. Also note that if each continuous mapping $f_{\alpha}$ is open, so is their join $f$. A similar conclusion holds for closed maps when $\mathcal{A}$ is finite.
    
\end{exmp}

\subsection{Lattice and frames}

Our modest goal is to introduce a few concepts from the theory of frames, mainly frame homomorphisms and congruences, for the proofs of several theorems included in this note. We recommend \cite{4} and \cite{8} for a proper introduction. Given a topological space $(X, \tau)$, let $\mathcal{P}(X)$ denote the power set. Clearly $\mathcal{P}(X)$ is a Boolean algebra, and $\tau$ is a sublattice of $\mathcal{P}(X)$, which is closed in arbitrary addition (union). The theory of frames is to take a distributive lattice satisfying the extra condition as a starting point, and ask to what extent we can retrieve the underlying topological space solely from the lattice. 

\begin{defn}\label{Definition 0.4}

A \textbf{frame} is a distributive lattice $(L, \vee, \wedge)$ such that for any $(x_i)_{i\in I}\subseteq L$ and $y\in L$:

\begin{enumerate}[label = (F\arabic*)]

    \item\label{F1} the supremum $\bigvee_{i\in I}x_i$ exists in $L$
    \item\label{F2} the following generalized distributive law holds:

    $$
    \left( \bigvee_{i\in I}x_i \right)\wedge y = \bigvee_{i\in I} (x_i\wedge y)
    $$
    
\end{enumerate}

\noindent
Clearly $\tau$ is a distributive lattice, and we call it a \textbf{spatial frame associated with} $X$.
    
\end{defn}

\noindent
A frame $L$ is in fact a complete lattice, since the existence of arbitrary joins implies the existence of arbitrary meets. Hence a frame is bounded, with a maximum $1$ and a minimum $0$. \\

\noindent
Note that there is a crucial difference between an infinite frame and a bounded distributive lattice from the viewpoint of universal algebra. An infinite frame $L$ requires an additional infinitary operation $\bigvee_{\kappa}$, one for each infinite cardinal $\kappa\leq \vert\, \mathcal{P}(L) \,\vert$. Thus a frame is not an algebra. Nevertheless, the algebraic notions of homomorphism and congruence have obvious analogues here.

\begin{defn}\label{Definition 0.5}

Given two frames $L, M$, a mapping $f:L\rightarrow M$ is a frame homomorphism provided that:

\begin{enumerate}[label = (FH\arabic*)]

    \item\label{FH1} for any $\mathcal{U} \subseteq L$:

    $$
    f\left( \bigvee\mathcal{U} \right) = \bigvee\big\{f(u):u\in \mathcal{U} \big\}
    $$

    \item\label{FH2} for any finite $\mathcal{V} \subseteq L$:

    $$
    f\left( \bigwedge \mathcal{V} \right) = \bigwedge\big\{ f(v):v\in \mathcal{V} \big\}
    $$
    
\end{enumerate}

\noindent
Since $\bigwedge\emptyset = 1$, \ref{FH2} is equivalent to the following:

\begin{enumerate}[label = (FH2')]
    \item for any $u, v\in L$:

    $$
    f(u\wedge v) = f(u) \wedge f(v)
    $$
    and $f(1_L) = 1_M$.
    
\end{enumerate}

\noindent
We call $f$ a \textbf{frame isomorphism} if $f$ is invertible and its inverse is also a frame homomorphism. 
    
\end{defn}

\begin{defn}\label{Definition 0.6}

In the set-up of \textbf{Definition \ref{Definition 0.5}}, given an equivalence relation $\theta$ defined on $L$, we call $\theta$ a \textbf{frame congruence} provided that:

\begin{enumerate}[label = (FC\arabic*)]

    \item\label{FC1} given a collection $(u_i, v_i)_{i\in I} \subseteq \theta$, we have:

    $$
    \left( \bigvee_{i\in I} u_i, \bigvee_{i\in U}v_i \right) \in \theta
    $$

    \item\label{FC2} given $(u_1, v_1), (u_2, v_2)\in \theta$, we have:

    $$
    (u_1\wedge u_2, v_1\wedge v_2)\in \theta
    $$
    
\end{enumerate}
Given $f:L\rightarrow M$ a frame homomorphism, the \textbf{kernel} of $f$, denoted by $\operatorname{ker}(f)$, is the following set:

$$
\operatorname{ker}(f) = \big\{(u, v)\in L\times L:f(u) = f(v) \big\}
$$
Clearly $\operatorname{ker}(f)$ is a frame congruence given that $f$ is a frame homomorphism. Later we will show that two notions, the frame homomorphism and frame congruence, are interchangeable, through the notion of kernel.
    
\end{defn}

\begin{exmp}\label{Example 0.11}

Given two topological spaces $(X, \tau_X)$, $(Y, \tau_Y)$, let $f:X\rightarrow Y$ be a continuous mapping. Then the following mapping:

$$
f^{\leftarrow}: \tau_Y\rightarrow \tau_X, \hspace{0.4cm} V\mapsto f^{-1}(V)
$$
is a frame homomorphism, known as the homomorphism induced by $f$.
    
\end{exmp}

\begin{theorem}

In the set-up of \textbf{Definition \ref{Definition 0.6}}, a frame homomorphism $f:L\rightarrow M$ induces a frame congruence, and every frame congruence defined on $L$ is the kernel of a frame homomorphism.
    
\end{theorem}

\begin{proof}

One direction is clear. It remains to show that a frame congruence is the kernel of a frame homomorphism. Fix $\theta$ a frame congruence defined on $L$. In $L\slash\theta$, given $[x]_{\theta}$, $[y]_{\theta}$, define a new join $\wedge_{\theta}$ as follows:

$$
[x]_{\theta} \wedge_{\theta} [y]_{\theta} = [x\wedge y]_{\theta}
$$
According to \ref{FC2} in \textbf{Definition \ref{Definition 0.6}}, $\wedge_{\theta}$ is well-defined. It is also clear by the definition that $\wedge_{\theta}$ is idempotent, associated and commutative. Next define an order $\leq_{\theta}$ on $L\slash \theta$ by the following: given $[x]_{\theta}$, $[y]_{\theta}$:

$$
[x]_{\theta} \leq_{\theta} [y]_{\theta} \hspace{0.5cm} \Longleftrightarrow \hspace{0.5cm} [x]_{\theta} \wedge_{\theta} [y]_{\theta} = [x]_{\theta}
$$
According to definition of $\wedge_{\theta}$, $\leq_{\theta}$ is a partial order. The maximum and minimum in $L\slash \theta$ is $[1]_{\theta}$ and $[0]_{\theta}$. Given $\big( [x_i]_{\theta} \big)_{i\in I} \subseteq L\slash\theta$, \ref{FC1} implies that the supremum of $\big( [x_i]_{\theta} \big)_{i\in I}$ is $\left[ \bigvee_{i\in I}x_i \right]$. We then can conclude that both \ref{F1} and \ref{F2} in \textbf{Definition \ref{Definition 0.4}} holds, and hence $\big( L\slash\theta, \leq_{\theta} \big)$ is a frame. Then the canonical quotient mapping $q:L\rightarrow L\slash\theta$ is the desired frame homomorphism with $\operatorname{ker}(q) = \theta$.
    
\end{proof}

\section{Specialization preorder}

\begin{defn}\label{Definition 1.1}

The \textbf{specialization preorder} $\preccurlyeq$ on a topological space $(X, \tau)$ is defined as follows: given $x, y\in X$, we write $x\preccurlyeq y$ if $\mathcal{N}_x \subseteq \mathcal{N}_y$. We will instead use $\preccurlyeq_X$ or $\preccurlyeq_{\tau}$ when we need to specify $X$ or $\tau$. For each $x\in X$, define:

$$
[x]_{\preccurlyeq} = \big\{y\in X:x\preccurlyeq y \big\}, \hspace{1cm} [x]_{\succcurlyeq} = \big\{y\in X: x\succcurlyeq y \big\}
$$
    
\end{defn}

\begin{prop}\label{Proposition 1.2}

In a topological space $(X, \tau)$, given $x, y\in X$, the following are equivalent:

\begin{enumerate}[label = (\alph*)]
    \item $x\preccurlyeq y\hspace{0.5cm}$
    \item $x\in \overline{\{y\}} \hspace{0.5cm}$
    \item $\overline{\{x\}} \subseteq \overline{\{y\}} \hspace{0.5cm}$
    \item $y\in \bigcap\mathcal{N}_x \hspace{0.5cm}$
    \item $\bigcap\mathcal{N}_y \subseteq \bigcap\mathcal{N}_x \hspace{0.5cm}$
\end{enumerate}

\noindent
Moreover:

$$
[x]_{\succcurlyeq} = \big\{y\in X:x \succcurlyeq y \big\} = \overline{\{x\}}
$$
and:

$$
[x]_{\preccurlyeq} = \big\{y\in X:x\preccurlyeq y\big\} = \bigcap\mathcal{N}_x
$$
    
\end{prop}

\begin{proof}

By definition, the equivalence $(1)\Longleftrightarrow (2)$, $(2)\Longleftrightarrow (3)$ are clear, so is the equivalence $(4)\Longleftrightarrow (5)$. We also have $(1)\implies (5)$ and $(4)\implies (2)$, and hence prove the equivalence among statements $(1)\sim(5)$. The rest of the results follows immediately by $(2)$ and $(4)$.
    
\end{proof}

\begin{defn}\label{Definition 1.3}

In the set-up of \textbf{Definition \ref{Definition 1.1}}, given $x\in X$, we call $x$ $\preccurlyeq$-\textbf{minimal} if for any $y\in X$, $y\preccurlyeq x$ implies $x\preccurlyeq y$. Dually, we call $x$ $\preccurlyeq$-\textbf{maximal} if for any $y\in X$, $x\preccurlyeq y$ implies $y\preccurlyeq x$.
    
\end{defn}

\begin{prop}\label{Proposition 1.4}

In the set-up of \textbf{Definition \ref{Definition 1.3}}, $x$ is $\preccurlyeq$-minimal if and only if $\overline{\{x\}} \subseteq \bigcap\mathcal{N}_x$, and $x$ is $\preccurlyeq$-maximal if and only if $\bigcap\mathcal{N}_x \subseteq \overline{\{x\}}$.
    
\end{prop}

\begin{proof}

The conclusion follows immediately by \textbf{Proposition \ref{Proposition 1.2}}.
    
\end{proof}

\begin{defn}\label{Definition 1.5}

In a topological space $(X, \tau)$, given $x, y\in X$, we say $x\sim y$ if $x\preccurlyeq y$ and $y\preccurlyeq x$. Hence $\sim$ defines a equivalence relation on $X$. If $x\sim y$, we call $x$ and $y$ \textbf{topological indistinguishable}. When we need to specify $X$ or $\tau$, we may instead use $\sim_X$ or $\sim_{\tau}$ as its notation.

\end{defn}

\begin{lem}\label{Lemma 1.6}

In the set-up of \textbf{Definition 1.5}, given $x, y\in X$, the following are equivalent:

\begin{enumerate}[label = (\alph*)]
    \item $x\sim y \hspace{0.5cm}$ \item $\mathcal{N}_x = \mathcal{N}_y \hspace{0.5cm}$ \item $\overline{\{x\}} = \overline{\{y\}} \hspace{0.5cm}$
\end{enumerate}

\noindent
Moreover:

\begin{equation}\label{e2}
[x]_{\sim} = \overline{\{x\}}\cap \bigcap\mathcal{N}_x
\end{equation}
    
\end{lem}

\begin{proof}

The conclusion follows immediately by \textbf{Proposition \ref{Proposition 1.2}}.
    
\end{proof}

\begin{lem}\label{Lemma 1.7}

In the set-up of \textbf{Definition \ref{Definition 1.5}}, the following statements are equivalent:

\begin{enumerate}[label = (\alph*)]
    \item $x$ is $\preccurlyeq$-minimal $\hspace{0.5cm}$ \item $[x]_{\sim}$ is closed $\hspace{0.5cm}$ \item For all $N\in\mathcal{N}_x$, $\overline{\{x\}} \subseteq N$
\end{enumerate}

\noindent
In this case, $[x]_{\sim} = \overline{\{x\}}$.
    
\end{lem}

\begin{proof}

When $x$ is $\preccurlyeq$-minimal, according to (\ref{e2}) and \textbf{Proposition \ref{Proposition 1.4}}, we have $[x]_{\sim} = \overline{\{x\}}$, and hence is closed. We then have $(1)\implies (2)$. When $[x]_{\sim}$ is closed, $\overline{\{x\}} \subseteq [x]_{\sim}$. Therefore, $[x]_{\sim}$ is contained in every neighborhood of $x$ according to (\ref{e2}), or $(2)\implies (3)$. $(3)\implies (1)$ follows immediately by \textbf{Proposition \ref{Proposition 1.4}}.
    
\end{proof}

\noindent
We will close this subsection by introducing a few results of the \textbf{Alexandroff topology} and how such type of topology corresponds to a preorder defined on $X$. We will show that given a preorder set $(X, \leq)$, the family $\Sigma(\leq)$ of all upward-closed subsets (see \textbf{Definition \ref{Definition 0.1}}) induces an Alexandroff topology. An important consequence is that every preorder defined on a set $X$ is the specialization preorder of some topology defined on $X$. After showing the correspondence between an Alexandroff topology and a preorder, we will include one more result that shows an Alexandroff topology is also a strong topology (see \textbf{Section 0.5}). Contents about Alexandroff topology are cited from \cite{11}.

\begin{defn}\label{Definition 1.8}

In a topological space $(X, \tau)$, we call $X$ \textbf{Alexandroff} if $\tau$ is closed under arbitrary intersections. Clearly $\tau$ is Alexandroff if and only if and any arbitrary union of $\tau$-closed subsets is also $\tau$-closed.
    
\end{defn}

\begin{theorem}\label{Theorem 1.9}

In a topological space $(X, \tau)$, let $\Sigma(\preccurlyeq)$ denote the family of all $\preccurlyeq$-upward closed subsets (see \textbf{Definition \ref{Definition 0.1}}) of $X$. Then the following are equivalent:

\begin{enumerate}[label = (\alph*)]

    \item $X$ is Alexandroff.
    \item for any collection of subsets $(A_i)_{i\in I}\subseteq X$:

    $$
    \overline{\bigcup_{i\in I}A_i} = \bigcup_{i\in I} \overline{A_i}
    $$

    \item for any $A\subseteq X$ and $x\in X$, if $x\in \overline{A}$, then $x\in \overline{\{y\}}$ for some $y\in A$.

    \item $\Sigma(\preccurlyeq) = \tau$.

    \item for some preorder $\leq$ defined on $X$, $\Sigma(\leq) = \tau$

    \item there exists $B_x\in\mathcal{N}_x$ such that $B_x\subseteq N$ for all $N\in \mathcal{N}_x$. Namely, $\{B_x\}$ forms a singleton neighborhood basis of $x$.
    
\end{enumerate}
    
\end{theorem}

\begin{proof}

The proof consists of two parts. In the first part we will show the implication$(a) \implies (b) \implies (c) \implies (a)$ and in the second part we will show the implication $(a) \implies (f) \implies (d) \implies (e) \implies (a)$.

\begin{itemize}

\item Assuming $X$ is Alexandroff, then arbitrary union of closed sets is also closed. Then $(b)$ follows immediately. When $(b)$ is true, $(c)$ also holds if we write $A=\bigcup_{a\in A}\{a\}$. If $(c)$ is true, given $A\subseteq X$, together with \textbf{Proposition \ref{Proposition 1.2}}, we have:

$$
\overline{A} = \bigcup_{a\in A} \overline{\{a\}} = \bigcup_{a\in A} [a]_{\succcurlyeq} = [A]_{\succcurlyeq}
$$
which implies that given a family of subsets $(A_i)_{i\in I}$, we have:

$$
\bigcup_{i\in I}\overline{A_i} = \bigcup_{i\in I}\big[ A_i \big]_{\succcurlyeq} = \left[ \bigcup_{i\in I} A_i \right]_{\succcurlyeq} = \overline{\bigcup_{i\in I}A_i}
$$
Therefore when $(c)$ is true, $X$ is Alexandroff.

\item When $X$ is Alexandroff, for each $x\in X$, $B_x = \bigcap\mathcal{U}(x)$ is the desired open neighborhood such that $\{B_x\}$ forms a singleton neighborhood basis of $x$. When $(f)$ is true, observe that for any $A\subseteq X$ according to \textbf{Proposition \ref{Proposition 1.2}}, we have:

$$
[A]_{\preccurlyeq} = \bigcup_{a\in A}[a]_{\preccurlyeq} = \bigcup_{a\in A}\bigcap\mathcal{N}_a = \bigcup_{a\in A}B_a\in \tau
$$
Then we have $\Sigma(\preccurlyeq) \subseteq \tau$. Since all $\tau$-open sets are $\preccurlyeq$-saturated, we have $\Sigma(\preccurlyeq) = \tau$. The implication $(d)\implies (e)$ is immediate. Assuming $(e)$ is true, we will show that $\Sigma(\leq)$ is close under arbitrary intersection. Fix $(A_i)_{i\in I} \subseteq \Sigma(\leq)$ and suppose this family has non-empty intersection. For any $y\in X$. Suppose $y\preccurlyeq x$ for some $x\in \bigcap_{i\in I}A_i$. Then:

$$
y\in [x]_{\leq} = \bigcap_{i\in I}[x]_{\leq} \subseteq \bigcap_{i\in I}[A_i]_{\leq} = \bigcap_{i\in I}A_i
$$
which implies $\bigcap_{i\in I}A_i$ is also $\leq$-saturated.

\end{itemize}
    
\end{proof}

\begin{rem}\label{Remark 1.10}

In the set-up of \textbf{Theorem \ref{Theorem 1.9}}, assuming $X$ is Alexandroff, we can see that for any $x\in X$, $B_x = \bigcap\mathcal{N}_x = [x]_{\geq}$ according to \textbf{Proposition \ref{Proposition 1.2}}. Since each point has a finite neighborhood basis, $X$ is necessarily first countable. Also, the closure and interior operator can be written as follows: for any $A\subseteq X$:

$$
\overline{A} = [A]_{\leq}, \hspace{1cm} A^{\circ} = \big( \big[ A^c \big]_{\leq} \big)^c
$$
In this case, a preorder set $(X, \leq)$ has a natural topology $\Sigma(\leq)$. In this case, by \textbf{Theorem \ref{Theorem 1.9}}, $\leq$ necessarily coincides with the specialization preorder associated with $\Sigma(\leq)$. 
    
\end{rem}

\begin{theorem}

Given a topological space $(X, \tau)$, let $\mathcal{F}(X)$ be the family of all finite subsets. Let $\tau_{\mathcal{F}}$ denote the strong topology induced by the family of inclusion mappings $\big\{ \iota_E: E\in \mathcal{F}(X) \big\}$. If we let $\preccurlyeq$ denote the specialization preorder associated with $\tau$, then $\Sigma(\preccurlyeq) = \tau_{\mathcal{F}}$. Hence $X$ is Alexandroff precisely when $\tau = \tau_{\mathcal{F}}$.
    
\end{theorem}

\begin{proof}

By definition (see \textbf{Section 0.5}), a set $A\subseteq X$ is $\tau_{\mathcal{F}}$-closed if and only if for all $F\in \mathcal{F}(X)$, $A\cap F$ is relatively closed in $F$. Then all $\tau$-closed subsets are $\tau_{\mathcal{F}}$-closed.
By \textbf{Proposition \ref{Proposition 0.3}}, it suffices to show that the set of $\tau_{\mathcal{F}}$-closed sets are precisely those $\succcurlyeq$-saturated sets. Fix $A\subseteq X$ such that $A=[A]_{\succcurlyeq}$. Then according to \textbf{Proposition \ref{Proposition 1.2}}, for any $F\in\mathcal{F}$:

$$
A \cap F = [A]_{\succcurlyeq}\cap F = \bigcup_{a\in A} [a]_{\succcurlyeq}\cap F = \bigcup_{a\in A} \big( \overline{\{a\}}^{\tau} \cap F \big)
$$
which is a finite union of $\tau\Big|_F$-closed subsets, and hence $\tau\Big|_F$-closed. Therefore $A$ is $\tau_{\mathcal{F}}$-closed. Conversely, given $C\subseteq X$ a $\tau_{\mathcal{F}}$-closed set and $y\in X$, suppose $y\preccurlyeq x$ for some $x\in C$. Assume that there exists $z\in \overline{\{y\}}$ such that $z\notin C$. By \textbf{Proposition \ref{Proposition 1.2}}, we have $z\preccurlyeq y \preccurlyeq x$, or $z\preccurlyeq x$ by transitivity. Then there exists $O\subseteq X$ a $\tau_{\mathcal{F}}$-open subset such that $z\in O$ and $O$ is disjoint from $C$. By definition, $O\cap\{z\} = \{z\}$ is $\tau$-open, which implies $z=x$ since $\mathcal{N}_z \subseteq \mathcal{N}_x$. From this we obtain a contradiction since $x\in C$. Therefore $\overline{\{y\}}\subseteq C$ for all $y\preccurlyeq x$. In particular, $[x]_{\succcurlyeq} \subseteq C$, and hence $C\in \Sigma(\succcurlyeq)$. Then as a consequence of \textbf{Theorem \ref{Theorem 1.9}}, $\tau=\tau_{\mathcal{F}}$ precisely when $X$ is Alexandroff.
    
\end{proof}

\begin{rem}

We can recast the previous discussion in terms of category theory. Let \textbf{ATop} be the full subcategory of \textbf{Top} whose objects are Alexandroff spaces. It is clear that a continuous mapping between two topological spaces preserve specialization preorder. This implies there is a functor \textbf{S}: \textbf{Top} $\rightarrow$ \textbf{Preord} which sends $(X, \tau)$ to $(X, \preccurlyeq_X)$ and a continuous map to itself. In general, given a preorder set $(X, \leq)$ and a topological space $(Y, \tau)$, one can check that a mapping $f:X\rightarrow Y$ is $\Sigma(\leq)-\tau$ continuous precisely when $f$ is $\leq-\preccurlyeq_Y$ monotone.\\

\noindent
Therefore, given $(X, \leq_X)$, $(Y, \leq_Y)$ two preorder sets, if $f:X\rightarrow Y$ is $\leq_X$-$\leq_Y$ monotone, $f$ is $\Sigma(\leq_X)$-$\Sigma(\leq_Y)$ continuous. Hence we have another functor \textbf{U}: \textbf{Preord} $\rightarrow$ \textbf{ATop} which sends $(X, \leq_X)$ to $\big(X, \Sigma(\leq_X) \big)$, and a continuous mapping to itself. By \textbf{Theorem \ref{Theorem 1.9}}, \textbf{SU} is the identity functor of \textbf{Preord} and $\textbf{US}$, when restricted to \textbf{ATop}, is the identity functor of \textbf{ATop}. This shows \textbf{U} is an isomorphism from \textbf{Preord} to \textbf{ATop} with inverse $\textbf{S}\big|_{\textbf{ATop}}$. Moreover, for any preorder set $(X, \leq_X)$ and topological space $(Y, \tau)$:

$$
\operatorname{Hom}_{\textbf{Top}}\Big( \textbf{U}\big( X, \leq_X \big) \,,\, (Y, \tau) \Big) = \operatorname{Hom}_{\textbf{Preord}} \Big( \big( X, \leq_X \big) \,,\, \textbf{S}(Y, \tau) \Big)
$$
Therefore \textbf{U} and \textbf{S} form an adjoint pair. We can now conclude that \textbf{ATop} is a \textbf{co-reflective subcategory} of \textbf{Top} (namely a subcategory whose inclusion functor has a right adjoint).
    
\end{rem}

\section{\texorpdfstring{$T_0$}{T0} spaces}

In this section we will introduce and provide several characterizations of a $T_0$ spaces. The first result shows that $X$ is $T_0$ if and only if the specialization preorder is a partial order. Then we may speak of the specialization order when $X$ is $T_0$.

\begin{defn}

A topological space $(X, \tau)$ is $T_0$ if for any two different points $x, y\in X$, there exists an open set that contains exactly one of them.
    
\end{defn}

\begin{prop}\label{Proposition 2.1}

Given a topological space $(X, \tau)$, the following are equivalent:

\begin{enumerate}[label = (\arabic*)]
    \item $X$ is $T_0$
    \item the specialization preorder $\preccurlyeq$ is anti-symmetric. Namely, if $x\preccurlyeq y$ and $y\preccurlyeq x$, then $x=y$.
    \item for any $x, y\in X$, $\mathcal{N}_x = \mathcal{N}_y$ implies $x=y$
\end{enumerate}
    
\end{prop}

\begin{proof}

The equivalence $(2)\Longleftrightarrow (3)$ follows immediately by \textbf{Lemma \ref{Lemma 1.6}}. It remains to show that $(1)$ is equivalent to $(3)$. First we assume $X$ is $T_0$. Given $x, y\in X$, by definition of a $T_0$ space, there exists an open subset $N\in \mathcal{N}_x \backslash \mathcal{N}_y$, which implies $\mathcal{N}_x \neq \mathcal{N}_y$. Conversely, given two different points $x, y$, we have $\mathcal{N}_x\neq \mathcal{N}_y$. Then $\mathcal{N}_x \backslash \mathcal{N}_y$ or $\mathcal{N}_y\backslash \mathcal{N}_x$ is non-empty, and hence there exists an open subset that contains exactly one of $x$ and $y$.
    
\end{proof}

\begin{prop}\label{Propopsition 2.2}

In the set-up of \textbf{Proposition \ref{Proposition 2.1}}, given $x\in X$, $[x]_{\sim} = \{x\}$ if and only if $\{x\}'$ (the accumulation points of $\{x\}$) is a union of closed sets.
    
\end{prop}

\begin{proof}

First assume that $[x]_{\sim} = \{x\}$. Observe that $\{x\}' = \overline{\{x\}} \backslash \{x\}$. Given $y\in \{x\}'$, by \textbf{Proposition \ref{Proposition 1.2}}, we have $y\preccurlyeq x$ and $y\neq x$. Since we assume that $[x]_{\sim} = \{x\}$, we must have $x\npreccurlyeq y$, or $x\notin \overline{\{y\}}$. This implies that for any $y\in \{x\}'$:

$$
\overline{\{y\}} \subseteq \overline{\{x\}} \backslash \{x\} = \{x\}'
$$
which implies:

$$
\{x\}' = \bigcup_{y\in \{x\}'} \overline{\{y\}}
$$
Conversely, assume that $\{x\}'$ is a union of closed sets. Let $y\in X$ be different from $x$. If $y\notin \overline{\{x\}}$, then $y\npreccurlyeq x$ according to \textbf{Proposition \ref{Proposition 1.2}}, or $y\notin [x]_{\sim}$. Otherwise, since $\{x\}'$ is a union of closed sets, from $y\in \overline{\{x\}}$, we must have $\overline{\{y\}} \subseteq \{x\}'$. This implies $x\notin \overline{\{y\}}$, or $x\npreccurlyeq y$. Now we can conclude that there are no points different from $x$ in $[x]_{\sim}$.
    
\end{proof}

\begin{defn}\label{Definition 2.3}

Given a topological space, let $\sim$ denote the equivalence relation defined in \textbf{Definition \ref{Definition 1.5}}. Define $X_0 = X\slash \sim$ and let $q:X\rightarrow X_0$ be the quotient mapping. Let $\tau_0$ denote the quotient topology on $X_0$. We then call $(X_0, \tau_0)$ the $T_0$-\textbf{quotient} of $(X, \tau)$.
    
\end{defn}

\begin{theorem}\label{Theorem 2.4}

In the set-up of \textbf{Definition \ref{Definition 2.3}}:

\begin{enumerate}[label = (\alph*)]

    \item $q$ is both open and closed
    \item Consider the following mapping:

    $$
    \phi:\tau\rightarrow \tau_0, \hspace{0.4cm} U\mapsto q(U)
    $$
    The mapping $\phi$ is bijective and satisfies that given a family of $\tau$-open subsets $(U_i)_{i\in I}$ and for any two different $i, j\in I$:

    \begin{equation}\label{e3}
    \phi\left( \bigcup_{i\in I}U_i \right) = \bigcup_{i\in I}\phi(U_i), \hspace{1cm} \phi(U_i\cap U_j) = \phi(U_i)\cap \phi(U_j)
    \end{equation}

    \item For any $x, y\in X$, $x\preccurlyeq_X y$ if and only if $q(x) \preccurlyeq_{X_0} q(y)$. It follows that $(X_0, \tau_0)$ is $T_0$.

    \item $X$ is $T_0$ if and only if $q$ is injective. In this case, $q$ is a homeomorphism according to $(a)$.

    \item The $T_0$-quotient has the following universal property: any continuous mapping from $X$ to a $T_0$ space factors through $q$. That is, if $f:X\rightarrow Y$ is continuous and $Y$ is $T_0$, there exists a unique $f_0:X_0 \rightarrow Y$ such that $f=f_0\circ q$.
    
\end{enumerate}
    
\end{theorem}

\begin{proof}

Note that for any $A\subseteq X$, $q^{-1}\big( q(A) \big) = [A]_{\sim}$. According to \textbf{Lemma \ref{Lemma 1.6}}, if $A$ is open or closed, then $[A]_{\sim} = A$, or $q^{-1}\big( q(A) \big) = A$. This implies that $q$ is open and closed, and proves $(a)$. Since open subsets of $X$ are $\sim$-saturated, given $U, V\in \tau$, if $q(U) = q(V)$, we must have $U=V$. By the definition of $\tau_0$, $\phi$ is then a bijection, and (\ref{e3}) follows immediately. This proves $(b)$. By $(b)$, we have:

\begin{equation}\label{e5}
\mathcal{U}\big( q(x) \big) = \big\{ q(U):U\in \mathcal{U}(x) \big\}
\end{equation}
Now fix $x, y\in X$. If $x\preccurlyeq y$, by \textbf{Proposition \ref{Proposition 1.2}} and the fact that $q$ is $\tau$-$\tau_0$ continuous, we have:

$$
q(x)\in q\big( \overline{\{y\}} \big) \subseteq \overline{\{q(y)\}}
$$
or $q(x)\preccurlyeq q(y)$. Conversely, suppose that $q(x)\preccurlyeq q(y)$, or $q(x)\in \overline{\{q(y)\}}$. By (\ref{e5}), we have that for any $U\in \mathcal{U}(x)$, $q(y)\in q(U)$, and hence $y\in U$ as $U$ is $\sim$-saturated. This implies $y\in U$ for all $U\in\mathcal{U}(x)$, or $x\preccurlyeq y$ again by \textbf{Proposition \ref{Proposition 1.2}}. Together with \textbf{Proposition \ref{Proposition 1.2}} and \textbf{Lemma \ref{Lemma 1.6}}, we have:

$$
\begin{aligned}
\big[ q(x) \big]_{\sim}
& = \big\{ q(y)\in X_0: q(y) \preccurlyeq q(x) \big\} \cap \big\{ q(y)\in X_0: q(x)\preccurlyeq q(y) \big\} \\
& = q\left( \big\{ y\in X: y\preccurlyeq x \big\} \right) \cap q\left( \big\{ y\in X: x\preccurlyeq y \big\} \right)\\
& = q\big( [x]_{\sim} \big) = \{q(x)\}
\end{aligned}
$$
which implies $X_0$ is $T_0$. This proves $(c)$. If $X$ is $T_0$, clearly $[x]_{\sim} = \{x\}$, and hence $q$ is injective. Conversely, when $q$ is injective, assume there is $y\in[x]_{\sim} \backslash \{x\}$. According to \textbf{Lemma \ref{Lemma 1.6}}, $y\in [y]_{\sim} \cap [x]_{\sim}$. Meanwhile, we have $q^{-1}\big( q(x) \big) \neq q^{-1}\big( q(y) \big)$, or $[x]_{\sim}$ is disjoint from $[y]_{\sim}$. This leads to a contradiction, and hence $X$ is $T_0$ when $q$ is injective. \\

\noindent
For part $(e)$, first given a continuous mapping $f:X\rightarrow Y$, observe that if $x,y\in X$ satisfies $x\preccurlyeq_X y$, then by \textbf{Proposition \ref{Proposition 1.2}}:

$$
f(x)\in f\big( \{\overline{\{y\}} \big)\subseteq \overline{\{f(y)\}}
$$
which implies $f(x)\preccurlyeq_Y f(y)$. Then for a fixed $x\in X$, by \textbf{Proposition \ref{Proposition 1.2}}, we have:

$$
\begin{aligned}
f\big( [x]_{\sim} \big)
& = f\big( \{y\in X:y\preccurlyeq x\} \cap \{y\in X: x\preccurlyeq y\} \big) \\
& \subseteq \big\{ y'\in Y: y'\preccurlyeq f(x) \big\} \cap \big\{y'\in Y: f(x)\preccurlyeq y' \big\} \\
& = \big[ f(x) \big]_{\sim} = \{f(x)\}
\end{aligned}
$$
since $Y$ is $T_0$. Therefore, the following mapping is well-defined:

$$
f_0: X_0\rightarrow Y, \hspace{0.3cm} q(x)\mapsto f(x)
$$
Indeed, by our previous reasoning, given $x, y\in X$ that satisfies $q(x)=q(y)$, we must have $f_0(x) = f_0(y)$. Next we will show that $f_0$ is continuous. Suppose $\big( q(x_i) \big)_{i\in I}$ is a net converging to $q(x)$ in $X_0$. By (\ref{e5}), an open neighborhood of $q(x)$ is equal to $q(U)$ for some $U\in \mathcal{U}(x)$. Fix $U\in\mathcal{U}(x)$ and suppose for some $J\in I$, $q(x_i)\in q(U)$ for all $i\geq J$. This implies:

$$
q^{-1} \left( \bigcup_{i\geq J} \big\{ q(x_i) \big\} \right) = \bigcup_{i\geq J} q^{-1}\big( \{q(x_i)\} \big) = \bigcup_{i\geq J}[x_i]_{\sim} \subseteq U \hspace{1cm} \Longrightarrow \hspace{1cm} (x_i)_{i\geq J} \subseteq U
$$
This implies that $(x_i)_{i\in I}$ converges to $x$ where each $x_i$ is one representative in $[x_i]_{\sim}$. We then have $f_0\big( q(x_i) \big)$ converges to $f_0\big( q(x) \big)$ as $f$ is continuous. The uniqueness of $f_0$ follows by $f_0\circ q = f$.
    
\end{proof}

\begin{rem}\label{Remark 2.6}

In the set-up of \textbf{Theorem \ref{Theorem 2.4}}, given a continuous mapping $f:(X, \tau_X) \rightarrow (Y, \tau_Y)$, by \textbf{Theorem \ref{Theorem 2.4}(e)}, there exists a unique continuous mapping $f_0:X_0 \rightarrow Y_0$ such that the following diagram commutes:

$$
\begin{tikzcd}
X \arrow{r}{f} \arrow{d}{q_X} & Y \arrow{d}{q_Y}\\
X_0  \arrow{r}{f_0} & Y_0 
\end{tikzcd}
$$
Let $\textbf{T}_0$ denote the full subcategory of \textbf{Top} consisting of $T_0$ spaces. With the diagram above, we can conclude the $T_0$-quotient defines a functor \textbf{Top} $\rightarrow$ $\textbf{T}_0$ that is left adjoint to the inclusion functor $\textbf{T}_0$ $\rightarrow$ \textbf{Top}. This says $\textbf{T}_0$ is a \textbf{reflective subcategory} of \textbf{Top} (see \cite{34} for more reflective subcategories in \textbf{Top}).
    
\end{rem}

\noindent
Passage to the $T_0$-quotient reduces almost all problems to the setting of $T_0$ spaces. We will later see several properties shared by both $X$ and $X_0$. Given $\mathcal{P}$ a property in topological spaces, if $\mathcal{P}$ implies $T_0$, we can in principal define a new property $r(\mathcal{P})$: 

$$
X\,\text{  satisfies  }\,r(\mathcal{P})\,\text{  if  }\,X_0\,\text{  satisfies  }\,\mathcal{P}
$$
Moreover, we can easily show that $X$ satisfies $\mathcal{P}$ if and only if $X$ is $T_0$ and satisfies $r(\mathcal{P})$. Thanks to the $T_0$ quotient, there is no absolute need to introduce such new properties. However, in the proofs of certain facts concerning $\mathcal{P}$, $T_0$ may have no roles to play, so in fact $r(\mathcal{P})$ is at the heart of the issue. Our policy is to avoid using the $T_0$-quotient unless there are considerable benefits of doing so. Below is a list of properties to be explored in the sequel:

\begin{center}
\begin{tabular}{| c | c |} 
 \hline
 $\mathcal{P}$ & $r(\mathcal{P})$ \\ [0.5ex] 
 \hline
 $T_1$ & $R_0$ \\ [0.5ex]
 $T_2$ & $R_1$ \\ [0.5ex]
 Urysohn & Weakly Urysohn \\ [0.5ex]
 Discrete & Almost discrete\\ [0.5ex]
 $T_D$ & $R_d$ \\ [0.5ex]
 Sober & Quasi-sober \\ [0.5ex]
 Locally $T_2$ & Locally $R_1$ \\[1ex]
 \hline
\end{tabular}
\end{center}

\noindent
Next we will demonstrate how to transfer concepts between a space and its $T_0$ quotient. Specifically, we will ask: what does it mean for a point $x\in X$ and a subset $A\subseteq X$ when $q(x)\in X_0$ is an accumulation point of $q(A)$, given that $q:X\rightarrow X_0$ is the quotient mapping to the $T_0$-quotient? 

\begin{defn}\label{Definition 2.6}

Given a topological space $(X, \tau)$ and $q:X\rightarrow X_0$ the quotient mapping from $X$ to its $T_0$ quotient $X_0$, we call $x\in X$ a \textbf{strong accumulation point} or $\sim$-\textbf{accumulation point} if $q(x)\in q(A)'$. Given $A\subseteq X$, define:

$$
A^{\triangledown} = q^{-1}\big( q(A)' \big)
$$
Clearly $A^{\triangledown} = A'$ when $X$ is $T_0$. We can also immediately see that $A^{\triangledown}$ is $\sim$-saturated and $A^{\triangledown} = \big( [A]_{\sim} \big)^{\triangledown}$.
    
\end{defn}

\begin{prop}\label{Proposition 2.6}

Given a topological space $(X, \tau)$, $x\in X$ and $A\subseteq X$:

\begin{enumerate}[label = (\alph*)]

    \item $x\in A^{\triangledown}$ if and only if $(U\cap A)\backslash \{x\} \neq \emptyset$ for all $U\in\mathcal{N}_x$

    \item $\overline{A} = [A]_{\sim} \cup A^{\triangledown}$

    \item $\{x\}^{\triangledown} = \overline{\{x\}} \backslash [x]_{\sim}$
    
\end{enumerate}
    
\end{prop}

\begin{proof}
$\hspace{1cm}$\\

\begin{enumerate}[label = (\alph*)]

    \item According to \textbf{Theorem \ref{Theorem 2.4}(b)}, we have:

    $$
    \mathcal{U}\big( q(x) \big) = \big\{ q(U):U\in \mathcal{U}(x) \big\}
    $$
    Therefore:
    
    \begin{equation}\label{e4}
    q(x)\in q(A)' \hspace{0.6cm} \Longleftrightarrow \hspace{0.6cm} \forall\,U\in \mathcal{U}(x), \hspace{0.3cm} \big( q(U)\cap q(A) \big) \backslash \{q(x)\} \neq\emptyset
    \end{equation}
    The condition on the right in (\ref{e4}) says that for each $U\in\mathcal{U}$, there exists $y\in A$ such that $y\nsim x$ and $q(y)\in q(U)$. Since all open sets in $X$ are $\sim$-saturated, $q^{-1}\big( q(U) \big) = U$ for all $U\in\mathcal{U}(x)$, which implies that $q(y)\in q(U)$ if and only if $y\in U$. This proves the claim equivalence.

    \item According to \textbf{Theorem \ref{Theorem 2.4}(a)}, $q$ is closed and continuous. Also, every closed set is $\sim$-saturated. Therefore:

    $$
    \overline{A} = q^{-1}\big( q(\overline{A}) \big) = q^{-1}\big(\overline{q(A)} \big) = q^{-1}\big( q(A)\cup q(A)' \big) = q^{-1}\big( q(A) \big) \cup q^{-1}\big( q(A)' \big) = [A]_{\sim}\cup A^{\triangledown}
    $$

    \item Applying (a) to $A=\{x\}$, we deduce that $y\in \{x\}^{\triangledown}$ if and only if every $U\in\mathcal{U}$ contains some $y\preccurlyeq x$ but $y\sim x$. Hence according to \textbf{Proposition \ref{Proposition 1.2}}:

    $$
    \{x\}^{\triangledown} = \big\{y\in X:y\preccurlyeq x\,,\,y\nsim x \big\} = \overline{\{x\}} \backslash [x]_{\sim}
    $$
     
\end{enumerate}
    
\end{proof}

\noindent
In \textbf{Proposition \ref{Proposition 2.1}} and \textbf{Theorem \ref{Theorem 2.4}}, we showed the characterizations of a $T_0$ space by points in $X$ and by the quotient mapping $q$. We will end this section by introducing the \textbf{essential derived set} (see \cite{21, 22}), and later show that a space is $T_0$ precisely when the derived set of an arbitrary subset coincides with its essential derived set. 

\begin{defn}\label{Definition 2.7}

In a topological space $(X, \tau)$, given $A\subseteq X$, the \textbf{essential derived set} of $A$ is:

$$
D(A) = \overline{A} \backslash \big[A \backslash A' \big]_{\sim}
$$
    
\end{defn}

\begin{prop}\label{Proposition 2.9}

In the set-up of \textbf{Definition \ref{Definition 2.7}}:

\begin{enumerate}[label = (\alph*)]

    \item For any $A\subseteq X$, $A^{\triangledown} \subseteq D(A) \subseteq A'$.
    \item For any $A\subseteq X$, $D(A)$ is the largest subset of $A'$ that is downward-closed with respect to $\preccurlyeq$ (see \textbf{Definition \ref{Definition 0.1}}) and satisfies:

    \begin{equation}\label{e6}
    D(A) = \bigcup\big\{F\subseteq X: F\subseteq A',\, F=\overline{F} \big\}
    \end{equation}

    \item The following statements are equivalent:

    \begin{itemize}
        \item $X$ is $T_0$.
        \item For any $A\subseteq X$, $D(A) = A'$.
        \item For any $x\in X$, $\{x\}' = D(\{x\})$
        
    \end{itemize}

    \item $X$ is discrete if and only if $A'=\emptyset$ for any $A\subseteq X$

    \item A point $x\in X$ is $\preccurlyeq$-minimal if and only if $D(\{x\}) = \emptyset$.
    
\end{enumerate}
    
\end{prop}

\begin{proof}

First, observe that for any $x\in X$, $\{x\}\cap\{x\}' = \emptyset$. Then by \textbf{Proposition \ref{Proposition 2.6}}:

$$
D\big( \{x\} \big) = \overline{\{x\}} \backslash \big[ \{x\}\backslash \{x\}' \big]_{\sim} = \overline{\{x\}}\backslash [x]_{\sim} = \{x\}^{\triangledown}
$$

\begin{enumerate}[label = (\alph*)]

    \item By definition, $A\backslash A'$ is the set of isolated points in $A$. Then we immediately have the following inclusion:

    $$
    D(A)\subseteq \overline{A} \backslash \big( A\backslash A' \big) \subseteq A'
    $$
    Also, according to \textbf{Proposition \ref{Proposition 2.6}(a)} and \textbf{Lemma \ref{Lemma 1.6}}, a point from $A^{\triangledown}$ cannot be $\sim$-equivalent to any isolated points in $A$. Then we must have $A^{\triangledown} \subseteq D(A)$.

    \item We will first show that (\ref{e6}) holds. Fix a closed subset $F\subseteq X$ such that $F\subseteq A'$. Then for any $y\in F$, we have:

    $$
    [y]_{\sim} \subseteq \overline{\{y\}} \subseteq F \subseteq A'\subseteq [A']_{\sim} \subseteq D(A)
    $$
    which implies $F\subseteq D(A)$. Therefore the right-inclusion in (\ref{e6}) holds. On the other hand, fix $x\in D(A)$. Since $x\in A'$, for any $N\in\mathcal{N}_x$, $N\cap A$ is not a singleton. Therefore for any $y\in \overline{A}$ with $y\preccurlyeq x$, $y\in A'$. This implies $\overline{\{x\}} \subseteq A'$ according to \textbf{Proposition \ref{Proposition 1.2}}. We then prove that the left inclusion in (\ref{e6}) also holds, and that $D(A)$ is downward-closed with respect to $\preccurlyeq$. Given $B\subseteq A'$ a downward-closed subset, for any $b\in B$, we have $\overline{\{b\}} \subseteq B \subseteq A'$, which implies that $B\subseteq D(A)$ according to (\ref{e6}).

    \item When $X$ is $T_0$, for any $A\subseteq X$:

    $$
    [A]_{\sim} = \bigcup_{a\in A}[a]_{\sim} = \bigcup_{a\in A}\{a\} = A
    $$
    and hence $D(A) = A'$. The implication from the second statement to the third is immediate. Now assume that for any $x\in X$, $\{x\}' = D\big( \{x\} \big)$. By our previous observation, we have:

    $$
    \overline{\{x\}} \backslash [x]_{\sim} = \overline{\{x\}}\backslash \{x\}
    $$

    which implies, by \textbf{Lemma \ref{Lemma 1.6}}:

    $$
    \emptyset = \bigcap\mathcal{N}_x \cap \overline{\{x\}} \backslash [x]_{\sim}= \bigcap\mathcal{N}_x \cap \overline{\{x\}}\backslash \{x\} = [x]_{\sim}\backslash \{x\} \hspace{0.4cm} \Longrightarrow \hspace{0.4cm} [x]_{\sim} = \{x\}
    $$

    \item The result follows immediately by the fact that $X$ is discrete if and only if for any $A\subseteq X$, all points in $A$ are isolated.

    \item According to \textbf{Proposition \ref{Proposition 1.2}}, $x$ is $\preccurlyeq$-minimal if and only if $\overline{\{x\}} = [x]_{\sim}$. Then the result follows immediately by our observation at the beginning of this proof.
    
\end{enumerate}
    
\end{proof}

\section{\texorpdfstring{$T_1$}{T1} spaces}

\begin{defn}\label{Definition 3.1}

A topological space $(X, \tau)$ is $T_1$ if for any two different points $x, y\in X$, there exists $U\in\mathcal{U}(x)$, $V\in \mathcal{U}(y)$ such that $y\notin U$ and $x\notin V$. 

\end{defn}

\begin{prop}\label{Proposition 3.2}

In a topological space $(X, \tau)$, the following are equivalent:

\begin{enumerate}[label = (\alph*)]

    \item $X$ is $T_1$.
    \item for any two points $x, y\in X$, $x\preccurlyeq y$ implies $x=y$.
    \item all singleton are closed.
    \item for any $x\in X$, $\bigcap\mathcal{N}_x = \{x\}$.
    \item for any $x\in X$, $\{x\}' = \emptyset$.
    \item for any $A\subseteq X$, $A=\bigcap\mathcal{N}_A$ (see \textbf{Definition \ref{Definition 0.1}}).
    
\end{enumerate}
    
\end{prop}

\begin{proof}

The equivalence of $(b)$, $(c)$ and $(d)$ is an immediate consequence of \textbf{Proposition \ref{Proposition 1.2}}. The equivalence between $(c)$ and $(e)$ follows by the fact that $\{x\}' = \overline{\{x\}}\backslash \{x\}$. According to \textbf{Definition \ref{Definition 3.1}}, $X$ being $T_1$ is equivalent to that for any two different $x, y\in X$, both $\mathcal{N}_x \backslash \mathcal{N}_y$ and $\mathcal{N}_y \backslash \mathcal{N}_x$ are non-empty. Therefore the equivalence between $(a)$ and $(b)$ follows. We will complete the proof by showing that $(c)$ implies $(f)$. Assume that all singleton in $X$ are closed. Fix $A\subseteq X$,. If $x\notin A$, then $\{x\}^c\in \mathcal{N}_A$, or $\bigcap\mathcal{N}_A \subseteq \{x\}^c$. This implies $A^c\subseteq \left( \bigcap\mathcal{N}_A \right)^c$. Hence $A=\bigcap\mathcal{N}_A$.
    
\end{proof}

\begin{defn}

In a topological space $(X, \tau)$, given $A\subseteq X$, a point $x\in X$ is called a $\omega$-\textbf{accumulation point} of $A$ if each neighborhood of $x$ meets $A$ at infinitely many points. 
    
\end{defn}

\begin{prop}\label{Propositionn 3.3}

A topological space $(X, \tau)$ is $T_1$ if and only if for any $A\subseteq X$, any accumulation point of $A$ is an $\omega$-accumulation point.
    
\end{prop}

\begin{proof}

First we assume that $X$ is $T_1$ and suppose $x\in A$ is not an $\omega$-accumulation point. Then there exists $U\in\mathcal{U}(x)$ such that $U\cap A = \{x_i\}_{i\leq n}$ for some $n\in \mathbb{N}$. Without losing generality, assume that $x_i\neq x$ for all $1\leq i \leq n$. Put $V=U\backslash \{x_i\}_{i\leq n}$ and by \textbf{Proposition \ref{Proposition 3.2}(c)}, $V\in \mathcal{U}(x)$. Then $V\cap A = \{x\}$, so $x$ is not an accumulation point of $A$. Conversely, suppose $X$ is not $T_1$. By \textbf{Proposition \ref{Proposition 3.2}(e)}, there exists $x\in X$ with $\{x\}'\neq\emptyset$. Clearly any accumulation point of $\{x\}$ is not an $\omega$-accumulation point.
    
\end{proof}

\noindent
It is well-known that a topological space $(X, \tau)$ is Hausdorff precisely when the diagonal line $\Delta_X$ (see \textbf{Section 0.2}) is closed with respect to the product topology induced by $\tau$. The last result of this subsection will show that a $T_0$ or $T_1$ space can also be characterized by $\Delta_X$ but with respect to certain Alexandroff topology (see \textbf{Theorem \ref{Theorem 1.9}}). 

\begin{prop}\label{Proposition 3.5}

In a topological space $(X, \tau)$, let $\preccurlyeq, \sim$ denote the specialization preorder and the associated equivalence relation. Let $\tau\times\tau$ denote the product topology that is defined in $X\times X$ and induced by $\tau$. Then let $\preccurlyeq^2$, $\sim^2$ denote the specialization preorder and the associated equivalence relation in $(X\times X, \tau\times\tau)$. We will use the notation defined in \textbf{Remark \ref{Remark 1.10}}, and use $\tau(\sim^2)$, $\tau(\preccurlyeq^2)$ to denote the Alexandroff topology induced by $\sim^2$ and $\preccurlyeq^2$ respectively. Then we have:

\begin{enumerate}[label = (\alph*)]

    \item $X$ is $T_0$ if and only if $\Delta_X$ is closed with respect to the Alexandroff topology $\tau\big( \sim^2 \big)$.
    \item $X$ is $T_1$ if and only if $\Delta_X$ is closed with respect to the Alexandroff topology $\tau\big( \preccurlyeq^2 \big)$.
    
\end{enumerate}
    
\end{prop}

\begin{proof}

By definition of $\tau\times\tau$, for any $(x, y)\in X\times X$, the following family of open sets (with respect to $\tau \times \tau$) forms a neighborhood basis of $(x, y)$:

$$
\mathcal{U}(x) \times \mathcal{U}(y) = \big\{U\times V: U\in \mathcal{U}(x) \,,\, V\in \mathcal{U}(y) \big\}
$$
Therefore, given $(x_1, y_1)$, $(x_2, y_2)\in  X\times X$:

\begin{equation}\label{e8}
\begin{aligned}
& (x_1, y_1) \preccurlyeq^2 (x_2, y_2) \hspace{0.4cm} \Longleftrightarrow \hspace{0.4cm} x_1\preccurlyeq y_1 \,\text{  and  }\, x_2\preccurlyeq y_2 \\
& (x_1, y_1) \sim^2 (x_2, y_2) \hspace{0.4cm} \Longleftrightarrow \hspace{0.4cm} x_1\sim y_1 \,\text{  and  }\, x_2\sim y_2
\end{aligned}
\end{equation}
By \textbf{Theorem \ref{Theorem 1.9}}, $\Delta_X$ is closed with respect to $\tau(\sim^2)$, $\tau(\preccurlyeq^2)$ if and only if it is $\sim^2$-, $\preccurlyeq^2$-saturated respectively.

\begin{enumerate}[label = (\alph*)]

    \item Suppose $X$ is $T_0$. If given $(x_1, y_1)\in X\times X$, $(x_1, y_1)\sim^2 (x, x)$ for some $x\in X$, then by (\ref{e8}), we have $x_1\sim y_1$, which implies $(x_1, y_1)\in \Delta_X$ according to \textbf{Proposition \ref{Propopsition 2.2}}. Therefore $\Delta_X\in \Sigma\big( \sim^2 \big)$. Next assume $X$ is not $T_0$. Then there exists two different $x, y\in X$ such that $x\sim y$. Again by (\ref{e8}), $(x, y)\sim^2(x, x)$ but $(x, y)\in \Delta_X$, which implies $\Delta_X$ is not $\sim^2$-saturated.

    \item Suppose $X$ is $T_1$. If given $(x_1, y_1)\in X \times X$, $(x_1, y_1) \preccurlyeq^2 (x, x)$ for some $x\in X$, we must have $x_1=y_1$ by (\ref{e8}) and \textbf{Proposition \ref{Proposition 3.2}}. If $X$ is not $T_1$, by \textbf{Proposition \ref{Proposition 3.2}}, there exists two distinct points $x, y\in X$ such that $x\preccurlyeq y$. By (\ref{e8}), $(x, y)\preccurlyeq^2 (y, y)$ but $(x, y)\notin \Delta_X$. Therefore $\Delta_X$ is not $\preccurlyeq^2$-saturated.
    
\end{enumerate}
    
\end{proof}

\section{\texorpdfstring{$R_0$}{R0} and \texorpdfstring{$R_1$}{R1} spaces}

\begin{defn}

A topological space $(X, \tau)$ is $R_0$ if the specialization preorder $\preccurlyeq$ is symmetric, namely $x\preccurlyeq y$ if and only if $y\preccurlyeq x$ for all $x, y\in X$. In this case $\preccurlyeq$ coincides with the $\sim$ given in \textbf{Definition \ref{Definition 1.5}}. 
    
\end{defn}

\begin{prop}\label{Proposition 4.2}

In a topological space $(X, \tau)$, the following statements are equivalent:

\begin{enumerate}[label = (\alph*)]

    \item $X$ is $R_0$.
    \item for any $x\in X$ and $U\in\mathcal{U}(x)$, we have $\overline{\{x\}} \subseteq U$.
    \item for any $x, y\in X$, if $\overline{\{x\}} \neq \overline{\{y\}}$, then $\overline{\{x\}} \cap \overline{\{y\}} = \emptyset$.
    \item for any $x, y\in X$ with $x\nsim y$, there exists $U\in \mathcal{U}(x)$ such that $y\notin U$.
    \item for every $A\subseteq X$, $[A]_{\sim} = \bigcap\mathcal{N}_A$.
    \item for every $x\in X$, $D(\{x\}) = \emptyset$ (see \textbf{Definition \ref{Definition 2.7}}).
    
\end{enumerate}
    
\end{prop}

\begin{proof}

We will first show the equivalence between the first four statements, and then the equivalence between $(a)$ and the last two. By \textbf{Proposition \ref{Proposition 1.2}}, when $X$ is $R_0$, for all $x\in X$ and $U\in \mathcal{U}(x)$:

\begin{equation}\label{e7}
[x]_{\sim} = \overline{\{x\}} = \bigcap\mathcal{N}_x \subseteq U
\end{equation}
Assuming $(b)$ is true, then we immediately have the first two equalities in (\ref{e7}) holds, which implies $[x]_{\sim} = [x]_{\preccurlyeq} = [x]_{\succcurlyeq}$, or $X$ is $R_0$. If $(c)$ is true, then whenever $x\nsim y$, we have $[x]_{\sim} \neq [y]_{\sim}$, or $\overline{\{x\}} \neq \overline{\{y\}}$. Then $(d)$ follows by our assumption that $\overline{\{x\}}$ is disjoint from $\overline{\{y\}}$. Note that when $(d)$ is true, we have that if $x\nsim y$, then $x\npreccurlyeq y$. In other words, $x\preccurlyeq y$ implies $x\sim y$, and hence $y\preccurlyeq x$. This implies $X$ is $R_0$.\\

\noindent
Next we will show $(a)\implies (e)$. We assume $X$ is $R_0$. Fix $x\in X$ and suppose $x\preccurlyeq a$ for some $a\in A$. Then we have $x\in [a]_{\sim}$ by assumption. According to $(b)$, for each $N\in \mathcal{N}_A$ with $N\in\tau$, we have $\overline{\{x\}}=[x]_{\sim} \subseteq N$ by (\ref{e7}), which implies $[x]_{\sim} \subseteq \bigcap\mathcal{N}_A$. Hence $[A]_{\sim} \subseteq \bigcap\mathcal{N}_A$. Suppose $y\notin [A]_{\sim}$. Then $[y]_{\sim}$, or $\overline{\{y\}}$ (by (\ref{e7}), is disjoint from $[A]_{\sim}$. Fix $z\in \overline{\{y\}}$. By $(c)$, for each $a\in A$ there exists $V_a\in \mathcal{U}(a)$ such that $z\notin V_a$. Define:

$$
V = \bigcup_{a\in A}V_a
$$
We then have $V\in \mathcal{N}_A$ and $z\notin V$. Hence $z\notin \bigcap\mathcal{N}_A$. When $(e)$ is true, according to \textbf{Proposition \ref{Proposition 1.2}}, for any $x\in X$, $[x]_{\sim} = [x]_{\preccurlyeq}$, which immediately implies that $X$ is $R_0$. The equivalence between $(a)$ and $(f)$ is also immediate according to \textbf{Proposition \ref{Proposition 2.9}}.
    
\end{proof}

\begin{prop}\label{Proposition 4.3}

A topological space $(X, \tau)$ is $T_1$ if and only if it is $T_0$ and $R_0$.
    
\end{prop}

\begin{proof}

According to \textbf{Proposition \ref{Proposition 3.2}}, $X$ is $T_1$ if and only if $\preccurlyeq$ reduces to the identity relation, which is precisely when $\preccurlyeq$ is both symmetric and anti-symmetric. Together with \textbf{Proposition \ref{Proposition 2.1}} the definition of a $R_0$ space, we can conclude that $X$ being $T_1$ is equivalent to being $T_0$ and $R_0$.
    
\end{proof}

\begin{defn}\label{Definition 4.4}

A topological space $X$ is $R_1$ or \textbf{pre-regular}, if for any two different $x, y\in X$ with $x\nsim y$, there exists $U\in \mathcal{U}(x)$, $V\in \mathcal{U}(y)$ such that $U\cap V=\emptyset$. As an immediate consequence, if we view $\sim$ as a subset in $X\times X$, then $\sim$ is closed with respect to product topology in $X\times X$ precisely when $X$ is $R_1$.
    
\end{defn}

\begin{prop}\label{Proposition 4.5}

A topological space $(X, \tau)$ is Hausdorff if and only if it is $T_0$ and $R_1$

\end{prop}

\begin{proof}

When $X$ is Hausdorff, the preorder $\preccurlyeq$ and $\sim$ coincide with the equality. In this case, $X$ is $T_0$ and $R_1$ by definition. Conversely, suppose that $X$ is $T_0$ and $R_1$. Given two different $x, y\in X$, we must have $x\nsim y$. as $\preccurlyeq$ is anti-symmetric in $T_0$ space. Then by definition of a $R_1$ space, there exists $U\in\mathcal{U}(x)$ and $V\in \mathcal{U}(y)$ such that $U\cap V = \emptyset$.
    
\end{proof}

\begin{prop}\label{Proposition 4.6}

A topological space $(X, \tau)$ is $R_0$ ($R_1$ resp.) if and only if its $T_0$ quotient is $T_1$ (Hausdorff resp.)
    
\end{prop}

\begin{proof}

The conclusion follows by that for any $x\in X$, $\overline{[x]_{\sim}} = \overline{\{ x \}}$, and that $q^{-1}\big( q(x) \big) = [x]_{\sim}$ where $q:X\rightarrow X_0$ is the quotient mapping to the $T_0$ quotient.
    
\end{proof}

\noindent
Recall that in the proof of \textbf{Proposition \ref{Proposition 3.5}}, we show that $\preccurlyeq^2$ ($\sim^2$ resp.) is the coordinate-wise extension of $\preccurlyeq$ ($\sim$ resp.). By similar reasoning, we can show that given a topological space $(X, \tau)$ and an index set $I$, the specialization preorder $\preccurlyeq^I$ that is defined on $\prod_{i\in I}X$ with respect to the product topology on $\prod_{i\in I}X$, is also the coordinate-wise extension of $\preccurlyeq$. We can then immediately concludes that both $R_0$, $R_1$ are productive topological properties. A unified result will be provided later as a productive topological property is a special case of an initial property (see \textbf{Definition \ref{Definition 0.3}}). Next in \textbf{Proposition \ref{Proposition 4.7}} we will compare $R_0$, $R_1$ with other separation axioms, and prove a refined version of it after introducing new definitions and their properties.

\begin{prop}\label{Proposition 4.7}

Given a topological space $(X, \tau)$:

\begin{enumerate}[label = (\alph*)]

    \item if $X$ is $R_1$, $X$ will be $R_0$.
    \item if $X$ is regular, $X$ is $R_1$.
    \item if $X$ is $R_0$ and normal, then $X$ is completely regular.
    
\end{enumerate}
    
\end{prop}

\begin{proof}

\begin{enumerate}[label = (\alph*)]

    \item If $X$ is $R_1$, given $x, y\in X$, we have that $x\npreccurlyeq y$ implies that $y\npreccurlyeq x$. Then we can conclude $\preccurlyeq$ is symmetric.

    \item Assume that $X$ is regular. Fix $x, y\in X$ such that $x\nsim y$. Without losing generality, we can assume that $x\npreccurlyeq y$, so that $x\notin \overline{\{y\}}$ by \textbf{Proposition \ref{Proposition 1.2}}. Regularity provides two disjoint open sets $U, V$ such that $x\in U$ and $\overline{\{y\}}\subseteq V$. Hence $X$ is $R_1$.

    \item Assume that $X$ is normal and $R_0$. Fix $x\in X$ and a closed set $F\subseteq X$ with $x\notin F$. Then for any $y\in F$, $x\notin \overline{\{y\}}$. By $R_0$ axioms $y\notin \overline{\{x\}}$ for all $y\in F$, or $F$ is disjoint from $\overline{\{x\}}$. We can then apply \textbf{Urysohn's Lemma} to find $f\in C\big( X, [0, 1] \big)$ such that $f(x)=1$ and $f\Big|_F=0$. This implies $X$ is completely regular.
    
\end{enumerate}
    
\end{proof}

\noindent
We will end this section by discovering disconnectness in the context of $R_0$ and $R_1$ spaces. The main reference here is \cite{10}. Recall that a topological space $(X, \tau)$ can always be decomposed into a disjoint union of connected components. For each $x\in X$, let $C(x)$ denote the component that contains $x$. 

\begin{defn}

In a topological space $(X, \tau)$, define binary relations $\sim_C$ and $\cong_C$ as follows: for any $x, y\in X$, $x\sim_C y$ if there exists a connected set that contains $x$ and $y$; $x\cong_C y$ if every clopen neighborhood of $x$ contains $y$.
    
\end{defn}

\begin{rem}\label{Remark 4.8}

Obviously both $\sim_C$ and $\cong_C$ are equivalence relations. For each $x\in X$, $C(x) = [x]_{\sim_C}$ and for any two $x,y\in X$, $x\sim_C y$ implies $x\cong_C y$. As a result, when viewed as a partition of $X\times X$, $\cong_C$ is coarser than $\sim_C$. For each $x\in X$, we use $QC(x)$ to denote $[x]_{\cong_C}$. When we need to specify $X$ or $\tau$, we will use $C_X(x)$ ($QC_X(x)$ resp.) or $C_{\tau}(x)$ ($QC_{\tau}(x)$ resp.) to denote $[x]_{\sim_C}$ ($[x]_{\cong_C}$ resp.). We then have:

\begin{equation}\label{e9}
\overline{\{x\}} \subseteq C(x) \subseteq QC(x)
\end{equation}
It is also obvious that $QC(x)$ is the intersection of all clopen neighborhoods of $x$.

\end{rem}

\begin{defn}\label{Definition 4.10}

A topological space $(X, \tau)$ is said to be \textbf{totally disconnected} if all components are singletons, and \textbf{totally separated} if all quasi-components are singletons. Generalizing above, $X$ is \textbf{weakly totally disconnected} if $C(x) = \overline{\{x\}}$ for all $x\in X$, and \textbf{weakly totally separated} if $QC(x) = \overline{\{x\}}$ for all $x\in X$. By (\ref{e9}), we can see a (weakly) totally separated space is (weakly) totally disconnected.

\end{defn}

\begin{exmp}\label{Example 4.11}

Here we will introduce an example of a weakly totally separated space. A topological space $(X, \tau)$ is \textbf{zero-dimensional} if $\tau$ has a base that consists of clopen sets. We can then immediately see that $X$ is zero-dimensional precisely when each point has a neighborhood base that consists of clopen sets. Hence a zero-dimensional space is weakly totally separated. Also, in this case when $X$ is regular and hence $R_1$ (by \textbf{Proposition \ref{Proposition 4.7}}). If we let $\mathcal{B}$ denote the base that consists of clopen sets, then the family of characteristic functions $\big( \chi_B:B\in \mathcal{B} \big)$ separates different points and disjoint closed sets. Therefore a zero-dimensional space is completely regular. 
    
\end{exmp}

\begin{prop}\label{Proposition 4.11}

Given $(X, \tau)$ a topological space, if $X$ is weakly totally disconnected, then $X$ is $R_0$; if $X$ is weakly totally separated, then $X$ is $R_1$.

\end{prop}

\begin{proof}

Fix $(X, \tau)$ a topological space. If $X$ is weakly totally disconnected, then for any $x, y\in X$ with $x\preccurlyeq y$, we have $x\in \overline{\{y\}} = C(y)$, which implies $C(y) \subseteq C(x) = \overline{\{x\}}$ by definition of $C(x)$. We then have $y\preccurlyeq x$ by \textbf{Proposition \ref{Proposition 1.2}} and hence $X$ is $R_0$. If $X$ is weakly totally separated, fix $x, y\in X$ with $x\nsim y$. Then $x\notin \overline{\{y\}} = QC(y)$, so we can find a clopen subset $U$ such that $x\in U$ and $y\in U^c$. This implies $X$ is $R_1$.
    
\end{proof}

\noindent
Recall that in \textbf{Theorem \ref{Theorem 2.4}} we prove that the $T_0$-quotient of a topological space is always $T_0$ and the quotient mapping has the universal property. Next we will show the analogous result for totally disconnectedness.

\begin{prop}\label{Proposition 4.13}

Given a topological space $(X, \tau_X)$, the quotient space $X\slash\sim_C$ is totally disconnected. Moreover, given a continuous mapping $f:(X, \tau_X) \rightarrow (Y, \tau_Y)$ where $Y$ is totally disconnected, there exists a unique continuous mapping $\tilde{f}:X\slash\sim_C \rightarrow C$ such that $f = \tilde{f} \circ p_X$ where $p_X:X\rightarrow X\slash\sim_C$ is the canonical quotient mapping. 
    
\end{prop}

\begin{proof}

We first establish the universal property of $p=p_X$. Let $(Y, \tau_Y)$ be a totally disconnected topological space and $f:X\rightarrow Y$ be continuous. For any $x, y\in X$ with $x\sim_C y$, we have $f\big( C(x) \big)$ is a connected subset of $Y$ that contains both $f(x)$ and $f(y)$. Since $Y$ is totally disconnected, we must have $f\big( C(x) \big) = \{f(x)\}$, and hence $f(x)=f(y)$. From this, the following mapping is well-defined and unique with respect to $f$:

$$
\tilde{f}:X\slash\sim_C \rightarrow Y, \hspace{0.3cm} p(x)\mapsto f(x)
$$
and satisfies $f = \tilde{f}\circ p$. Since the quotient topology in $X\slash\sim_C$ is the strong topology induced by $p$, the continuity of $f$ implies the continuity of $\tilde{f}$.\\

\noindent
To show that $X$ is totally disconnected, we will need to borrow arguments from {\cite[$\S$ I.$11.5$, \textbf{Proposition 9}]{38}}. Let $A$ be a closed subset of $X\slash \sim_C$ having at least two points. Then $p^{-1}(A)$ is closed in $X$ and contains at least two distinct components. Then assume $p^{-1}(A) = E\cup F$ where $E, F$ are disjoint sets closed in $X$. We shall show that $E$ is $\sim_C$-saturated. Pick $x\in p^{-1}\big( p(E) \big)$. Then there exists $y\in E$ with $p(x) = p(y)$. As $E, F$ separate $p^{-1}(A)$, $E, F$ also separate $C(y)$. Since $C(y)$ is connected, without losing generality, assume $C(y) \subseteq E$ and hence $x\in E$. This shows that $p^{-1}\big( p(E) \big) = [E]_{\sim_C} \subseteq E$, or $[E]_{\sim_C} = E$. In particular, $p(E)$ is closed in $X\slash \sim_C$. By the same token, $[F]_{\sim_C} = F$. We then have $A=p(E)\cup p(F)$ is the disjoint union of two subsets closed in $X_{\sim_C}$, and hence $A$ is disconnected. We can now conclude that the only connected subsets in $X\slash\sim_C$ are singletons, and hence $X\slash\sim_C$ is totally disconnected.
    
\end{proof}

\begin{rem}

Similar to the reasoning in \textbf{Remark \ref{Remark 2.6}}, \textbf{Proposition \ref{Proposition 4.13}} shows that the universal property of $p_X$ signifies that the category of totally disconnected spaces form a reflective subcategory of \textbf{Top}.
    
\end{rem}

\noindent
We are now ready to prove a refined version of \textbf{Proposition \ref{Proposition 4.7}}. Recall that a topological space $(X, \tau)$ is \textbf{Urysohn} if for any two different points $x, y\in X$, there exists $U\in \mathcal{U}(x)$, $V\in \mathcal{U}(y)$ such that $\overline{U}$ and $\overline{V}$ are disjoint. We will need a weaker version of the Urysohn property.

\begin{defn}

A topological space $(X, \tau)$ is \textbf{weakly Urysohn} if for any $x, y\in X$, whenever $x\nsim y$, there exists $U\in \mathcal{U}(x)$, $y\in \mathcal{U}(y)$ such that $\overline{U}\cap \overline{V} = \emptyset$. 
    
\end{defn}

\begin{prop}\label{Proposition 4.15}

Given a topological space $(X, \tau)$:

\begin{enumerate}[label = (\alph*)]

    \item if $X$ is weakly Urysohn (Urysohn resp.), then $X$ is $R_1$ (Hausdorff resp.).
    \item if $X$ is (weakly) totally separated, then $X$ is (weakly) Urysohn.
    \item if $X$ is regular, then $X$ is weakly Urysohn.
    \item $X$ is weakly Urysohn precisely when $X_0$ the $T_0$-quotient of $X$ is Urysohn.
    
\end{enumerate}
    
\end{prop}

\begin{proof}

Part $(a)$ follows immediately by definitions. For $(b)$, when $X$ is totally separated, fix $x, y\in X$ with $x\nsim y$. Then we must have $x\neq y$, or $QC(x)$ is disjoint from $QC(y)$. There exists $U\in\mathcal{N}_x$ a clopen neighborhood of $x$ such that $y\notin U$. Then let $V\in\mathcal{U}(y)$ be disjoint from $U$. If $V=\{y\}$, by (\ref{e9}) we have that $\{y\}$ is a clopen neighborhood of $y$, and is disjoint from $U$. Otherwise, $QC(y)\subsetneq V$, which implies there exists a clopen neighborhood $V'\in y$ such that $V'\subseteq V$ and $V'\cap U = \emptyset$. We can now conclude that $X$ is Urysohn.\\

\noindent
It remains to show part $(c)$ and $(d)$. To show $(c)$, we will use an equivalent formulation of regularity. When $X$ is regular, for any $x\in X$ and $O\in \mathcal{U}(x)$, there exists $V\in\mathcal{U}(x)$ such that $\overline{V} \subseteq O$. Now fix $x, y\in X$ such that $x\nsim y$. Without losing generality, assume $x\npreccurlyeq y$, or $x\notin \overline{\{y\}}$ by \textbf{Proposition \ref{Proposition 1.2}}. By regularity, there exists two disjoint open sets $U, V$ such that $x\in U$ and $\overline{\{y\}} \subseteq V$. Then we can find $U'\in \mathcal{U}(x)$, $V'\in\mathcal{U}(y)$ such that $\overline{U'} \subseteq U$ and $\overline{V'} \subseteq V$, and $\overline{U'}\cap \overline{V'} = \emptyset$.\\

\noindent
To show $(d)$, recall that a closed set in $X$ is $\preccurlyeq$-saturated. Namely for any closed subset $C\subseteq X$, if $x\preccurlyeq c$ for some $c\in C$, we then have $x\in C$. First assume $X$ is weakly Urysohn. For any $x, y\in X$ with $q(x)\neq q(y)$, we have $x\nsim y$, and hence there exists $U\in\mathcal{U}(x)$, $V\in \mathcal{U}(y)$ such that $\overline{U}\cap \overline{V} = \emptyset$. If, for some $z\in X$, $q(z) \in q(\overline{U}) \cap q(\overline{V})$, we will have:

$$
z\in q^{-1}\big( q(\overline{U}) \big) \cap q^{-1}\big( q(\overline{V}) \big) = [\overline{U}]_{\sim}\cap [\overline{V}]_{\sim} \subseteq [\overline{U}]_{\preccurlyeq} \cap [\overline{V}]_{\preccurlyeq} = U\cap V = \emptyset
$$
which is absurd. Then by \textbf{Theorem \ref{Theorem 2.4}(a)}, we have $q(\overline{U}) = \overline{q(U)}$ and $q(\overline{V}) = \overline{q(V)}$. This implies that $X_0$ is Urysohn. Conversely, given $x\nsim y$, if $X_0$ is weakly Urysohn, by \textbf{Theorem \ref{Theorem 2.4}(b)}, there exists $U\in \mathcal{U}(x)$, $V\in \mathcal{U}(y)$ such that $q(x)\in q(U)$, $q(y)\in q(V)$, and $\overline{q(U)}$ is disjoint from $\overline{q(V)}$. By \textbf{Theorem \ref{Theorem 2.4}(a)}, we have:

$$
x\in q^{-1}\big( \overline{q(U)} \big) = q^{-1}\big( q(\overline{U}) \big) = [\overline{U}]_{\sim} \subseteq [\overline{U}]_{\preccurlyeq} = \overline{U} 
$$
and similarly $y\in \overline{V}$. Therefore $\overline{U}$ is disjoint from $\overline{V}$, and hence $X$ is weakly Urysohn.

\end{proof}

\noindent
In the rest of this section, we will discover the relation between being weakly totally disconnected (separated resp.) and totally disconnected (separated resp.). We can expect conditions similar to \textbf{Proposition \ref{Proposition 4.15}(d)} where $T_0$ plays an important role in building connections between these two pairs of properties. Then we will introduce the \textbf{extremely disconnected} space. This property is stronger than being weakly totally separated and also related to a zero-dimensional space (see \textbf{Example \ref{Example 4.11}}). Since we will revisit zero-dimensional spaces again, we will include descriptions of when a zero-dimensional space has discrete $T_0$-quotient. We will close this section by proving all the initial properties (see \textbf{Section 0.3}) we have introduced.

\begin{prop}\label{Proposition 4.16}

In a topological space $(X, \tau)$, let $q:X\rightarrow X_0$ denote its $T_0$-quotient mapping.

\begin{enumerate}[label = (\alph*)]

    \item for any $A\subseteq X$ and $B\subseteq X_0$, $A$ is connected if and only if $q(A)$ is connected, and $B$ is connected if and only if $q^{-1}(B)$ is connected. 
    \item for any $x\in X$ (see \textbf{Remark \ref{Remark 4.8}}):

    $$
    q\big( C_X(x) \big) = C_{X_0}\big( q(x) \big) \hspace{1cm} q\big( QC_X(x) \big) = QC_{X_0}\big( q(x) \big)
    $$
    As a result, the mapping $A\rightarrow q(A)$ where $A$ is any connected set in $X$, and the mapping $B\mapsto q^{-1}(B)$ where $B$ is any connected subset in $X_0$, induce homeomorphisms between $X\slash\sim_C$ and $X_0\slash\sim_C$, and also between $X\slash\cong_C$ and $X_0\slash\cong_C$.

    \item $X$ is weakly totally disconnected if and only if $X_0$ is totally disconnected.
    \item $X$ is weakly totally separated if and only if $X_0$ is totally separated.
    
\end{enumerate}
    
\end{prop}

\begin{proof}

\begin{enumerate}[label = (\alph*)]

    \item Since $q:X\rightarrow X_0$ is continuous, when $A$ is connected, $q(A)$ is connected and when $q^{-1}(B)$ is connected, $B$ is connected. Now suppose that $q^{-1}(B)$ is disconnected. Then there exists $O_1, O_2\in \tau$ such that $B_1 = O_1\cap q^{-1}(B)$ is disjoint from $B_2 = O_2\cap q^{-1}(B)$ and $q^{-1}(B)$ is the union of $B_1$ and $B_2$. By applying \textbf{Theorem \ref{Theorem 2.4}(b)} to the subspace $q^{-1}(B)$ equipped with the relative topology in $X$, we have that $q(B_1)\cap q(B_2) = \emptyset$ and $B = q(B_1)\cup q(B_2)$. According to \textbf{Theorem \ref{Theorem 2.4}(a)}, both $q(B_1)$ and $q(B_2)$ are open in $B$ and hence $B$ is disconnected. In the same token, when $A$ is disconnected, we can show that $q(A)$ is disconnected.

    \item According to \textbf{Theorem \ref{Theorem 2.4}(a)} and the fact that $q$ is continuous, for any $x\in X$, a clopen neighborhood of $q(x)$ is necessarily in the form of $q(O)$ for some clopen neighborhood of $x$. Then the second equality follows immediately. By part (a), a connected component in $X_0$ is necessarily in the form of $q(C)$ for some connected component $C\subseteq X$. Hence the first equality follows.

    \item If $X_0$ is totally disconnected, since for any $x\in X$, $\{q(x)\}$ is connected, we have $\overline{\{q(x)\}} = \{q(x)\}$. Hence by \textbf{Theorem \ref{Theorem 2.4}(a)}:
    
    $$
    [x]_{\sim} = q^{-1}\big( \{q(x)\} \big) = q^{-1}\big( \overline{\{q(x)\}} \big) = \overline{[x]_{\sim}} = \overline{\{x\}}
    $$
    is a connected subset that contains $x$. Hence $\overline{\{x\}} \subseteq C_X(x)$. Meanwhile, by part (b) we have $C_X(x)\subseteq q^{-1}\big( q(x) \big) = [x]_{\sim} = \overline{\{x\}}$, which implies $\overline{\{x\}} = C_X(x)$. To prove the converse, first observe that for any connected subset $C\subseteq X_0$, by \textbf{Theorem \ref{Theorem 2.4}}, $q^{-1}(C)$ is also connected. Hence for any $x\in X$, we have:

    $$
    C_{X_0}\big( q(x) \big) = q\Big( q^{-1}\big( C_{X_0}(q(x)) \big) \Big) \subseteq q\big( C_X(x) \big) = q\big( \overline{\{x\}} \big) \subseteq \overline{\{q(x)\}} \subseteq C_{X_0}\big( q(x) \big)
    $$
    Then together with \textbf{Proposition \ref{Proposition 4.11}} and \textbf{Proposition \ref{Proposition 4.6}}, we have $C_{X_0}\big( q(x) \big) = \{q(x)\}$, and hence $X_0$ is totally disconnected.

    \item By \textbf{Theorem \ref{Theorem 2.4}}, $q$ is clopen. Hence when $X$ is weakly totally separated, $X_0$ is totally separated. Conversely, if $X_0$ is totally separated, observe that:

    $$
    \overline{\{x\}} \subseteq QC(x) \subseteq \bigcap_{\substack{C\in \mathcal{U}\big( q(x) \big) \\ C\,\text{ clopen }}} q^{-1}(C) = q^{-1}\left( \bigcap_{\substack{C\in \mathcal{U}\big( q(x) \big) \\ C\,\text{ clopen }}} C\right) = q^{-1}\Big( QC\big( q(x) \big) \Big) = q^{-1}\big( \{q(x)\} \big) \subseteq \overline{\{x\}}
    $$
     
\end{enumerate}
    
\end{proof}

\begin{prop}\label{Proposition 4.17}

A topological space $(X,\tau)$ is \textbf{extremally disconnected} if the closure of each open set is also open. Fix $(X, \tau)$ a topological space. Then the following are equivalent:

\begin{enumerate}[label = (\alph*)]

    \item $X$ is extremally disconnected if and only if disjoint open subsets have disjoint closures.
    \item if $X$ is extremally disconnected and $R_1$, then $X$ is weakly totally disconnected. Thus, if $X$ is extremally disconnected and Hausdorff, $X$ is totally separated.
    \item if $X$ is extremally disconnected and regular, then $X$ is zero-dimensional (see \textbf{Example \ref{Example 4.11}}).
    \item if $X$ is zero-dimensional, then $X$ is extremally disconnected precisely when the Boolean algebra $(\mathcal{B}, \subseteq)$ is complete where $\mathcal{B}$ is the family of all clopen subsets.
    
\end{enumerate}

\end{prop}

\begin{proof}

\begin{enumerate}[label = (\alph*)]

    \item First we assume that disjoint open subsets have disjoint closures. Fix $O\in\tau$. Since $O^c$ and $\overline{O}^c$ are disjoint open subsets, we have $\overline{O}$ and $\overline{\overline{O}^c}$ are also disjoint. Then for any $x\in \overline{O}$, there exists $V\in \mathcal{U}(x)$, such that $V\cap \overline{O}^c = \emptyset$, or $V\subseteq \overline{O}$. Hence $\overline{O}$ is open. Conversely, suppose there exists two disjoint open subsets $U, V\in \tau$ such that $x\in\overline{U} \cap \overline{V}$. Then for any $O\in \mathcal{U}(x)$, $O\cap V\neq\emptyset$ and hence $O\cap \overline{U}^c\neq \emptyset$. This implies $\overline{U}$ is not open and hence $X$ is not extremally disconnected.

    \item Fix $x\in X$. By (\ref{e9}) we have $\overline{\{x\}} \subseteq QC(x)$. If there exists $y\in QC(x) \backslash \overline{\{x\}}$, by $R_1$ axiom, there exist $U\in\mathcal{U}(x)$, $V\in \mathcal{U}(y)$ such that $U\cap V = \emptyset$. Since $X$ is extremally disconnected, $\overline{U}$ and $\overline{V}$ are both open and disjoint by part (a). This implies $x\ncong_C y$ as $\overline{V}$ is a clopen neighborhood of $y$, and leads to a contradiction.

    \item Assuming $X$ is extremally disconnected and regular, we will show that $\tau$ has a base that consists of clopen sets. For any $O\in\tau$, by regularity there exists $V\in\tau$ such that $\overline{V}\subseteq O$. Since $X$ is extremally disconnected, $\overline{V}\in\tau$ and hence clopen.

    \item Assume $X$ is zero-dimensional. If $X$ is extremally disconnected, for any $\mathcal{S}\subseteq \mathcal{B}$, by \textbf{Example \ref{Example 4.11}}, $\bigvee\mathcal{S} = \bigcup\mathcal{S} \in \mathcal{B}$ and hence $\mathcal{S}$ always has a supremum. Conversely, if $\mathcal{B}$ is complete, fix $U\in\tau$ and define:

    $$
    \mathcal{S}_U = \big\{ A\in \mathcal{B}:A\subseteq U \big\} \hspace{1cm} B=\bigvee\mathcal{S}_U
    $$
    Hence $\overline{U}\subseteq B$. If $B\backslash \overline{U} \neq \emptyset$, we can find a non-empty $C\in \mathcal{B}$ such that $C\subseteq B\backslash \overline{U}$. In this case, $B\backslash C\in\mathcal{B}$ and $U\subseteq B\backslash C$, contradicting the definition of $B$. Hence $\overline{U}$ is open.
    
\end{enumerate}
    
\end{proof}

\begin{prop}\label{Proposition 4.18}

Given a topological space $(X, \tau)$, the following are equivalent:

\begin{enumerate}[label = (\alph*)]

    \item $X_0$ the $T_0$-quotient of $X$ is discrete.
    \item there exists an equivalence relation $\sim_X$ defined on $X$ such that $\tau = \Sigma(\sim_X)$ the family of all $\sim_X$-saturated sets (see \textbf{Definition \ref{Definition 0.1}}).
    \item every open subset is closed.
    \item for every $x\in X$, $\overline{\{x\}}$ is open.
    \item $X$ is Alexandroff and zero-dimensional.
    \item $X$ is $R_0$ and Alexandroff.
    \item for any $A\subseteq X$, $A^{\triangledown} = \emptyset$ (see \textbf{Definition \ref{Definition 2.6}}).
    
\end{enumerate}

\noindent
We call $X$ \textbf{almost discrete} if $X$ satisfies any one of the conditions above.
    
\end{prop}

\begin{proof}

When $X_0$ is discrete, we claim that $\tau = \Sigma(\sim)$. Since $X_0$ is discrete, $[x]_{\sim} \in\tau$ for all $x\in X$. Also, for any $O\in\tau$ and $x\in X$, if $x\sim y$ for some $y\in O$, by \textbf{Lemma \ref{Lemma 1.6}} we have $\mathcal{N}_x = \mathcal{N}_y$, which implies there exists $N\in \mathcal{U}(x)$ such that $N\subseteq O$, and hence $x\in O$. Therefore $O = [O]_{\sim} = \bigcup_{x\in O}[x]_{\sim}$.\\

\noindent
If $\tau = \Sigma(\sim_X)$ for some equivalence relation $\sim_X$ defined on $X$, then for any $O\in \tau$, there exists $O'\subseteq X$ such that $[O']_{\sim_X} = O$. Then $[O']_{\sim_X}^c = \big[ X\backslash O' \big]_{\sim_X}\in \tau$, which implies $O$ is closed. If every open subset is closed, then every closed subset is open, and hence $\overline{\{x\}}$ is open for each $x\in X$.\\

\noindent
Since for each $x\in X$, $\overline{\{x\}}$ is open, we have $\overline{\{x\}}\subseteq \bigcap\mathcal{N}_x$, or each $x\in X$ is $\preccurlyeq$-minimal (see \textbf{Definition \ref{Definition 1.3}}). Then by \textbf{Lemma \ref{Lemma 1.7}}, for each $N\in \mathcal{N}_x$, $\overline{\{x\}} \subseteq N$. This implies that given $(O_i)_{i\in I}\subseteq \tau$, if $\bigcap_{i\in I}O_i$ is non-empty, then for any $x\in \bigcap_{i\in I}O_i$, $\overline{\{x\}} \subseteq \bigcap_{i\in I}O_i$, and hence $\bigcap_{i\in I}O_i$ is open. Also, for each $x\in X$, $\big\{ \overline{\{x\}} \big\}$ is a singleton neighborhood base that consists of clopen sets, and hence $X$ is zero-dimensional.\\

\noindent
When $X$ is Alexandroff and zero-dimensional, by \textbf{Theorem \ref{Theorem 1.9}}, for each $x\in X$, there exists a clopen set $B_x\in \mathcal{N}_x$ such that $B_x\subseteq N$ for all $N\in \mathcal{N}_x$. Fix $x, y\in X$ with $x\preccurlyeq y$. We then have $B_x\in \mathcal{N}_y$, or $B_y\subseteq B_x$. If $x\in B_x\backslash B_y$, there exists $O\in\mathcal{U}(x)$ such that $O\subseteq B_x\backslash B_y$, contradicting that $x\preccurlyeq y$. Hence we have $x\in B_y$, which implies $B_x\subseteq B_y$. We can now conclude that for all $N\in \mathcal{N}_y$, $x\in B_y\subseteq N$, and hence $y\preccurlyeq x$.\\

\noindent
When $X$ is Alexandroff, by our previous reasoning, we have for each $x\in X$ and $y\in B_x$, $x\preccurlyeq y$. If $X$ is also $R_0$, we have that for each $x\in X$, according to \textbf{Lemma \ref{Lemma 1.7}}:

$$
\overline{\{x\}} \subseteq B_x \subseteq [x]_{\preccurlyeq} = [x]_{\sim} = \overline{\{x\}}
$$
which implies $X_0$ is discrete as $B_x$ is clopen. When $X_0$ is discrete, for all $A\subseteq X$, $q(A)'=\emptyset$ and hence $A^{\triangledown} = \emptyset$. Conversely, if $A^{\triangledown}=\emptyset$ for all $A\subseteq X$, by \textbf{Proposition \ref{Proposition 2.6}} and the fact that $q:X\rightarrow X_0$ is surjective, we will have:

$$
X\backslash \overline{A} = X\backslash [A]_{\sim} = [A^c]_{\sim} = q^{-1}\big( q(A^c) \big) = q^{-1}\big( q(A)^c \big)
$$
for all $A\subseteq X$, which implies $q(A)$ is open for all $A\subseteq X$. Therefore $X_0$ is discrete.
    
\end{proof}

\begin{prop}

All of the following properties are initial properties (see \textbf{Section 0.3}): $R_0$, $R_1$, weakly Urysohn, regular, completely regular, weakly totally disconnected, weakly totally separated, zero-dimensional.
    
\end{prop}

\begin{proof}

Fix a set $X$ and a family of mappings $\mathcal{F}_A = \big\{ f_a: X\rightarrow (Y_a, \tau_a) \big\}_{a\in A}$. Let $\tau\big( \mathcal{F}_A \big)$ denote the weak topology induced by $\mathcal{F}_A$. Let $\preccurlyeq$, $\sim$ denote the specialization preorder and the induced equivalence relation in $\big(X, \tau(\mathcal{F}_A) \big)$. For each $a\in A$, let $\preccurlyeq_a$, $\sim_a$ denote the specialization preorder and the induced equivalence relation in $(Y_a,\tau_a)$.\\

\noindent
Fix $x, y\in X$. By definition of $\tau(\mathcal{F}_A)$,  $x\preccurlyeq y$ ($x\sim y$ resp.) if and only if for every $a\in A$, $f_a(x) \preccurlyeq_a f_a(y)$ ($f_a(x) \sim_a f_a(y)$). When each $(Y_a, \tau_a)$ is $R_0$, $\big( X, \tau(\mathcal{F}_A) \big)$ is $R_0$. If $x\nsim y$, there exists $a\in A$ such that $f_a(x) \nsim f_a(y)$. In this case, if $(Y_a, \tau_a)$ is $R_1$, then clearly there are two disjoint open sets that separate $x$ and $y$; if $(Y_a, \tau_a)$ is weakly Urysohn, suppose $U\in \mathcal{U}(f(x))$, $V\in \mathcal{U}(f(y))$ are two open sets such that $\overline{U} \cap \overline{V}=\emptyset$. We then have $\overline{f^{-1}(U)} \subseteq f^{-1}(\overline{U})$ and $\overline{f^{-1}(V)} \subseteq f^{-1}(\overline{V})$ are disjoint. Hence if each $(Y_a, \tau_a)$ is $R_1$ or weakly Urysohn, so is $\big( X, \tau(\mathcal{F}_A) \big)$.\\

\noindent
Next fix $x\in X$, $U\in \mathcal{U}(x)$. Then there exists $(a_i)_{i\leq n}\subseteq A$ for some $n\in \mathbb{N}$, and $U_i\in \mathcal{U}\big( f_{a_i}(x) \big)$ such that:

$$
\bigcap_{1\leq i\leq n} f_{a_i}^{-1}(U_i) \subseteq U
$$
If each $(Y_a, \tau_a)$ is regular, for each $1\leq i \leq n$, let $V_i\in \mathcal{U}\big( f_{a_i}(x) \big)$ such that $\overline{V_i} \subseteq U_i$. This implies:

$$
\overline{\bigcap_{1\leq i \leq n} f_{a_i}^{-1}(V_i)} \subseteq \bigcap_{1\leq i \leq n} f_{a_i}^{-1}\big( \overline{V_i} \big) \subseteq \bigcap_{1\leq i \leq n} f_{a_i}^{-1}(U_i) \subseteq U
$$
and hence $X$ is regular.\\

\noindent
Fix $x\in X$ and consider a point $y\in C_X(x)$ the connected component that contains $x$. If $(Y_a, \tau_a)$ is weakly totally disconnected, then for each $a\in A$:

$$
f_a(y) \in f_a\big( C_X(x) \big) \subseteq C_{Y_a}\big( f_a(x) \big) = \overline{\{f_a(x)\}} \hspace{0.5cm} \Longleftarrow \hspace{0.5cm} f_a(y)\preccurlyeq_a f_a(x)
$$
which implies $y\preccurlyeq a$, or $y\in \overline{\{x\}}$. We then have $C_X(x) = \overline{\{x\}}$. Consider $y\in QC_X(x)$. For each $a\in A$, a clopen neighborhood $V$ of $f_a(x)$ will bring $f_a^{-1}(V)$ a clopen neighborhood of $x$. Hence when each $(Y_a, \tau_a)$ is weakly totally separated, we have:

$$
QC_X(x) \subseteq \bigcap_{a\in A} f_a^{-1}\Big( QC_{Y_a}\big( f_a(x) \big) \Big) = \bigcap_{a\in A} f_a^{-1}\big( \overline{\{f_a(x)\}} \big)
$$
In this case, if $y\in QC_X(x)$, then $f_a(y) \preccurlyeq_a f_a(x)$ for all $a\in A$, or $y\preccurlyeq x$. This shows that $QC_X(x) \subseteq \overline{\{x\}}$, and hence $\big(X, \tau(\mathcal{F}_A) \big)$ is weakly totally separated. If each $(Y_a, \tau_a)$ is zero-dimensional, note that for any finite subset $F\subseteq A$, $\bigcap_{a\in F} f_a^{-1}(U_a)$ is clopen if each $U_a$ is clopen. This implies that $\tau(\mathcal{F}_A)$ has a base that consists of clopen sets.\\

\noindent
To show that complete regularity is an initial property, first suppose each $(Y_a, \tau_a)$ is completely regular. For each $a\in A$, there exists a family of continuous functions $\mathcal{F}_a = \big\{ h_{a, i}: Y_a\rightarrow [0, 1] \big\}_{i\in I_a}$ (indexed by $\tau_a$-closed sets) such that $\big\{ Y_a\backslash h_{a, i}^{-1}(0, 1]: i\in I_a \big\}$ forms a base of $\tau_a$. As a result, the following family:

$$
\begin{aligned}
& \hspace{0.4cm} \left\{ \bigcap_{a\in F} f_a^{-1}\big( Y_a\backslash h_{a, i}^{-1}(0, 1] \big): h_{a, i}\in \mathcal{F}_a,\, a\in F\subseteq A,\, \vert\, F\,\vert < \infty \right\} \\
& = \left\{ \bigcap_{a\in F} X\backslash \big( h_{a, i} \circ f_a \big)^{-1}\big( (0, 1] \big): i\in I_a,\, a\in F\subseteq A,\, \vert\, F\,\vert < \infty \right\}
\end{aligned}
$$
forms a base of $\tau(\mathcal{F}_A)$. Hence for each $x\in X$ and a closed subset $E\subseteq X$ with $x\notin E$, we can find a finite subset $F\subseteq A$ and $h_{a, i}\in \mathcal{F}_a$ for each $a\in F$ such that:

$$
\bigcap_{a\in F} X\backslash \big( h_{a, i} \circ f_a \big)^{-1}\big( (0, 1] \big) \subseteq E^c \hspace{0.5cm} \Longleftarrow \hspace{0.5cm} E\subseteq \bigcup_{a\in F} \big( h_{a, i} \circ f_a \big)^{-1}\big( \{0\} \big)
$$
In this case, for each $a\in F$, we can assume $h_{a, i}\circ f_a(x) = 1$. Define:

$$
h_F:X\rightarrow [0, 1], \hspace{0.3cm} x\mapsto \prod_{a\in F} h_{a, i}\circ f_a(x)
$$
Then $h_F\in C(X, [0, 1])$ and satisfies $h_F(x)=1$ and $h_F\big|_E = 0$.
    
\end{proof}

\begin{rem}

One can refer to \cite{9} and \cite{10} for more interesting properties of extremally spaces. We point out that non-$R_1$ extremally disconnected spaces can be connected and non-$R_0$. Simple examples include the \textbf{Sierpinski space}, the space $Z=\{0, 1\}$ equipped with the topology $\tau = \{\emptyset, \{0\}, Z\}$, which is also $T_0$, and any infinite set equipped with the cofinite topology, which is also $T_1$. The cofinite topology also shows that an extremally disconnected $T_1$ space need not even be totally disconnected, let along being totally separated. It is known that if $X$ is a extremally disconnected space, then so is any open subspace and dense subspace of $X$. Moreover, extremally disconnectedness is preserved by continuous mapping. In the references given above, one can see, however, a closed subspace of a extremally disconnected may not be extremally disconnected, and there exists a extremally disconnected space $X$ such that $X\times X$ is not extremally disconnected.
    
\end{rem}

\section{\texorpdfstring{$R_1$}{R1} and compactness}

In this section we will discover how results concerning compactness in a Hausdorff space extend to $R_1$ spaces. Although a compact subset in an $R_1$ space need not be closed (e.g. singleton), its closure remains compact according to our next result.

\begin{prop}\label{Proposition 5.1}

In an $R_1$ topological space $(X, \tau)$, if $K\subseteq X$ is compact, then:

\begin{enumerate}[label = (\alph*)]

    \item $\overline{K} = [K]_{\sim} = \bigcup_{x\in K} \overline{\{x\}}$.
    \item if $U\in \tau$ contains $K$, then $U$ contains $\overline{K}$.
    \item for any $A\subseteq X$, if $K\subseteq A \subseteq \overline{K}$, then $A$ is compact. In particular, $\overline{K}$ is compact.
    
\end{enumerate}
    
\end{prop}

\begin{proof}

\begin{enumerate}[label = (\alph*)]

    \item By \textbf{Proposition \ref{Proposition 4.7}(a)}, $X$ is $R_0$. Then by $R_0$ axioms, for each $x\in X$, $\overline{\{x\}} = [x]_{\sim}$. Hence it remains to show that $\overline{K} = \bigcup_{x\in K} \overline{\{x\}}$. The right inclusion is obvious. For the left inclusion, first we fix $y\in X\backslash \bigcup_{x\in K} \overline{\{x\}}$. By $R_1$ axiom, for each $x\in K$ there exists $U_x\in \mathcal{U}(x)$ such that $y\notin \overline{U_x}$. Since $K$ is compact, there exists a finite subset $F\subseteq K$ such that:

    $$
    K \subseteq \bigcup_{x\in F}U_x \hspace{0.4cm} \Longrightarrow \hspace{0.4cm} \overline{K} \subseteq \bigcup_{x\in F} \overline{O_x}
    $$
    Hence $y\notin \overline{K}$.

    \item Fix $U\in\tau$ such that $K\subseteq U$. Again by $R_0$ axioms and \textbf{Lemma \ref{Lemma 1.7}}, for each $x\in U$, $\overline{\{x\}} \subseteq U$. Then the conclusion follows by part (a).

    \item Fix $A\subseteq X$ such that $K\subseteq A\subseteq \overline{K}$. Given $\mathcal{U}$ a open cover of $A$, let $\mathcal{U}' \subseteq \mathcal{U}$ be a finite subcover that covers $K$. Then by part (b) we have $A\subseteq \overline{K} \subseteq \bigcup\mathcal{U}'$, so $\mathcal{U}'$ is also a subcover of $A$.
    
\end{enumerate}
    
\end{proof}

\noindent
The next step concerns the separation of disjoint compact sets. Recall that a compact set in a Hausdorff space is closed, and hence two disjoint compact sets in a Hausdorff space have disjoint open neighborhoods. However, this may fail in an $R_1$ space. For instance, consider two different point $x, y$ in an $R_1$ space such that $x\sim y$. In this case $\{x\}$ and $\{y\}$ are disjoint compact sets that cannot be separated by two disjoint open sets. Hence we need a stronger form of ``disjointness" if we can extend this useful result.

\begin{defn}

In a topological space $(X, \tau)$, given $A, B\subseteq X$, we say $A$ and $B$ are \textbf{strongly disjoint} if for any $a\in A$ and $b\in B$, $a\nsim b$. We say $A$ and $B$ are \textbf{strictly disjoint} if for any $a\in A$ and $b\in B$, $a\npreccurlyeq b$ and $b\npreccurlyeq a$.
    
\end{defn}

\begin{prop}\label{Proposition 5.3}

In an $R_1$ topological space $(X, \tau)$, strongly disjoint compact sets have disjoint open neighborhoods.
    
\end{prop}

\begin{proof}

Suppose $K_1, K_2\subseteq X$ are strongly disjoint compact sets. Fix $y\in K_2$ and for every $x\in K_2$, by $R_1$ axioms there exists $U(x, y)\in \mathcal{U}(x)$ such that $y\notin \overline{U(x, y)}$. Now $\mathcal{F}(y) = \big( U(x, y):x\in K_1 \big)$ is an open cover of $K_1$. Then there exists $F(y)$ a finite subcover of $\mathcal{F}(y)$. Define $U_y = \bigcup F(y)$, $V_y = \overline{U_y} ^c$. Now $K_1\subseteq U_y$ and $V_y\in \mathcal{U}(y)$. Then $\mathcal{F} = \big( V_y^c:y\in K_2 \big)$ is an open cover of $K_2$. Then there exists $F\subseteq K_2$ a finite subset such that $\big( V_y^c: y\in F \big)$ is a subcover of $\mathcal{F}$. Define:

$$
V = \bigcup_{y\in F}V_y^c \hspace{1cm} U = \bigcap_{y\in F}U_y
$$
Therefore $U$ and $V$ are disjoint open neighborhoods of $K_1$ and $K_2$ respectively.
    
\end{proof}

\begin{cor}\label{Corollary 5.4}

A compact $R_1$ topological space is normal.
    
\end{cor}

\begin{proof}

Suppose $(X, \tau)$ is compact $R_1$. Let $A, B\subseteq X$ be two disjoint closed subsets. Then both $A$ and $B$ are compact. Moreover, by $R_1$ axioms and \textbf{Proposition \ref{Proposition 5.1}}, $A$ and $B$ are strictly disjoint. Then the conclusion follows by \textbf{Proposition \ref{Proposition 5.3}}.
    
\end{proof}

\begin{rem}

Recall that a topological space is \textbf{paracompact} if every open cover has a locally finite refinement. Namely, a topological space $(X, \tau)$ is paracompact if for any $(U_i)_{i\in I}\subseteq \tau$, there exists another collection $(O_j)_{j\in J} \subseteq \tau$ such that for each $j\in J$ there exists $i(j)\in I$ such that $O_i\subseteq U_{i(j)}$, and for each $x\in X$ there exists $U_x\in \mathcal{U}(x)$ such that $\big\{ j\in J: U_x\cap O_j\neq\emptyset \big\}$ is finite. Clearly a compact space is paracompact and a closed subset in a paracompact space is also paracompact. Also, one can show that a paracompact $R_1$ space is normal. Since regularity implies $R_1$ (see \textbf{Proposition \ref{Proposition 4.7}}), a paracompact regular space is also normal, which provides an (less elementary) alternate proof of \textbf{Corollary \ref{Corollary 5.4}}.
    
\end{rem}

\noindent
Next we will examine local compactness. In the weak sense, a space $(X, \tau)$ is \textbf{locally compact} if every point has a compact neighborhood. In the strong sense, $(X, \tau)$ is \textbf{locally compact} if every point has a neighborhood base that consists of closed compact sets. The strong version is used usually when non-Hausdorff spaces are involved. Clearly these two versions of definitions agree on a Hausdorff space. We will show that it turns out that this remains true for $R_1$ spaces. Beside, we will prove a more general statement.

\begin{defn}\label{Definition 5.6}

Given a topological space $(X, \tau)$ and $x\in X$, if $x$ has a compact neighborhood, then $X$ is said to be \textbf{locally compact at} $x$. We call $X$ \textbf{locally compact in the weak sense} if $X$ is locally compact at each point, and \textbf{locally compact in the strong sense} if each point has a neighborhood base that consists of closed compact neighborhoods.
    
\end{defn}

\begin{prop}\label{Proposition 5.7}

Given a topological space $(X, \tau)$, if $X$ is $R_1$ and locally compact at a point $x\in X$, then $x$ has a neighborhood base that consists of closed compact sets.
    
\end{prop}

\begin{proof}

Fix $K$ a compact neighborhood of $x$. By \textbf{Proposition \ref{Proposition 5.1}}, $\overline{K}$ is also compact, so we may assume $K$ is closed. If for all $U\in \mathcal{U}(x)$, $K\subseteq U$ then we are done. Otherwise, let $V\in\mathcal{U}(x)$ such that $K\backslash V \neq \emptyset$. Fix $y\in K\backslash U$. Since $y\nsim x$ and $K\backslash V$ is compact, by \textbf{Proposition \ref{Proposition 5.3}}, there exists $O\in \mathcal{U}(y)$ such that $\overline{O}$ is disjoint from $K\backslash V$. Hence $\overline{O}\cap K$ is a closed compact neighborhood such that $\overline{O}\cap K \subseteq V$. This implies that for all $U\in \mathcal{U}(x)$, if $K\backslash U\neq \emptyset$, then there exists a closed compact neighborhood contained in $U$.
    
\end{proof}

\begin{prop}\label{Proposition 5.8}

Let $(X, \tau)$ be a locally compact $R_1$ space. Suppose $K$ a compact subset is contained in an open subset $U$. Then there exists a closed compact neighborhood $V$ of $\overline{K}$ such that $V\subseteq U$. Moreover, there exists $C(X, [0, 1])$ such that $f\Big|_K=1$ and $f\Big|_{V^c}=0$
    
\end{prop}

\begin{proof}

Fix $x\in K$. By \textbf{Proposition \ref{Proposition 5.7}}, there exists $V_x\in \mathcal{U}(x)$ such that $\overline{V_x}$ is compact and contained in $U$. Then $(V_x)_{x\in K}$ forms an open cover of $K$. Let $\big( V_{x_i}: 1\leq i \leq n \big)$ be a finite subcover of $(V_x)_{x\in K}$. Define:

$$
V = \bigcup_{1\leq i \leq n}V_{x_i}
$$
Then by \textbf{Proposition \ref{Proposition 5.1}(b)}, $\overline{K} \subseteq V \subseteq \overline{V} \subseteq U$. Also $\overline{V}$ is compact. Clearly $R_1$ is a hereditary property, and hence $\overline{V}$ when viewed as a subspace is also $R_1$. Then by \textbf{Corollary \ref{Corollary 5.4}}, the subspace $\overline{V}$ is normal. Now we have $\overline{K}$ and $\overline{V} \backslash V$ two disjoint closed subsets of $\overline{V}$. By \textbf{Urysohn's Lemma}, there exists $g\in C(\overline{V}, [0, 1])$ such that $g\Big|_{\overline{K}} = 1$ and $g\Big|_{\overline{V} \backslash V} = 0$. Define $f:X\rightarrow [0, 1]$ by letting $f=g$ on $\overline{V}$ and $f=0$ elsewhere. Then $f$ satisfies $f(x)=1$ for all $x\in K$ and $f(y)=0$ for all $y\in X\backslash \overline{V}$.
     
\end{proof}

\begin{prop}

Let $X$ be an $R_1$ space and $E\subseteq X$ be a dense and $\sim$-saturated subset such that $E$, when viewed as a subspace equipped with the relative topology, is locally compact. Then $E$ is open in $X$.
    
\end{prop}

\begin{proof}

Fix $x\in E$. Note that $E$, when viewed as a subspace equipped with $\tau\Big|_E$ the relative topology, is $R_1$. By \textbf{Proposition \ref{Proposition 5.1}(c)}, there exists $K\subseteq E$ a neighborhood of $x$ such that $K$ is closed and compact with respect to $\tau\Big|_E$. Let $U\in\tau$ such that $\operatorname{int}_E(K)$, the interior of $K$ with respect to $\tau\Big|_E$, is $U\cap E$. Since $E$ is a dense subset of $X$ and $U\in \tau$, $\overline{U\cap E} = \overline{U}$. Let $\operatorname{cl}_E(K)$ denote the closure of $K$ with respect to $\tau\Big|_E$. Then:

$$
\overline{U}\cap E = \overline{U\cap E}\cap E = \operatorname{cl}_E(U\cap E) \subseteq \operatorname{cl}_E(K) = K
$$
Hence $\overline{U}\cap E$ is closed with respect to $\tau\Big|_E$ and hence is compact with respect to $\tau\Big|_E$. Also, $\overline{U}\cap E$, when viewed as a subset of $X$, is an intersection of two $\sim$-saturated subsets, and hence is $\sim$-saturated. By \textbf{Proposition \ref{Proposition 5.1}(a)}, $\overline{U}\cap E$ is closed with respect to $\tau$. Therefore:

$$
\overline{U}\cap E \subseteq \overline{U} =\overline{U\cap E} \subseteq \overline{\overline{U}\cap E} = \overline{U}\cap E
$$
which implies $\overline{U}\cap E = \overline{U}$, or $\overline{U}\subseteq E$. Thus $x\in U\subseteq E$, and hence $E$ is open.
    
\end{proof}

\noindent
In general, from \textbf{Definition \ref{Definition 4.10}} and \textbf{Remark \ref{Remark 4.8}}, we show that a weakly totally separated space is also weakly totally disconnected. We will next show that in a locally compact $R_1$ space, the converse is also true.

\begin{prop}\label{Proposition 5.10}

In a compact $R_1$ topological space $(X, \tau)$, for each $x\in X$, $C(x) = QC(x)$ (see \textbf{Definition \ref{Definition 4.10}}).
    
\end{prop}

\begin{proof}

For each $x\in X$, let $\mathcal{B}_x$ denote the family of clopen neighborhoods of $x$. By \textbf{Remark \ref{Remark 4.8}}, $QC(x) = \bigcap\mathcal{B}_x$ and $C(x)\subseteq QC(x)$. It remains to show that $QC(x) \subseteq C(x)$, or $QC(x)$ is connected. Fix $x\in X$ and let $Q=QC(x)$. Suppose $Q=A\cup B$ where $A, B$ satisfies $Q\cap A$ is disjoint from $Q\cap B$ and both $A, B$ are closed with respect to $\tau\Big|_Q$ the relative topology. Assume $x\in A$. Since $Q$ is closed with respect to $\tau$, both $A$ and $B$ are closed with respect to $\tau$. By \textbf{Corollary \ref{Corollary 5.4}}, $X$ is normal. Then there exists disjoint open subsets $U, V\subseteq X$ such that $A\subseteq U$ and $B\subseteq V$. Then $A=Q\cap U$ and:

$$
(U\cup V)^c \subseteq Q^c = \bigcup\big\{H^c: H\in\mathcal{B}_x \big\}
$$
Since $(U\cup V)^c$ is compact, suppose $(H_i^c)_{i\leq n}$ is a finite subcover of $\big( H^c: H\in \mathcal{B}_x \big)$. Let $H=\bigcap_{1\leq i \leq n}H_i$. Then $H\in \mathcal{B}_x$ and $H\subseteq U\cup V$. However, since $U\cap V=\emptyset$, $H\cap U = H\backslash V$, which shows that $H\cap U$ is clopen. Since we assume $x\in A$, $H\cap U\in\mathcal{B}_x$. Therefore:

$$
Q = (H\cap U)\cap Q \subseteq U\cap Q = A\subseteq Q
$$
which show that $B=\emptyset$. Hence $Q$ is connected.
    
\end{proof}

\begin{prop}

Let $(X, \tau)$ be a locally compact $R_1$ space. If $X$ is weakly totally disconnected, then $\tau$ has a base that consists of clopen compact sets. In this case $X$ is weakly totally separated.
    
\end{prop}

\begin{proof}

It suffices to show that each point has a neighborhood base that consists of compact clopen sets. Fix $x\in X$ and $U\in \mathcal{U}(x)$. By \textbf{Proposition \ref{Proposition 5.7}}, there exists $V$ a closed compact neighborhood of $x$ such that $V\subseteq U$. Since $X$ is $R_1$, by \textbf{Lemma \ref{Lemma 1.7}}, $\overline{\{x\}} \subseteq V^{\circ}$ (the interior of $V$ with respect to $\tau$). If $V=V^{\circ}$, then we are done. Otherwise, define $\partial V = V\backslash V^{\circ}$. Consider the subspace $\left(V, \tau\Big|_V \right)$ equipped with the relative topology. Since $\left( V, \tau\Big|_V \right)$ is $R_1$ and compact, according to \textbf{Proposition \ref{Proposition 5.10}}, we have $C_K(x) = QC_K(x)$ (these notations are defined in \textbf{Remark \ref{Remark 4.8}}). Since $V$ is closed and $X$ is weakly totally disconnected, we have $\overline{\{x\}} = QC(x) = QC_K(x)$. Fix $y\in \partial V$, and we have $y\notin \overline{\{x\}}$. By \textbf{Proposition \ref{Proposition 5.10}}, we have $C_K(x) = QC_K(x)$. Hence there exists $H_y\subseteq V$ such that $H_y$ is clopen with respect to $\tau\Big|_V$ and $x\in V\backslash H_y$. As $y\in \partial V$ is arbitrarily fixed, we then have an open cover $(H_y:y\in \partial V)$ of the subset $\partial V$ that is compact with respect to $\tau$. We can then find a finite subset $(y_i)_{i\leq n} \subseteq \partial V$ such that:

$$
\partial V \subseteq \bigcup_{1\leq i \leq n}H_{y_i}
$$
Define:

$$
W = V\backslash \bigcup_{1\leq i \leq n}H_{y_i}
$$
We then have $x\in W$, $W$ is clopen with respect to $\tau\Big|_V$ and $W\subseteq V^{\circ}$. We can now conclude that $W$ is clopen and compact with respect to $\tau$
    
\end{proof}

\noindent
We have been showing several topological properties that can either be preserved or upgraded by the $T_0$-quotient, and will definitely do the same to the local compactness property. We will close this section by showing that the locally compactness in the weak sense (see \textbf{Definition \ref{Definition 5.6}}) can be preserved by $T_0$-quotient, and revisiting several previously mentioned properties in the \textbf{one-point compacification} of a topological space.

\begin{prop}

Let $(X, \tau)$ be a topological space and $q:X\rightarrow X_0$ be its $T_0$-quotient. Then for each $x\in X$, $X$ is locally compact at $x$ if and only if $X_0$ is locally compact at $q(x)$.
    
\end{prop}

\begin{proof}

Fix $x\in X$. Suppose $K\subseteq X$ is a compact neighborhood of $x$. Since $q:X\rightarrow X_0$, $q(K)$ is compact and is a neighborhood of $q(x)$ by \textbf{Theorem \ref{Theorem 2.4}}. Conversely, suppose $C\subseteq X_0$ is a compact neighborhood of $q(x)$. Again by \textbf{Theorem \ref{Theorem 2.4}}, $q^{-1}(C)$ is also a neighborhood of $x$. Suppose $(O_i)_{i\in I}$ is an open cover of $q^{-1}(C)$. Since $q$ is open, suppose $\big( q(O_j): 1\leq j \leq n \big)$ is a finite subcover of $\big( q(O_i): i\in I \big)$. We then have:

$$
q^{-1}(C) \subseteq \bigcup_{1\leq j \leq n} q^{-1}\big(q(O_j) \big) = \bigcup_{1\leq j \leq n}[O_j]_{\sim}
$$
Since open sets in $X$ are $\sim$-saturated, $(O_j)_{1\leq j \leq n}$ is a finite subcover of $(O_i)_{i\in I}$.
    
\end{proof}

\begin{defn}[{\cite[p.150]{2}}]\label{Definition 5.13}

Given $(X, \tau)$ a topological space, let $\mathcal{C}$ be the family of closed compact subsets of $X$. Fix an element $\infty$ that is not in $X$ and define $X^* = X\cup\{\infty\}$. Define $\mathcal{F} = \big\{ X^*\backslash C:C\in \mathcal{C} \big\}$, and $\tau^* = \tau\cup\mathcal{F}$. One can check that $\tau^*$ is a topology defined on $X^*$. The new topological space $(X^*, \tau^*)$ is called the \textbf{one-point compacification of} $(X, \tau)$. According to \cite{2}, $(X^*, \tau^*)$ satisfies the following properties:

\begin{enumerate}[label = (\alph*)]
    \item $(X^*, \tau^*)$ is compact.
    \item $\{\infty\}$ is closed with respect to $\tau^*$.
    \item $X$ is a dense subspace of $X^*$ if and only if $X$ is not compact.
    \item $X^*$ is Hausdorff if and only if $X$ is Hausdorff and locally compact (as we point out, when $X$ is Hausdorff, the weak version of locally compactness agrees with the strong version of locally compactness).
\end{enumerate}
    
\end{defn}

\begin{prop}

In the set-up of \textbf{Definition \ref{Definition 5.13}}:

\begin{enumerate}[label = (\alph*)]

    \item $X^*$ is $T_0$ ($R_0$ resp.) if and only if $X$ is $T_0$ ($R_0$ resp.).
    \item $X^*$ is $R_1$ if and only if $X$ is locally compact in the weak sense and $R_1$.
    
\end{enumerate}
    
\end{prop}

\begin{proof}

By definition of $\tau^*$, for any $x\in X$, the family of neighborhoods of $x$ with respect to $\tau^*$ coincides with the family of neighborhoods of $x$ with respect to $\tau$. Therefore, there will be no confusion if we use $\mathcal{N}_x$ to denote the family of neighborhoods of $x$ and $\mathcal{U}(x)$ to denote the family of open neighborhoods of $x$. The specialization order of $(X^*, \tau^*)$, when restricted to $X$, coincides with the specialization order of $(X, \tau)$. In particular, observe that:

$$
\mathcal{U}(\infty) = \big\{X^*\backslash C: C\in \mathcal{C} \big\}
$$

\begin{enumerate}[label = (\alph*)]

\item As we point out, when $X^*$ is $T_0$ ($R_0$ resp.), $X$ is $T_0$ ($R_0$ resp.). Now fix $x\in X$. When $X$ is $T_0$, suppose $x\preccurlyeq\infty$ and $\infty\preccurlyeq x$. By results listed in \textbf{Definition \ref{Definition 5.13}}, we have $x\in \overline{\{\infty\}} = \{\infty\}$, or $x=\infty$. According to \textbf{Proposition \ref{Propopsition 2.2}}, $X^*$ is $T_0$. When $X$ is $R_0$, first suppose that $x\preccurlyeq\infty$. In this case we just showed that $x=\infty$. Conversely, if $\infty\preccurlyeq x$, we have $x\in X\backslash C$ for all $C\in\mathcal{C}$. However, when $X$ is $R_0$, according to \textbf{Lemma \ref{Lemma 1.7}}, $\overline{\{x\}}\in \mathcal{C}$, which leads to a contradiction. We can now conclude that $X^*$ is $R_0$. 

\item When $X^*$ is $R_1$, clearly $X$ is $R_1$ according to our observation at the beginning. Moreover, for each $x\in X$, there exists $C\in\mathcal{C}$, $U\in\mathcal{U}$ such that $U\cap (X\backslash C)=\emptyset$, or $U\subseteq C$. This implies that $C\in\mathcal{N}_x$, and hence $X$ is locally compact in the weak sense. By the same reasoning, we can conclude the converse is also true.

\end{enumerate}
    
\end{proof}

\section{Locally closed sets: the Skula topology}

In this section, we will introduce the definition of \textbf{locally closed} sets, and, for a given topology $\tau$, the \textbf{Skula topology} of $\tau$, denoted by $\operatorname{Sk}(\tau)$, which is a topology finer than $\tau$ and induced by the family of locally closed set in $\tau$. This section mainly serves as a transition to the next section where we will introduce a topological property that is strictly weaker than $T_1$ and strictly stronger than $T_0$.

\begin{defn}\label{Definition 6.1}

In a topological space $(X, \tau)$, given $A\subseteq X$, we call $A$ \textbf{locally closed} if $A$ is an intersection of an open set and a closed set. For each $x\in X$, we say $A$ is \textbf{locally closed at} $x$ if there exists $V\in \mathcal{N}_x$ such that $V\cap A$ is (relatively) closed in $V$.
    
\end{defn}

\begin{rem}\label{Remark 6.2}

With the definition given, we have several immediate observations. First a locally closed set is the difference of two open sets, or equivalently, two closed sets. Clearly all open and closed sets are locally closed. In the set-up of \textbf{Definition \ref{Definition 6.1}}, for a point $x\in X$, if $V\in\mathcal{N}_x$ satisfies that $V\cap A$ is closed with respect to $\tau\Big|_V$, we can find $F\subseteq X$ closed with respect to $\tau$ such that $V\cap A = V\cap F$. In this case, if we let $U = V^{\circ}$ be $V$'s interior with respect to $\tau$, we also have $U\cap A = U\cap F$. Hence in the second part of \textbf{Definition \ref{Definition 6.1}} we may replace $\mathcal{N}_x$ by $\mathcal{U}(x)$. If $x\notin \overline{A}$, then $X\backslash \overline{A}\in\mathcal{U}(x)$ and $\big(X\backslash \overline{A} \big) \cap A = \emptyset$, also a closed set with respect to $\tau$. This will imply $\overline{A}$ is locally closed at $x$ for all $x\notin \overline{A}$. 

\end{rem}

\begin{prop}\label{Proposition 6.3}

In a topological space $(X, \tau)$, given $A\subseteq X$, the following are equivalent:

\begin{enumerate}[label = (\alph*)]

    \item $A$ is locally closed.
    \item $A$ is locally closed at $x$ for each $x\in X$.
    \item $A$ is open with respect to $\tau\Big|_{\overline{A}}$.
    
\end{enumerate}
    
\end{prop}

\begin{proof}

First we assume $A$ is locally closed and suppose $A=O\cap F$ where $O\subseteq X$ is open and $F\subseteq X$ is closed. By \textbf{Remark \ref{Remark 6.2}}, $A$ is locally closed at $x$ for all $x\notin \overline{A}$. For each $x\in A$, $O\in\mathcal{U}(x)$ and $O\cap A = O\cap F$, which is clearly relatively closed in $O$. Fix $x\in A'\backslash A$. First we have $\overline{A} \subseteq \overline{O}\cap F$, so $x\in F$. Since $x\notin A$, we must have $x\in F\backslash O$. In this case, for all $V\in\mathcal{U}(x)$, $V\cap (F\backslash O)$ is relatively closed in $V$ as $F\backslash O$ is closed with respect to $\tau$.\\

\noindent
Next suppose for each $x\in X$, $A$ is locally closed at $x$. Fix $x\in A$. By \textbf{Remark \ref{Remark 6.2}} there exists $V\in \mathcal{U}(x)$ such that $V\cap A = V\cap F$ for some set $F\subseteq X$ closed with respect to $\tau$. Since $V\backslash F = V\backslash A\in \tau$, we have $V\cap \overline{A} = V\cap F\subseteq A$. Since $V\cap \overline{A}$ is a $\tau\Big|_{\overline{A}}$-open neighborhood of $x$, we can conclude that $A$ is relatively open in $\overline{A}$. $(c)\implies (a)$ is immediate as $\overline{A}$ is closed with respect to $\tau$.
    
\end{proof}

\begin{defn}[\cite{34}]\label{Definition 6.4}

In a topological space $(X, \tau)$, define:

$$
\mathcal{B}_{Sk} = \big\{U\backslash V:U, V\in\tau \big\}
$$
which is the set of all locally closed sets in $X$. Clearly $\mathcal{B}_{Sk}$ is closed under finite intersection. Hence $\mathcal{B}_{Sk}$ is the base of a topology $\operatorname{Sk}(\tau)$ on $X$, called the \textbf{Skula topology} of $\tau$. This is introduced by \cite{34} and was called the \textbf{b-topology} of $X$. The space $\operatorname{Sk}(X) = \big( X, \operatorname{Sk}(\tau) \big)$ is called the \textbf{Skula-modification} of $X$. Clearly $\tau\subseteq \operatorname{Sk}(\tau)$. If we use $\sim_{\tau}$ to denote the equivalence relation induced by the specialization preorder in $(X, \tau)$, then all sets in $\operatorname{Sk}(\tau)$ are $\sim_{\tau}$-saturated, and hence $\operatorname{Sk}(\tau) \subseteq \Sigma(\sim_{\tau})$ (see \textbf{Definition \ref{Definition 0.1}}). From now on, for any $A\subseteq X$, we still use $\overline{A}$ to denote the closure of $A$ with respect to $\tau$, but $\operatorname{cl}_{\operatorname{Sk}(\tau)}(A)$ to denote the closure of $A$ with respect to $\operatorname{Sk}(\tau)$.
    
\end{defn}

\begin{prop}\label{Proposition 6.5}

In the set-up of \textbf{Definition \ref{Definition 6.4}}, let $\preccurlyeq_{\operatorname{Sk}(\tau)}$ be the specialization preorder with respect to $\operatorname{Sk}(\tau)$. Then for each $x, y\in X$ and $A\subseteq X$:

\begin{enumerate}[label = (\alph*)]

    \item the collection $\big\{ U\cap \overline{\{x\}}: U\in\mathcal{U}(x) \big\}$ is a neighborhood base of $x$ with respect to $\operatorname{Sk}(\tau)$.
    \item $x\in \operatorname{cl}_{\operatorname{Sk}(\tau)}(A)$ if and only if $x\in \overline{A\cap \overline{\{x\}}}$.
    \item $x\preccurlyeq_{\operatorname{Sk}(\tau)} y$ if and only if $x\sim_{\tau} y$. Hence $[x]_{\sim} = \operatorname{cl}_{\operatorname{Sk}(\tau)}(x)$.
    \item given another topological space $(Y, \tau_Y)$, if $f:(X, \tau) \rightarrow (Y, \tau_Y)$ is continuous, then so is $f:\operatorname{Sk}(X) \rightarrow \operatorname{Sk}(Y)$.
    \item the following are equivalent:
    \begin{itemize}
        \item $X$ is $T_0$.
        \item $\operatorname{Sk}(X)$ is Hausdorff.
        \item $\operatorname{Sk}(X)$ is $T_0$.
    \end{itemize}
    
\end{enumerate}
    
\end{prop}

\begin{proof}

\begin{enumerate}[label = (\alph*)]

    \item for each $V\in\tau$ with $x\in V^c$, we have $\overline{\{x\}} \subseteq V^c$. Hence given $U, V\in \tau$, if $x\in U\backslash V$, then $U\in\mathcal{U}(x)$, $U\cap\overline{\{x\}} \subseteq U\backslash V$ and clearly $U\cap \overline{\{x\}} \in \operatorname{Sk}(\tau)$.

    \item by part (a), $x\in \operatorname{cl}_{\operatorname{Sk}(\tau)}(A)$, if and only if for all $U\in\mathcal{U}(x)$, $U\cap \overline{\{x\}}\cap A \neq\emptyset$, which is equivalent to $x\in \overline{A\cap \overline{\{x\}}}$.

    \item by part (b) and \textbf{Proposition \ref{Proposition 1.2}}, $x\preccurlyeq_{\operatorname{Sk}(\tau)} y$ if and only if $x\in \overline{\{y\}\cap \overline{\{x\}}}$, if and only if $x\preccurlyeq y$ and $y\preccurlyeq x$, or $x\sim_{\tau} y$, as $\operatorname{cl}_{\operatorname{Sk}(\tau)}(\{y\})$ is never empty.

    \item by part (a), when $f$ is continuous, for each $x\in X$, observe that for each $V\in\mathcal{U}\big( f(x) \big)$:

    $$
    f^{-1}\left( V\cap \overline{\{f(x)\}} \right)\in \operatorname{Sk}(\tau)
    $$
    and hence $f$ is also $\operatorname{Sk}(\tau)$-$\operatorname{Sk}(\tau_Y)$ continuous.

    \item By \textbf{Proposition \ref{Proposition 4.5}}, when $\operatorname{Sk}(X)$ is Hausdorff, $\operatorname{Sk}(X)$ is $T_0$. When $\operatorname{Sk}(X)$ is $T_0$, suppose we have $x, y\in X$ such that $x\sim_{\tau} y$. This implies $\overline{\{x\}} = \overline{\{y\}}$ by \textbf{Proposition \ref{Proposition 1.2}}, or:

    $$
    \overline{\{y\}\cap \overline{\{x\}}} = \overline{\{x\}\cap \overline{\{y\}}}
    $$
    which implies $x\preccurlyeq_{\operatorname{Sk}(\tau)} y$ and $y\preccurlyeq_{\operatorname{Sk}(\tau)}x$ according to part (c). Then by our assumption $x=y$. When $X$ is $T_0$, fix two different $x, y\in X$. Then without losing generality, we can assume $x\notin \overline{\{y\}}$. Let $U\in\mathcal{U}(x)$ be such that $U\cap \overline{\{y\}} = \emptyset$. Then for any $V\in\mathcal{U}(x)$, $U\cap \overline{\{x\}}$ and $V\cap \overline{\{y\}}$ are disjoint.
    
\end{enumerate}
    
\end{proof}

\begin{exmp}\label{Example 6.6}

We introduce characterizations of a \textbf{almost discrete} topological space in \textbf{Proposition \ref{Proposition 4.18}}. Given a topological space $(X, \tau)$, by \textbf{Proposition \ref{Proposition 4.18}(d)} and \textbf{Proposition \ref{Proposition 6.5}(a)}, we can see $\tau=\operatorname{Sk}(\tau)$ if and only if $X$ is almost discrete. Another example is the \textbf{Sogenfrey topology} defined on $\mathbb{R}$, which is the topology with the base $\big\{ (a,b]:a, b\in \mathbb{R} \big\}$. The \textbf{right-order topology} defined on $\mathbb{R}$ is the topology with the base $\big\{ (a, \infty):a\in \mathbb{R} \big\}$. Clearly the Sogenfrey topology is the Skula-modification of the right-order topology. 
    
\end{exmp}

\noindent
We will again close this section by discovering the relation between the $T_0$-quotient of a topological space $(X, \tau)$ and $\operatorname{Sk}(X)$. From \textbf{Proposition \ref{Proposition 6.5}(e)}, we can see that the Skula-modification preserves the $T_0$-quotient.

\begin{prop}

Given a topological space $(X, \tau)$, $\operatorname{Sk}(X)_0 = \operatorname{Sk}(X_0)$.
    
\end{prop}

\begin{proof}

From \textbf{Proposition \ref{Proposition 6.5}(c)}, we can see $\sim_{\operatorname{Sk}(\tau)}$ coincide with $\sim_{\tau}$. Hence the two quotient mappings with respect to $\tau$ and $\operatorname{Sk}(\tau)$ coincide. Without confusion we use $q$ to denote both quotient mappings. Let $\tau_0$ denote the quotient topology with respect to $\tau$, and $\operatorname{Sk}(\tau)_0$ denote the quotient topology with respect to $\operatorname{Sk}(\tau)$. It remains to show that $\operatorname{Sk}(\tau)_0 = \operatorname{Sk}(\tau_0)$. This immediately follows by \textbf{Theorem \ref{Theorem 2.4}}.
    
\end{proof}

\section{\texorpdfstring{$T_D$}{TD} spaces}

\begin{defn}

Given a topological space $(X, \tau)$, $x\in X$ and $S\subseteq X$, we call $x$ a $T_D$-\textbf{point}, or \textbf{locally closed point} of $X$ if $\{x\}$ is locally closed; we cal $x$ a \textbf{weakly isolated point} of $S$ if there is an open set $U$ such that $x\in U\cap S\subseteq \overline{\{x\}}$. We call $X$ a $T_D$-\textbf{space} if each $x\in X$ is a $T_D$-point.
    
\end{defn}

\noindent
The next two results combine results from {\cite[p.93]{7}} and {\cite[p.8]{4}}, and provide characterizations of $T_D$-spaces.

\begin{lem}\label{Lemma 7.2}

Given a topological space $(X, \tau)$ and $x\in X$, the following are equivalent:

\begin{enumerate}[label = (\alph*)]

    \item $x$ is a $T_D$-point of $X$.
    \item there exists $U\in\mathcal{U}(x)$ such that $U\backslash \{x\}$ is open.
    \item $\{x\}'$ is closed.
    \item for any $S\subseteq X$, if $x$ is weakly isolated in $S$, then $x$ is isolated in $S$.
    \item $x$ is isolated in $\operatorname{Sk}(X)$.
    \item for any $A\subseteq X$ with $x\in A'$, $A'$ is closed.
    
\end{enumerate}
    
\end{lem}

\begin{proof}

We will first prove the equivalence between statements $(a)\sim(e)$, and then show that $(b)\implies (f)$, and $(f)\implies (a)$. First assume $x$ is a $T_D$-point. Suppose $\{x\}=U\cap F$ for some open subset $U$ and closed subset $F$. Then $U\backslash \{x\} = U\backslash F$ is open. Next, if there exists $U\in\mathcal{U}(x)$ such that $U\backslash \{x\}$ is open, let $V=U\backslash \{x\}$. Then observe that:

$$
\{x\}' = \overline{\{x\}} \backslash \{x\} = \overline{\{x\}} \cap \big( V\cup U^c \big) = \overline{\{x\}}\cap U^c
$$
and hence $\{x\}'$ is closed. Now suppose $\{x\}'$ is closed. Fix $S\subseteq X$ and suppose $\{x\}$ is weakly isolated in $S$. Let $U$ be an open subset such that $x\in U\cap S\subseteq \overline{\{x\}}$. Then $V=U\backslash \{x\}'$ is an open neighborhood of $x$ such that:

$$
x\in V\cap S \subseteq \overline{\{x\}}\backslash \{x\}' = \{x\}
$$
Hence $x$ is isolated in $S$. Next suppose that for any $S\subseteq X$, if $x$ is weakly isolated in $S$, then $x$ is isolated in $S$. Note that $x$ is always weakly isolated in $X$, and hence is isolated in $X$. We then have for any $U\in\mathcal{U}(x)$, $\{x\} = U\cap \overline{\{x\}}$. Hence $\{x\}\in \operatorname{Sk}(\tau)$. The implication $(e)\implies (a)$ follows by the definition.\\

\noindent
Now assume there exists $U\in\mathcal{U}(x)$ such that $U\backslash \{x\}$ is open. Let $A\subseteq X$ be such that $x\notin A'$. Then there exists $V\in \mathcal{U}(x)$ such that $V\cap A\backslash \{x\} = \emptyset$. We then have:

\begin{equation}\label{e10}
\big( U\backslash\{x\} \big) \cap (V\cap A) = U\cap \big( V\cap A\backslash \{x\} \big) = \emptyset
\end{equation}
Since $U\backslash \{x\}\cap V$ is open and we assume that $x\notin A'$, (\ref{e10}) implies that:

$$
\big( U\backslash \{x\} \big)\cap V\cap \overline{A} = \emptyset \hspace{0.4cm} \Longrightarrow \hspace{0.4cm} U\cap V\cap A'=\emptyset
$$
Since $U\cap V\in \mathcal{U}(x)$, we can conclude that $x\notin \overline{A'}$, which implies $\overline{A'} \subseteq A'$. Next assume that for any $A\subseteq X$, if $x\notin A'$, then $x\notin \overline{A'}$. We then have $x\notin \overline{\{x\}'}$. Hence there exists $U\in\mathcal{U}(x)$ such that:

$$
\emptyset = U\cap\{x\}'=  U\cap \overline{\{x\}}\backslash \{x\} \hspace{0.4cm} \Longrightarrow \hspace{0.4cm} \{x\} = U\cap \overline{\{x\}}
$$
which implies $\{x\}$ is locally closed.
    
\end{proof}

\begin{prop}\label{Proposition 7.3}

In the set-up of \textbf{Lemma \ref{Lemma 7.2}}, the following are equivalent:

\begin{enumerate}[label = (\alph*)]

    \item $X$ is $T_D$ space.
    \item for each $x\in X$, there exists $U\in\mathcal{U}(x)$ such that $U\backslash \{x\}$ is open.
    \item for any $x\in X$, $\{x\}'$ is closed.
    \item for any $A\subseteq X$, $A'$ is closed.
    \item for any $S\subseteq X$, any weakly isolated points of $S$ is an isolated point.
    \item $\operatorname{Sk}(X)$ is discrete.
    
\end{enumerate}
    
\end{prop}

\begin{proof}

The proof follows immediately by \textbf{Lemma \ref{Lemma 7.2}}.
    
\end{proof}

\begin{prop}\label{Proposition 7.4}

$T_1\Longrightarrow T_D \Longrightarrow T_0$
    
\end{prop}

\begin{proof}

Fix a topological space $(X, \tau)$ and $x\in X$. If $\{x\}$ is closed, then $\{x\}$ is locally closed. With \textbf{Proposition \ref{Proposition 3.2}}, this shows $T_1\Longrightarrow T_D$. If $\{x\}$ is locally closed, by \textbf{Lemma \ref{Lemma 7.2}}, $\{x\}'$ is closed. This shows that $\{x\}=[x]_{\sim}$ by \textbf{Proposition \ref{Proposition 2.1}}. We can then conclude that $T_D\Longrightarrow T_0$.
    
\end{proof}

\begin{exmp}\label{Example 7.5}

Let $(X, \tau)$ be an Alexandroff space (see \textbf{Definition \ref{Definition 1.8}}). By \textbf{Theorem \ref{Theorem 1.9}(f)}, for each $x\in X$, $[x]_{\sim}$ is locally closed. Hence a $T_0$ Alexandroff space is $T_D$. However, a $T_1$ Alexandroff space is discrete according to \textbf{Theorem \ref{Theorem 1.9}(b)}, and the converse is obviously true.
    
\end{exmp}

\begin{exmp}

Recall that in \textbf{Example \ref{Example 6.6}} we introduce $\tau_{\rightarrow}$ the \textbf{right-order topology} defined on $\mathbb{R}$. Such topology can also be defined on any totally ordered set. Let $(X, \leq)$ be a totally ordered set and let it be equipped with $\tau_{\rightarrow}$ the right-order topology. Then observe that for any $x\in X$, $\overline{\{x\}} = [-\infty, x]$ and $\bigcap\mathcal{N}_x = [x, \infty)$. This implies that $(X, \tau_{\rightarrow})$ is $T_0$, and, by \textbf{Lemma \ref{Lemma 7.2}}, for any $x\in X$, $\{x\}$ is locally closed precisely when $x=\min X$ or it has an immediate predecessor. We can then conclude that, when $X=\mathbb{R}$, $(X, \tau_{\rightarrow})$ is $T_0$ but not $T_D$. In this case, none of the points in $\mathbb{R}$ are $T_D$-points.
    
\end{exmp}

\noindent
In \textbf{Example \ref{Example 7.5}}, we consider a space $(X, \tau)$ where each point in $X_0$ is locally closed. Clearly this property is equivalent to that $q(X)$ is $T_D$. Next we will study this non-$T_0$ version of $T_D$ property.

\begin{defn}\label{Definition 7.7}

In a topological space $(X, \tau)$, a point $x\in X$ is called an $R_d$-\textbf{point} if $[x]_{\sim}$ is locally closed. We call $X$ a $R_d$-space if all points are $R_d$-points. Note that in \cite{22} a $R_d$ space is called an \textit{essentially} $T_D$-space.
    
\end{defn}

\begin{rem}

One can find the $R_D$ axioms, a topological property different from $R_d$, in, for instance, \cite{28} and \cite{35}. Note that a topological space $(X, \tau)$ is $R_D$ if for each $x\in X$, $\{x\}'$ is closed whenever $x$ satisfies $[x]_{\sim} = \{x\}$. We shall not have more to say about this, except noting that $R_d$ implies $R_D$ according to \textbf{Lemma \ref{Lemma 7.9}(c)} and the fact that $\{x\}' = \overline{\{x\}} \backslash \{x\}$. Moreover, it is obvious that $X$ is $T_D$ precisely when it is $T_0$ and $R_D$. This could be a more useful criterion for checking the $T_D$ property than verifying $T_0$ and $R_d$.
    
\end{rem}

\noindent
Now we shall mainly be interested in the characterizations of $R_d$ spaces in terms of the Skula topology. Other equivalent formulations are stated for the sake of completeness. As such, in the next lemma, we will only prove what we need, while the rest can be proven by invoking \textbf{Lemma \ref{Lemma 7.2}} or passing to the $T_0$-quotient. Given a topological space $(X, \tau)$ and $A\subseteq X$, recall the definition of $A^{\triangledown}$ given in \textbf{Definition \ref{Definition 2.6}}.

\begin{lem}\label{Lemma 7.9}

In a topological space $(X, \tau)$, let $\sim_{\tau}$ denote the equivalence relation induced the specialization preorder with respect to $\tau$. For any $x\in X$, the following are equivalent:

\begin{enumerate}[label = (\alph*)]

    \item $x$ is an $R_d$-point.
    \item there exists $U\in\mathcal{U}(x)$ such that $U\backslash [x]_{\sim_{\tau}}$ is open.
    \item $\{x\}^{\triangledown}$ is closed.
    \item for any $S\subseteq X$, if $[x]_{\sim_{\tau}} \subseteq U\cap S \subseteq \overline{\{x\}}$ for some $U\in\mathcal{U}(x)$, then $[x]_{\sim_{\tau}}$ is (relatively) open in $S$.
    \item $[x]_{\sim_{\tau}} \in \operatorname{Sk}(\tau)$.
    \item for any $A\subseteq X$, $A^{\triangledown}$ is closed.
    
\end{enumerate}
    
\end{lem}

\begin{proof}

We only need to prove $(a)\Longleftrightarrow (e)$. Clearly $(a)\implies (e)$. Conversely, if $[x]_{\sim_{\tau}}\in \operatorname{Sk}(\tau)$, by \textbf{Proposition \ref{Proposition 6.5}} there exists $U\in\tau$ such that $U\cap \overline{\{x\}} \subseteq [x]_{\sim_{\tau}}$. Since $U\cap \overline{\{x\}}$, as an intersection of two $\sim_{\tau}$-saturated sets, is $\sim_{\tau}$-saturated, we have $[x]_{\sim_{\tau}} \subseteq U\cap \overline{\{x\}}$. Hence $[x]_{\sim_{\tau}} = U\cap \overline{\{x\}}$ is locally closed. 
    
\end{proof}

\begin{prop}\label{Proposition 7.10}

Given a topological space $(X, \tau)$, let $\sim_{\tau}$ denote the equivalence relation induced the specialization preorder with respect to $\tau$. Then the following are equivalent:

\begin{enumerate}[label = (\alph*)]

    \item $X$ is an $R_d$-space.
    \item for each $x\in X$, there exists $U\in \mathcal{U}(x)$ such that $U\backslash [x]_{\sim_{\tau}}$ is open.
    \item for each $x\in X$, $\{x\}^{\triangledown}$ is closed.
    \item for any $A\subseteq X$, $A^{\triangledown}$ is closed.
    \item $\operatorname{Sk}(\tau) = \Sigma(\sim_{\tau})$
    
\end{enumerate}
    
\end{prop}

\begin{proof}

Again we only need to prove $(a)\Longleftrightarrow (e)$. If $A\subseteq X$ is $\sim_{\tau}$-saturated, then $A = \bigcup_{a\in A}[a]_{\sim_{\tau}}$. Then by \textbf{Lemma \ref{Lemma 7.9}}, $X$ is an $R_d$-space if and only if $\Sigma(\sim_{\tau})\subseteq \operatorname{Sk}(\tau)$. Meanwhile, by definition we always have $\operatorname{Sk}(\tau)\subseteq \Sigma(\sim_{\tau})$, and hence we can conclude $X$ is an $R_d$-space if and only if $\operatorname{Sk}(\tau) = \Sigma(\sim_{\tau})$.
    
\end{proof}

\begin{prop}\label{Proposition 7.11}

Given a topological space $(X, \tau)$, let $\sim_{\tau}$ denote the equivalence relation induced the specialization preorder with respect to $\tau$. Then:

\begin{enumerate}[label = (\alph*)]

    \item $X$ is $R_d$ if and only if $\operatorname{Sk}(X)$ is Alexandroff.
    \item $\operatorname{Sk}\big( \operatorname{Sk}(\tau) \big) = \Sigma(\sim_{\tau})$.
    
\end{enumerate}
    
\end{prop}

\begin{proof}

We use $\preccurlyeq_{\operatorname{Sk}(\tau)}$ to denote the specialization preorder in $\operatorname{Sk}(X)$. When $X$ is $R_d$, we have $\operatorname{Sk}(\tau) = \Sigma(\sim_{\tau})$, which makes $\operatorname{Sk}(X)$ Alexandroff according to \textbf{Theorem \ref{Theorem 1.9}}. Conversely, if $\operatorname{Sk}(X)$ is Alexandroff, $\operatorname{Sk}(\tau) = \Sigma\big( \preccurlyeq_{\operatorname{Sk}(\tau)} \big)$. Hence by \textbf{Proposition \ref{Proposition 6.5}}, $\operatorname{Sk}(\tau) = \Sigma(\sim_{\tau})$ and hence $X$ is $R_d$ by \textbf{Proposition \ref{Proposition 7.10}}. According to \textbf{Proposition \ref{Proposition 6.5}(c)}, for each $x\in X$, we have:

$$
[x]_{\sim_{\operatorname{Sk}(\tau)}} = \operatorname{cl}_{\operatorname{Sk}(\tau)}\big( \{x\} \big) = [x]_{\sim_{\tau}}
$$
which implies the equivalence class of $x$ with respect to $\operatorname{Sk}(\operatorname{Sk}(\tau))$ is closed with respect to $\operatorname{Sk}(\tau)$. Then our conclusion follows by \textbf{Proposition \ref{Proposition 7.10}(b)}.
    
\end{proof}

\begin{prop}

Given a topological space $(X, \tau)$, $X$ is $R_d$ if and only if for any $A\subseteq X$, $D(A)$ is closed with respect to $\tau$ (see \textbf{Definition \ref{Definition 2.7}} for the definition of $D(A)$).
    
\end{prop}

\begin{proof}

When $X$ is $R_d$, by \textbf{Proposition \ref{Proposition 7.11}}, $\operatorname{Sk}(X)$ is Alexandroff. Since $\tau\subseteq \operatorname{Sk}(\tau)$, $D(A)$ is closed by \textbf{Proposition \ref{Proposition 2.9}(b)}. Similar to what we showed in the proof of \textbf{Proposition \ref{Proposition 2.9}}, for each $x\in X$:

$$
D\big(\{x\} \big) = \overline{\{x\}} \backslash \big[ \{x\} \backslash \{x\}' \big] = \overline{\{x\}} \backslash [x]_{\sim} = \{x\}^{\triangledown}
$$
Then the converse follows immediately by \textbf{Lemma \ref{Lemma 7.9}(c)}.
    
\end{proof}

\noindent
As we are also interested in lattices and frames, below we will include their relations with $T_D$ and $R_d$ spaces.

\begin{prop}

In a topological space $(X, \tau)$, for any $A\subseteq X$, define an equivalence relation $E_A$ on $\tau$ by the following: for any $U, V\in \tau$, $(U, V)\in E_A$ if $U\cap A=V\cap A$. Then:

\begin{enumerate}
    \item for any $A\subseteq X$, $E_A$ is a frame congruence defined on $\tau$.

    \item for any $A\subseteq B\subseteq X$, $E_A \subseteq E_B$ if and only if $\operatorname{cl}_{\operatorname{Sk}(\tau)}(B) \subseteq \operatorname{cl}_{\operatorname{Sk}(\tau)}(A)$.

    \item the mapping $A\rightarrow E_A$ is injective if and only if $(X, \tau)$ is a $T_D$-space.
\end{enumerate}
    
\end{prop}

\begin{proof}

\begin{enumerate}[label = (\alph*)]

\item As we point out in \textbf{Example \ref{Example 0.11}}, for each $A\subseteq X$, $E_A = \operatorname{Ker}\big( \iota_A^{\leftarrow} \big)$ where $\iota_A:A\rightarrow X$ is the inclusion mapping. Hence for each $A\subseteq X$, $E_A$ is a frame congruence.

\item Fix $A\subseteq B\subseteq X$. Note that for any $U, V\in \tau$, $(U, V)\in E_A$ precisely when $(U\,\Delta\,V) \cap A=\emptyset$ (where $U\,\Delta\,V$ is the symmetric difference of $U$ and $V$). When $E_A\subseteq E_B$, fix $x\in \operatorname{cl}_{\operatorname{Sk}(\tau)} (B)$. Then for any $U, V\in \tau$ such that $x\in U\backslash V$, we have:

$$
(U\backslash V)\cap B\neq\emptyset \hspace{0.4cm} \Longrightarrow \hspace{0.4cm} (U\,\Delta\,V)\cap B\neq \emptyset \hspace{0.4cm} \Longrightarrow \hspace{0.4cm} (U\,\Delta\,V)\cap A \neq\emptyset \hspace{0.4cm} \Longrightarrow \hspace{0.4cm} x\in \operatorname{cl}_{\operatorname{Sk}(\tau)}(A)
$$
Conversely, if $\operatorname{cl}_{\operatorname{Sk}(\tau)}(B) \subseteq \operatorname{cl}_{\operatorname{Sk}(\tau)}(A)$, let $(U, V)\in E_A$. Then we have $U\,\Delta\,V \subseteq X\backslash \operatorname{cl}_{\operatorname{Sk}(\tau)}(A) \subseteq X\backslash \operatorname{cl}_{\operatorname{Sk}(\tau)}(B)$, which implies $(U\,\Delta\, V)\cap B = \emptyset$, or $(U, V)\in E_B$.

\item First suppose $X$ is $T_D$. Fix $A, B\subseteq X$. Fix $x\in X$ and by \textbf{Lemma \ref{Lemma 7.2}(b)}, there exists $U\in \mathcal{U}(x)$ such that $U\backslash \{x\}$ is open. Assume $E_A=E_B$. Then $U\backslash \big( U\backslash \{x\} \big) \cap A = \emptyset$ if and only if $U\backslash \big( U\backslash \{x\} \big) \cap B = \emptyset$. This implies for any $x\in X$, $x\notin A$ if and only if $x\notin B$, and hence $A=B$. Next suppose the mapping $A\mapsto E_A$ is injective. Note that for any $A\subseteq X$, we have $\operatorname{cl}_{\operatorname{Sk}(\tau)}(A') = \operatorname{cl}_{\operatorname{Sk}(\tau)} \big( \overline{A'} \big)$ where $A'$ is the derived set of $A$. By assumption we have $A=\overline{A'}$, which implies $X$ is $T_D$ by \textbf{Proposition \ref{Proposition 7.3}}.

\end{enumerate}
\end{proof}

\begin{prop}

Given two topological space $(X, \tau_X)$, $(Y, \tau_Y)$, recall from \textbf{Example \ref{Example 0.11}} that a $\tau_X$-$\tau_Y$ continuous mapping $f:X\rightarrow Y$ induces a frame homomorphism $f^{\leftarrow}: \tau_Y \rightarrow \tau_X$. If $X$ is $T_D$ and $Y$ is $T_0$, then any frame isomorphisms $\phi:\tau_Y \rightarrow \tau_X$ is induced by a unique injective $\tau_X$-$\tau_Y$ continuous mapping $f_{\phi}:X\rightarrow Y$ (namely $\phi = f_{\phi}^{\leftarrow}$). As a consequence, when both $X$ and $Y$ are $T_D$, the unique continuous mapping $f_{\phi}$ induced by the given frame homomorphism is a homeomorphism. In this case, $\phi\mapsto f_{\phi}$ is a bijective correspondence between frame isomorphism from $\tau_Y$ to $\tau_X$ and homeomorphisms from $X$ to $Y$.
    
\end{prop}

\begin{proof}

Fix $x\in X$. By \textbf{Lemma \ref{Lemma 7.2}} let $U\in\mathcal{U}(x)$ be such that $U\backslash \{x\}$ is open. Suppose there exists two different points $y, y'$ in $\phi^{-1}(U) \backslash \phi^{-1}\big( U\backslash \{x\} \big)$. Since $(Y, \tau_Y)$ is $T_0$, suppose there exists $V\in\mathcal{U}(y)$ such that $y'\notin V$ and $V\subseteq \phi^{-1}(U)$, $V\backslash \phi^{-1} \big( U\backslash \{x\} \big) \neq\emptyset$. Let $V'\in \mathcal{U}(y')$ such that $V'\neq V$. We then have:

$$
\begin{aligned}
& \big( U\backslash \{x\} \big) \cup \phi(V) = \phi\Big( \phi^{-1}\big( U\backslash \{x\} \big) \Big) \cup \phi(V) = \phi\Big( \phi^{-1}\big( U\backslash \{x\} \big) \cup V \Big) \subsetneq \phi\big( \phi^{-1}(U) \big) = U\\
& \big( U\backslash \{x\} \big) \cup \phi(V') = \phi\Big( \phi^{-1}\big( U\backslash \{x\} \big) \Big) \cup \phi(V') = \phi\Big( \phi^{-1}\big( U\backslash \{x\} \big) \cup V' \Big) \subseteq \phi\big( \phi^{-1}(U) \big) = U\\
\end{aligned}
$$
which implies $\phi(V) = \phi(V')$. However, $V\neq V'$ by our assumption, and hence we obtain a contradiction. We can now conclude that $\phi^{-1}(U) \backslash \phi^{-1}\big( U\backslash \{x\} \big)$ contains only one point, say $\{y\}$. Suppose for a different $U'\in \mathcal{U}(x)$, $U'\backslash \{x\}$ is open. By the same reasoning above, we can show that the set $\phi^{-1}(U') \backslash \phi\big( U'\backslash \{x\} \big)$ contains only one point $y'$, and so does the following set:

\begin{equation}\label{e17}
\phi^{-1} \big( U\cup U' \big) \backslash \phi^{-1} \big( (U\cup U') \backslash \{x\} \big) = \Big( \{y\} \backslash \phi^{-1} \big( U'\backslash \{x\} \big) \Big) \cup \Big( \{y'\} \backslash \phi^{-1} \big( U\backslash \{x\} \big) \Big)
\end{equation}
Since $x\in U\cap U'$, $\phi^{-1}(U\cap U') \neq \phi^{-1}\big( (U\cup U') \backslash \{x\} \big)$. Then observe that:

$$
\begin{aligned}
\emptyset
& \neq  \phi^{-1}(U\cap U') \backslash \phi^{-1} \big( (U\cup U') \backslash \{x\} \big) \\
& = \phi^{-1}(U\cap U') \cap \phi^{-1} \big( U\cup U' \big) \backslash \phi^{-1} \big( (U\cup U') \backslash \{x\} \big) \\
& = \Big( \phi^{-1}(U')\cap \{y\} \backslash \phi^{-1} \big( U'\backslash \{x\} \big) \Big) \cup \Big( \phi^{-1}(U) \cap \{y'\} \backslash \phi^{-1} \big( U\backslash \{x\} \big) \Big) = \{y'\} \cap \{y\}
\end{aligned}
$$
which implies $y=y'$. We can now conclude that for any $U\in \mathcal{U}(x)$ with $U\backslash \{x\}\in \tau_X$, the point in $\phi^{-1}(U) \backslash \phi^{-1}\big( U\backslash \{x\} \big)$ is independent to the choice of such $U$. We use $y_x$ to denote the point in $\phi^{-1}(U) \backslash \phi^{-1}\big( U\backslash \{x\} \big)$. Then the following mapping is well-defined:

$$
f_{\phi}:X\rightarrow Y, \hspace{0.4cm} x\mapsto y_x
$$
Clearly for any $V\in\tau_Y$, $f_{\phi}^{-1}(V) = \phi(V)$, which implies $f$ is continuous. Suppose there exists two different points $x, x'\in X$ such that $f_{\phi}(x) = f_{\phi} (x') = z$. By \textbf{Proposition \ref{Proposition 7.4}}, we can assume that for some $U\in \mathcal{U}(x)$, $U'\in \mathcal{U}(x')$ with $x\notin U'$ such that $U\backslash \{x\}\in \tau_X$ and $U'\backslash \{x'\} \in \tau_X$. Then there exists $V, V'\in \mathcal{U}(z)$ such that $V\subseteq \phi^{-1}(U)$ and $V'\subseteq \phi^{-1}(U')$. This implies:

$$
x\in f_{\phi} ^{-1}(V)\cap f_{\phi}^{-1}(V') = \phi(V\cap V') \subseteq U\cap U'
$$
which leads to a contradiction. Hence we must have $f_{\phi} (x) \neq f_{\phi} (x')$ whenever $x\neq x'$, or $f_{\phi}$ is injective. If $g:X\rightarrow Y$ is another $\tau_X$-$\tau_Y$ continuous function such that $g^{\leftarrow} = \phi$, suppose for some $x\in X$, $f(x)\neq g(x)$. Since $Y$ is $T_0$, there exists $U\in \mathcal{U}\big( f_{\phi} (x) \big)$, $U'\in \mathcal{U}\big( g(x) \big)$ such that $f_{\phi} (x)\notin U'$. We then have:

$$
x\in f_{\phi}^{-1}(U) \cap g^{-1}(U') = \phi(U\cap U') = f_{\phi}^{-1}(U\cap U')
$$
which leads to a contradiction. Hence $f_{\phi}$ is the unique $\tau_X$-$\tau_Y$ continuous injective function such that $f_{\phi}^{\leftarrow} = \phi$. When $X$ and $Y$ are $T_D$, there exists a unique injective $\tau_Y$-$\tau_X$ continuous function $f_{\phi^{-1}}: Y\rightarrow X$ such that $f_{\phi^{-1}}^{\leftarrow} = \phi^{-1}$. From the assumption that $\phi$ is a frame isomorphism we must have both $f_{\phi}$ and $f_{\phi^{-1}}$ are surjective. For any $x\in X$, let $U\in \mathcal{U}(x)$ be such that $U\backslash \{x\} \in\tau_X$. Then:

$$
\big( f_{\phi^{-1}} \circ f_{\phi} \big)^{-1} \big(\{x\} \big) = \big( f_{\phi^{-1}} \circ f_{\phi} \big)^{-1}\Big( U\backslash \big( U\backslash \{x\} \big) \Big) = \phi\circ\phi^{-1}(U) \backslash \phi\circ \phi^{-1}\big( U\backslash \{x\} \big) = \{x\}
$$
which shows that $f_{\phi^{-1}} \circ f_{\phi} =\operatorname{id}_X$. Similarly, we can show $f_{\phi} \circ f_{\phi^{-1}} = \operatorname{id}_Y$. Since $f_{\phi}$ is invertible and has continuous inverse, $f_{\phi}$ is a homeomorphism. 
    
\end{proof}

\begin{defn}

Let $L$ be a distributive lattice. A \textbf{filter} of $L$ is a sublattice $F$ of $L$ such that for any $a\in F$ and $b\in L$, $a\vee b\in F$. When $F$ is a filter of $L$, we call $F$ a \textbf{prime filter} if for any $a, b\in L$, $a\vee b\in F$ implies $a\in F$ or $b\in F$; we call $F$ a \textbf{slicing filter} if $F$ is a prime filter and there exists $a\in F$, $b\in L\backslash F$ such that $b$ immediately precedes $a$, namely $b<a$ and for any $x\in L$ such that $b\leq x\leq a$, we have either $x=b$ or $x=a$.
    
\end{defn}

\begin{prop}

Given a topological space $(X, \tau)$, let $\sim$ denote the equivalence relation associated with the specialization preorder. Note that $(\tau, \subseteq)$ is a bounded distributive lattice. Then:

\begin{enumerate}[label = (\alph*)]

    \item for any $x\in X$, if $x$ is an $R_d$-point, then $\mathcal{U}(x)$ is a slicing filter.
    \item for any $U, V\in\tau$, if $V$ immediately precedes $U$, then $V=U\backslash[x]_{\sim}$ for some $x\in X$.
    \item any slicing filter in $(\tau, \subseteq)$ is of the form $\mathcal{U}(x)$ for some $R_d$-point $x\in X$.
    \item $X$ is $R_d$ precisely when $\mathcal{U}(x)$ is a slicing filter for all $x\in X$.
    
\end{enumerate}
    
\end{prop}

\begin{proof}
\begin{enumerate}[label = (\alph*)]

\item Fix $x\in X$. Note that $\mathcal{U}(x)$ is clearly a prime filter. By \textbf{Lemma \ref{Lemma 7.9}}, there exists $U\in\mathcal{U}(x)$ such that $U\backslash [x]_{\sim}$. Suppose for some $V\in\tau$, $U\backslash [x]_{\sim} \subsetneq V$. If $x\in V$, we must have $[x]_{\sim} \subseteq U$ since $V$ is $\sim$-saturated. If $x\notin V$, $V$ must be disjoint from $[x]_{\sim}$, or $V\subseteq U\backslash [x]_{\sim}$. This shows that $U\backslash[x]_{\sim}$ immediately precedes $U$, and hence $\mathcal{U}(x)$ is a slicing filter.

\item Given $U\in\tau$, fix $x\in U$ so that $x\in \mathcal{U}(x)$. In part (a) we showed that $U\backslash [x]_{\sim}$ immediately precedes $U$. Given another $V\in\tau$ such that $V$ immediately precedes $U$. Since $V\subseteq V$, it suffices to show that $V=U\backslash [x]_{\sim}$. The proof follows by the same reasoning in part (a).

\item Fix $F\subseteq \tau$ a slicing filter. By part (b) there exists $U\in F$ and $x\in X$ such that $U\backslash [x]_{\sim}$ immediately precedes $U$ and $U\in \mathcal{U}(x)$. By \textbf{Lemma \ref{Lemma 7.9}} we have $x$ is an $R_d$-point. Since for any $U'\in F$, $U\cap U'\neq\emptyset$, we must have $F\subseteq \mathcal{U}(x)$. Let $U\in F$ be given by the slicing property such that $U\backslash [x]_{\sim}$ immediately precedes $U$. Then for any $V\in \mathcal{U}(x)$, we have $U\backslash [x]_{\sim} \cup V = U\cup V\in F$. Since $U\backslash [x]_{\sim} \notin F$, we must have $V\in F$ by $F$ being prime. Hence $\mathcal{U}(x) \subseteq F$, or $F=\mathcal{U}(x)$.

\item By part (a), if $X$ is $R_d$ then $\mathcal{U}(x)$ is a slicing filter for all $x\in X$. Conversely, if every $\mathcal{U}(x)$ is a slicing filter, then by part (c), every $x\in X$ is an $R_d$-point.

\end{enumerate}
\end{proof}

\noindent
We will close this section by showing that $T_D$ is hereditary, finite productive but infinitely productive. Moreover, Robinson and Wu in \cite{33} show that an infinite product of $T_D$ but not $T_1$ spaces is never $T_D$ when equipped with the product topology. We will slightly extend this result by considering the case where not all component spaces are $T_D$. However, when equipped with the box topology, an infinite product of spaces is $T_D$ precisely when each component space is $T_D$.

\begin{prop}

Given a topological space $(X, \tau)$:

\begin{enumerate}[label = (\alph*)]

    \item if $X$ is $T_D$, then for any $A\subseteq X$, $\left(A, \tau\Big|_A \right)$ is $T_D$.
    \item if $X=\prod_{\alpha\in \mathcal{A}}^{\text{box}}X_{\alpha}$ (see \textbf{Section 0.4}) where each $X_{\alpha}$ is non-empty, then $X$ is $T_D$ if and only if $X_{\alpha}$ is $T_D$ for all $\alpha\in \mathcal{A}$.
    
\end{enumerate}
    
\end{prop}

\begin{proof}

We will first show that $T_D$ is a hereditary property. Fix $A\subseteq X$ and $a\in A$. If $\{a\}$ is locally closed, then there exists an open set $U$ and a closed set $F$ such that $\{a\} = U\cap F$. Hence $\{a\} = (U\cap A) \cap (F\cap A)$. For the second part, note that in \textbf{Section 0.4}, we show that each $X_{\alpha}$ is homeomorphic to a subspace of $X$, and hence when $X$ is $T_D$, so is each $X_{\alpha}$. Conversely, fix $x\in X$. For each $\alpha\in \mathcal{A}$, using terminologies from \textbf{Section 0.4}, we have $\pi_{\alpha}(x)$ is a $T_D$-point for all $\alpha\in\mathcal{A}$. Then for each $\alpha\in \mathcal{A}$, there exists an open subset $U_{\alpha}\subseteq X_{\alpha}$ such that $\{\pi_{\alpha}(x)\} = U_{\alpha} \cap \overline{\{\pi_{\alpha}(x)\}}$. Now $U=\prod_{\alpha\in \mathcal{A}} U_{\alpha}$ is open in $X$. We also show that (\ref{e0}) holds when $X$ is equipped with the box topology. Therefore:

$$
\{x\} = \prod_{\alpha\in \mathcal{A}} \big\{\pi_{\alpha}(x) \big\} = \prod_{\alpha\in \mathcal{A}}\Big( U_{\alpha}\cap \overline{\{\pi_{\alpha}(x)\}} \Big) = \left( \prod_{\alpha\in \mathcal{A}}U_{\alpha} \right) \cap \prod_{\alpha\in \mathcal{A}} \overline{\{\pi_{\alpha}(x)\}} = \left( \prod_{\alpha\in \mathcal{A}} U_{\alpha} \right) \cap \overline{\prod_{\alpha\in \mathcal{A}} \big\{ \pi_{\alpha}(x) \big\}} = U\cap \overline{\{x\}}
$$
and hence $X$ is a $T_D$-space.
    
\end{proof}

\begin{prop}

Given $\big( (X_{\alpha}, \tau_{\alpha}) \big)_{\alpha\in \mathcal{A}}$ an infinite family non-$T_1$ topological spaces, the product space $X=\prod_{\alpha\in \mathcal{A}}X_{\alpha}$ (equipped with the product topology $\tau$) is not $T_D$.
    
\end{prop}

\begin{proof}

By assumption, there exists $x\in X$ such that $\big\{ \pi_{\alpha}(x) \big\}$ is not closed in $X_{\alpha}$ for each $\alpha\in \mathcal{A}$ according to \textbf{Proposition \ref{Proposition 3.2}}. Now we will show that this $x$ is not a $T_D$-point. Assume by contradiction that $x$ is a $T_D$-point. Then there exists $U\in \mathcal{U}(x)$ such that $\{x\} = U\cap \overline{\{x\}}$. We can then find a finite subset $F\subseteq \mathcal{A}$ and open subset $U_{\beta} \subseteq X_{\beta}$ for each $\beta\in  F$ such that:

$$
U_F = \bigcap_{\beta\in F}\pi_{\beta}^{-1}\big( U_{\beta} \big) \subseteq U \hspace{0.4cm} \Longrightarrow \hspace{0.4cm} \{x\} = U_F\cap \overline{\{x\}}
$$
Observe that in general, given $y, z\in X$, $y\preccurlyeq z$ with respect to $\tau$ if and only if $\pi_{\alpha}(y) \preccurlyeq \pi_{\alpha}(z)$ with respect to $\tau_{\alpha}$ for each $\alpha\in \mathcal{A}$. Hence we have:

$$
\begin{aligned}
& \hspace{1cm} \overline{\{x\}} = \prod_{\alpha\in \mathcal{A}} \overline{\{\pi_{\alpha}(x) \}}\\
& \implies\, U_F \cap \overline{\{x\}} = \left( \prod_{\beta\in F} U_{\beta}\cap \overline{\{ \pi_{\beta}(x)\}} \right) \times \left( \prod_{\alpha\in \mathcal{A}\backslash F} \overline{\{ \pi_{\alpha}(x) \}} \right)
\end{aligned}
$$
Note that for any $\alpha\in \mathcal{A}\backslash F$, $\overline{\{\pi_{\alpha}(x) \}}$ is not a singleton, and hence $U_F\cap \overline{\{x\}}$ is not a singleton since $\mathcal{A}$ is infinite. This leads to a contradiction. 
    
\end{proof}

\section{Sober spaces}

This section also serves as a transition to the next section. Since our main focus is on separation axioms in non-Hausdorff spaces, later sections will focus on introducing properties that are strictly weaker than Hausdorff. The property of being \textbf{sober} is a key to their set-up. Similar to Skula-modification, we will introduce the \textbf{sobrification}, which will lead us to several properties that are strictly stronger than $T_D$ but strictly weaker than Hausdorff. 

\begin{defn}

A topological space $(X, \tau)$ is \textbf{irreducible} if $X$ is non-empty and it is not the union of two closed proper subsets. A subset of $X$ is \textbf{irreducible} if it is irreducible when it is considered as a subspace equipped with the relative topology.
    
\end{defn}

\noindent
Given a topological space $(X, \tau)$, note that many authors will not require $X$ to be non-empty when defining irreducibility. Also we can immediately tell that $X$ is irreducible if and only if any two non-empty open proper subsets have non-empty intersections. For this reason, in other literature, an irreducible space is also called \textbf{hyper-connected} or \textbf{anti-Hausdorff}. We point out that irreducibility is an intrinsic property. Namely, if $A\subseteq Y\subseteq X$, then $A$ is irreducible as a subspace of $Y$ if and only if it is irreducible as a subspace of $X$. This is because the relation of being subspace is transitive. For now we present some characterizations of irreducible subsets.

\begin{prop}\label{Proposition 8.2}

In a topological space $(X, \tau)$, given a non-empty $A\subseteq X$, the following are equivalent:

\begin{enumerate}[label = (\alph*)]

    \item $A$ is irreducible.
    \item if $E_1, E_2\subseteq X$ are closed such that $A\subseteq E_1\cup E_2$, then $A\subseteq E_1$ or $A\subseteq E_2$.
    \item for any open subsets $U_1, U_2\subseteq X$, if both $A\cap U_1$ and $A\cap U_2$ are non-empty, then $A\cap (U_1\cap U_2)$ is also non-empty.
    \item every non-empty subset that is relatively open in $A$ is dense in $A$.
    \item $\overline{A}$ is irreducible.
    
\end{enumerate}
    
\end{prop}

\begin{proof}

The implication $(a)\implies (b)$ is immediate. To show $(b)\implies (c)$, suppose that there exist open subsets $U_1, U_2\subseteq X$ such that $A\cap (U_1\cap U_2) = \emptyset$, or $A\subseteq U_1^c \cup U_2^c$. By assumption we have $A\subseteq U_1^c$ or $A\subseteq U_2^c$, which leads to a contradiction. The implication $(c)\implies (d)\implies (a)$ is immediate. It remains to show the equivalence between $(a)$ and $(e)$. If $A$ is irreducible, suppose $E_1$, $E_2\subseteq X$ satisfies $\overline{A} \subseteq E_1\cup E_2$. Then by what we just proved, we have $A\subseteq E_1$ or $A\subseteq E_2$, which implies $\overline{A} \subseteq E_1$ or $\overline{A} \subseteq E_2$. This shows that $\overline{A}$ is irreducible if $A$ is. The converse can be proven in the same fashion, together with the fact that the closure of the union of two sets is the union of their closures.
    
\end{proof}

\begin{defn}\label{Definition 8.3}

In a topological space $(X, \tau)$, if a closed set $F\subseteq X$ satisfies that $F=\overline{\{x\}}$ for some $x\in F$, we call $x$ a \textbf{generic point} of $F$. We call $X$ \textbf{quasi-sober} if every irreducible closed subset has a generic point. We call $X$ \textbf{sober} if every irreducible closed subset has a unique generic point. By \textbf{Lemma \ref{Lemma 1.6}}, $X$ is being sober is equivalent to being quasi-sober and $T_0$.
    
\end{defn}

\begin{exmp}

Recall that we introduce the space $(\mathbb{R}, \tau_{\rightarrow})$ the real line equipped with the right-order topology (see \textbf{Example \ref{Example 7.5}}). In this case a closed subset is necessarily of the form $(-\infty, a]$ for some $a\in \mathbb{R}$, and hence $(\mathbb{R}, \tau_{\rightarrow})$ cannot be the union of two closed subsets, or is irreducible. However, $\mathbb{R}$ does not have a generic point, and hence $(\mathbb{R}, \tau_{\rightarrow})$ is not quasi-sober.
    
\end{exmp}

\noindent
Note that any closed set that has a generic point is irreducible according to the fact that a singleton is always irreducible and \textbf{Proposition \ref{Proposition 8.2}}. One can immediately see that a topological space $(X, \tau)$ is sober if and only if it is $T_0$ and quasi-sober. Hence, $X$ is quasi-sober if and only if $X_0$ the $T_0$-quotient is sober. It is also clear that if $X$ is Hausdorff, then $X$ is sober since any closed subset having at least two points is not irreducible (by \textbf{Proposition \ref{Proposition 8.2}(c)}). Next we will introduce more properties of irreducibility, (quasi)-sobriety and compare it with other topological properties introduced in previous sections.

\begin{prop}

A topological space $(X,\tau)$ is irreducible if and only if it is both connected and extremally disconnected.
    
\end{prop}

\begin{proof}

First we assume $X$ is irreducible. Then by \textbf{Proposition \ref{Proposition 8.2}(b)} $X$ is connected. By \textbf{Proposition \ref{Proposition 8.2}(d)} any open subset is dense in $X$ and hence the closure of any open subset is also open, which shows that $X$ is extremally disconnected. Conversely, suppose that $X$ is both connected and extremally disconnected. Assume by contradiction that there exists a non-empty open subset $O\subseteq X$ such that $\overline{O}\neq X$. Let $V=X\backslash \overline{O}$. Then by \textbf{Proposition \ref{Proposition 4.17}}, $\overline{O}$ and $\overline{V}$ are disjoint clopen subsets of $X$ such that $X=\overline{O}\cup \overline{V}$. This contradicts the connectedness of $X$.
    
\end{proof}

\begin{prop}

Suppose $(X, \tau)$ is the product space of $\big( X_{\alpha}, \tau_{\alpha} \big)_{\alpha\in \mathcal{A}}$ and is equipped with the product topology.

\begin{enumerate}
    \item if each $X_{\alpha}$ is irreducible, then so is $X$.
    \item if $F\subseteq X$ is a closed irreducible subset, then $F=\prod_{\alpha \in\mathcal{A}} F_{\alpha}$ where each $F_{\alpha}$ is a closed irreducible subset of $X_{\alpha}$.
    \item if each $X_{\alpha}$ is quasi-sober (sober resp.), then so is $X$.
\end{enumerate}
    
\end{prop}

\begin{proof}

\begin{enumerate}[label=\alph*)]

    \item Fix two open proper subsets $O_1, O_2\subseteq X$. Then there exists two different finite subsets $F_1, F_2\subseteq \mathcal{A}$ such that:

    $$
    O(F_1) = \left( \prod_{\alpha\in F_1}U_{\alpha} \right) \times \left( \prod_{\alpha\in \mathcal{A} \backslash F_1} X_{\alpha} \right) \subseteq O_1, \hspace{1cm} O(F_2) = \left( \prod_{\beta\in F_2}V_{\beta} \right) \times \left( \prod_{\beta\in \mathcal{A} \backslash F_1} X_{\beta} \right) \subseteq O_2
    $$
    for some open sets $U_{\alpha}\in\tau_{\alpha}$ for each $\alpha\in F_1$ and $V_{\beta}\in \tau_{\beta}$ for each $\beta\in F_2$. By assumption, if there exists $\alpha\in F_1\cap F_2$, then $U_{\alpha}\cap V_{\alpha} \neq \emptyset$ by \textbf{Proposition \ref{Proposition 8.2}}. Hence $O(F_1)\cap O(F_2)\neq \emptyset$, or $O_1\cap O_2\neq\emptyset$. We then have $X$ is irreducible again by \textbf{Proposition \ref{Proposition 8.2}}.

    \item For each $\alpha\in\mathcal{A}$, take $F_{\alpha} = \overline{\pi_{\alpha}(F)}$. Since each projection $\pi_{\alpha}:X\rightarrow X_{\alpha}$ is continuous, clearly each $F_{\alpha}$ is irreducible together with \textbf{Proposition \ref{Proposition 8.2}}. Then we have $F\subseteq \prod_{\alpha\in \mathcal{A}} F_{\alpha}$. Next fix $x\in X\backslash F$. Then there exists $U\in \tau_{\alpha}$ such that $\pi_{\alpha}^{-1}(U) \subseteq F^c$. This implies $\pi_{\alpha}(x)\notin F_{\alpha}$, or $x\notin \prod_{\alpha\in \mathcal{A}} F_{\alpha}$.

    \item When each $X_{\alpha}$ is quasi-sober, by $(b)$, given a closed irreducible subset $F\subseteq X$, we will have $F=\prod_{\alpha\in \mathcal{A}} \overline{\{x_{\alpha}\}}$ for some $x_{\alpha}\in X_{\alpha}$. Let $x\in X$ be such that $\pi_{\alpha}(x) = x_{\alpha}$. We then have $F=\overline{\{x\}}$. Indeed, in $X$, $x\preccurlyeq y$ with respect to the product topology $\tau$ precisely when $\pi_{\alpha}(x) \preccurlyeq \pi_{\alpha}(y)$ with respect to $\tau_{\alpha}$ for all $\alpha\in \mathcal{A}$. We can now conclude that when each $X_{\alpha}$ is (quasi-)sober, so is $X$.
    
\end{enumerate}
    
\end{proof}

\begin{prop}

An $R_1$ topological space is quasi-sober.
    
\end{prop}

\begin{proof}

Fix an $R_1$ topological space $(X, \tau)$ and suppose $F\subseteq X$ is an irreducible closed subset. Pick $x\in F$. If $F\neq\overline{\{x\}}$, then any $y\in F\backslash \overline{\{x\}}$ satisfies $y\nsim x$. Since $X$ is $R_1$, there exists disjoint open subsets $U\in\mathcal{U}(x)$ and $V\in \mathcal{U}(y)$. Then $F\cap U$ and $F\cap V$ are disjoint subsets relatively open in $F$, contradicting irreducibility according to \textbf{Proposition \ref{Proposition 8.2}}. 
    
\end{proof}

\noindent
We will end this section by studying \textbf{hereditary sobriety}, and show that this is a property strictly stronger than $T_D$ but also weaker than $T_1$. The proof that shows hereditary sobriety is equivalent to being $T_D$ and sober is originally from \cite{24}, which also reveals the surprising connection between quasi-sober subspaces and Skula topology. 

\begin{lem}\label{Lemma 8.8}

Given a topological space $(X, \tau)$ and $A\subseteq X$, recall that we use $\operatorname{cl}_{\operatorname{Sk}(\tau)} (A)$ to denote the closure of $A$ with respect to $\operatorname{Sk}(\tau)$ (see \textbf{Definition \ref{Definition 6.4}}). We use $\sim_{\tau}$ ($\sim_{\operatorname{Sk}(\tau)}$ resp.) to denote the equivalence relation associated with the specialization preorder with respect to $\tau$ ($\operatorname{Sk}(\tau)$ resp.). Then:

\begin{enumerate}[label = (\alph*)]

    \item $\operatorname{cl}_{\operatorname{Sk}(\tau)}(A)$ is $\sim_{\tau}$-saturated. Hence $[A]_{\sim_{\tau}} \subseteq \operatorname{cl}_{\operatorname{Sk}(\tau)}(A)$.

    \item if $A$ is quasi-sober, then $\operatorname{cl}_{\operatorname{Sk}(\tau)}(A) = [A]_{\sim}$. The converse holds if $X$ is quasi-sober.
    
\end{enumerate}
    
\end{lem}

\begin{proof}

\begin{enumerate}[label = (\alph*)]

    \item Note that a locally closed set is $\sim_{\tau}$-saturated since an open or closed set always is. Then $X\backslash \operatorname{cl}_{\operatorname{Sk}(\tau)}(A)$, as the interior of $A^c$ with respect to $\operatorname{Sk}(\tau)$, is a union of locally closed subsets and hence is $\sim_{\tau}$-saturated. This implies $\operatorname{cl}_{\operatorname{Sk}(\tau)}(A)$ is also $\sim_{\tau}$-saturated. Since $A\subseteq \operatorname{cl}_{\operatorname{Sk}(\tau)}(A)$, $[A]_{\sim_{\tau}} \subseteq \operatorname{cl}_{\operatorname{Sk}(\tau)} (A)$.

    \item First suppose $A$ is quasi-sober. Fix $x\in \operatorname{cl}_{\operatorname{Sk}(\tau)}$. By \textbf{Proposition \ref{Proposition 6.5}(b)}, $x\in \overline{A\cap \overline{\{x\}}}$. We claim that $A\cap \overline{\{x\}}$ is irreducible. Let $E_1, E_2\subseteq X$ be two closed subsets such that $A\cap \overline{\{x\}} \subseteq E_1\cup E_2$. Assume $x\in E_1$. Then $A\cap \overline{\{x\}} \subseteq \overline{\{x\}} \subseteq E_1$, and hence $A\cap \overline{\{x\}}$ is irreducible by \textbf{Proposition \ref{Proposition 8.2}}. Since $A$ is quasi-sober, there exists $y\in A$ such that $A\cap \overline{\{x\}} = A\cap \overline{\{y\}}$. This shows:

    $$
    y\in \overline{A\cap \overline{\{y\}}} = \overline{A\cap \overline{\{x\}}} \subseteq \overline{\{x\}} \hspace{1cm} x\in \overline{A\cap \overline{\{x\}}} = \overline{A\cap \overline{\{y\}}} \subseteq \overline{\{y\}}
    $$
    which implies $x\sim_{\tau} y$. This shows $\operatorname{cl}_{\operatorname{Sk}(\tau)}(A) \subseteq [A]_{\sim_{\tau}}$. The equality holds by (a). Conversely, assume $X$ is quasi-sober and $\operatorname{cl}_{\operatorname{Sk(\tau)}}(A) = [A]_{\sim_{\tau}}$. Suppose $E\subseteq A$ is relatively closed and irreducible. By \textbf{Proposition \ref{Proposition 8.2}}, $\overline{E}$ the closure with respect to $\tau$ is also irreducible. By assumption, there exists $e\in X$ such that $\overline{E} = \overline{\{e\}}$. Then $E=A\cap \overline{E} = A\cap \overline{\{e\}}$. Since $e\in \overline{A\cap \overline{\{e\}}}$, $e\in \operatorname{cl}_{\operatorname{Sk}(\tau)} (A)$. By assumption there exists $a\in A$ such that $e\sim_{\tau}a$, and hence $E=A\cap\overline{\{a\}}$, which is the closure of $\{y\}$ with respect to $\tau\Big|_A$ the relative topology. We can now conclude that $A$ is quasi-sober.
    
\end{enumerate}
    
\end{proof}

\begin{theorem}

In the set-up of \textbf{Lemma \ref{Lemma 8.8}}, $(X, \tau)$ is hereditarily sober if and only if it is sober and $T_D$.
    
\end{theorem}

\begin{proof}

According to \textbf{Definition \ref{Definition 7.7}} and \textbf{Definition \ref{Definition 8.3}}, $X$ being $T_D$ is equivalent to being $T_0$ and $R_d$, and $X$ being sober is equivalent to being quasi-sober and $T_0$. It suffices to show that $X$ is hereditarily quasi-sober if and only if it is quasi-sober and $R_d$. Then the desired equivalence follows by imposing $T_0$ on both sides.\\

\noindent
By \textbf{Proposition \ref{Proposition 7.11}}, we have $\operatorname{Sk}(\tau) \subseteq \Sigma(\sim_{\tau})$ the family of all $\sim_{\tau}$-saturated sets. If $X$ is hereditarily quasi-sober, by \textbf{Lemma \ref{Lemma 8.8}}, all $\sim_{\tau}$-saturated sets are closed with respect to $\operatorname{Sk}(\tau)$. Since $\Sigma(\sim_{\tau})$ is closed under complement, we then have $\Sigma(\sim_{\tau}) \subseteq \operatorname{Sk}(\tau)$. We then have $X$ is $R_d$ according to \textbf{Proposition \ref{Proposition 7.10}}. Conversely, again by \textbf{Proposition \ref{Proposition 7.10}}, for any $A\subseteq X$, we have $\operatorname{cl}_{\operatorname{Sk}(\tau)}(A) = [A]_{\sim_{\tau}}$. Then $A$ is quasi-sober by \textbf{Lemma \ref{Lemma 8.8}}.
    
\end{proof}

\section{Sobrification}

Given $(X, \tau)$ a topological space, let $X^s$ be the set of irreducible closed subsets of $X$. We define a topology on $X^s$ as the following: for each $U\in\tau$, define:

$$
U_s = \big\{C\in X^s: C\cap U\neq\emptyset \big\}
$$
Note that for any $(U_i)_{i\in I} \subseteq \tau$ and $U, V\in \tau$, we have:

\begin{equation}\label{e12}
\begin{aligned}
& X_s = X^s, \hspace{0.5cm} \emptyset_s = \emptyset\\
& \left( \bigcup_{i\in I}U_i \right)_s = \bigcup_{i\in I}(U_i)_s\\
& \big( U\cap V \big)_s = U_s\cap V_s
\end{aligned}
\end{equation}
With properties listed in (\ref{e12}), we can conclude that the family $\tau_s = \big\{ U_s: U\in\tau \big\}$ is a topology on $X^s$.

\begin{defn}

Given a topological space $(X, \tau)$, the space $(X^s, \tau_s)$ is the \textbf{sobrification} of $X$. The canonical embedding from $X$ to $X^s$ is defined below:

\begin{equation}\label{e13}
\iota_X:X\rightarrow X^s, \hspace{0.3cm} x\mapsto \overline{\{x\}}
\end{equation}
With $(X, \tau)$ given, we use $\preccurlyeq_X$ to denote the specialization preorder defined on $(X, \tau)$ and $\sim_{\tau}$ to denote its associated equivalence relation; we use $\preccurlyeq_{X^s}$ to denote the specialization preorder defined in $(X^s, \tau_s)$ and $\sim_{\tau_s}$ to denote its associated equivalence relation.
    
\end{defn}

\begin{prop}\label{Proposition 9.2}

Given a topological space $(X, \tau)$, $(X^s, \tau_s)$ is sober and, for any $C_1, C_2\in X^s$, $C_1\preccurlyeq_{X^s} C_2$ if and only if $C_1\subseteq C_2$.
    
\end{prop}

\begin{proof}

We first prove $\preccurlyeq_{X^s}$ coincides with the inclusion. From this it follows that $(X^s, \tau_s)$ is $T_0$. Let $C_1, C_2\in X^s$. If $C_1\subseteq C_2$, then for any $U\in\tau$ such that $C_1\in U_s$, we have $\emptyset \neq C_1\cap U \subseteq C_2\cap U$, or $C_2\in U_s$. This implies $C_1\preccurlyeq_{X^s} C_2$. Conversely, if $C_1\preccurlyeq_{X^s} C_2$, let $V=X\backslash C_2$. We then have $C_2\notin V_s$ and hence $C_1\notin V_s$. Hence $C_1\cap V=\emptyset$, or $C_1\subseteq C_2$.\\

\noindent
It remains to show $(X^s, \tau_s)$ is quasi-sober. Let $E\subseteq X^s$ be an irreducible closed subset. Then $E^c = U_s$ for some open proper subset $U\subseteq X$. We claim $F=U^c$ is irreducible. Suppose $V_1, V_2\in \tau$ satisfy $V_1\cap V_2\subseteq F^c$. Then by (\ref{e12}), we have $(V_1)_s \cap (V_2)_s\subseteq U_s = E^c$. Since $E$ is irreducible, by \textbf{Proposition \ref{Proposition 8.2}} we can assume $(V_1)_s\subseteq E^c$. Let $\iota_X$ be defined in (\ref{e13}). Pick $x\in V_1$. Then $\iota_X(x)\in (V_1)_s$, or $\iota_X(x)\cap U\neq\emptyset$. This implies $V_1\subseteq U = F^c$, proving the irreducibility of $F$. Now $F\in X^s$ and $F\cap U=\emptyset$. Then $F\in E$, or $\overline{\{F\}}$, the closure with respect to $\tau_s$ of singleton $\{F\}$ in $X^s$, is contained in $E$. For any $C\in E$, we have $C\cap U = \emptyset$ and hence $C\subseteq F$. By what we just showed, $C\preccurlyeq_{X^s} F$ for all $C\in E$, which implies $E\subseteq \overline{\{F\}}$. We can now conclude $E=\overline{\{F\}}$.
    
\end{proof}

\begin{prop}

Given a topological space $(X, \tau)$, let $\iota_X$ be defined in (\ref{e13}). Then for any $A\subseteq X$ and $E\subseteq X^s$:

$$
\begin{aligned}
& \overline{E} = X^s\backslash \left( X\backslash \overline{\bigcup E} \right)_s = \left\{C\in X^s:C\subseteq \overline{\bigcup E} \right\}\\
& \overline{\iota_X(A)} = \big\{C\in X^s: C\subseteq \overline{A} \big\}
\end{aligned}
$$
    
\end{prop}

\begin{proof}

We will start from proving the first equation. Note that the second equality in the first equation is immediate. It remains to show that $\overline{E}$ is equal to the third set. Fix $C\in X^s$ such that $C\backslash \overline{\bigcup E} \neq \emptyset$. Let $x\in C\backslash \overline{\bigcup E}$ and $U\in\mathcal{U}(x)$ such that $U\cap \overline{\bigcup E} = \emptyset$. We then have $C\in U_s$ and for each $F\in E$, $F\cap U=\emptyset$. If there exists $F'\in \overline{E}\cap U_s$, then $U_s\in \mathcal{U}(F')$, and by definition $U_s\in E\neq\emptyset$, which contradicts what we just showed. Hence $U_s\cap \overline{E}=\emptyset$. Conversely, let $C'\in X^s \backslash \overline{E}$. Then let $V\in\tau$ be such that $C'\in V_s$ and $V_s\cap \overline{E} = \emptyset$. This shows $V_s\cap F=\emptyset$ for all $F\in E$ and hence $F\cap \overline{\bigcup E}=\emptyset$. Since $C'\in V_s$, $C'\backslash \overline{\bigcup E} \neq \emptyset$.\\

\noindent
For the second equation, by part (a), we have:

$$
\overline{\iota_X(A)} = \left\{ C\in X^s: C\subseteq \overline{\bigcup \iota_X(A)} \right\} = \left\{C\in X^s: C\subseteq \overline{\bigcup_{a\in A} \overline{\{a\}}} \right\} = \left\{C\in X^s: C\subseteq \overline{A} \right\}
$$
    
\end{proof}

\begin{rem}\label{Remark 9.4}

Let $\iota_X$ be given in (\ref{e13}). Observe that for any $U\in \tau$ and $x\in X$:

$$
\iota_X(x)\in U_s \hspace{0.5cm} \Longleftrightarrow \hspace{0.5cm} \overline{\{x\}}\cap U\neq\emptyset \hspace{0.5cm} \Longleftrightarrow \hspace{0.5cm} x\in U
$$
Therefore for any $U\in\tau$, $\iota_X^{-1}(U_s)=U$ and $\iota_X(U) = U_s\cap \iota_X(X)$. This tells $\iota_X$ is a continuous and open mapping onto its range. Also, the assignment $U\mapsto U_s$ is an injective mapping from $\tau$ to $\tau_s$. There is an obvious relation between $\iota_X$ and $q:X\rightarrow X_0$ the $T_0$-quotient mapping. Note for any $x, y\in X$, by \textbf{Proposition \ref{Proposition 1.2}}, $x\preccurlyeq_X y$ if and only if $\iota_X(x) \subseteq \iota_X(y)$. It follows that $q(x) = q(y)$ if and only if $\iota_X(x) = \iota_X(y)$. Hence the unique injective mapping $j:X_0\rightarrow X^s$ defined by $j\circ q = \iota_X$ is a homeomorphism. As a result, we may naturally identify $\iota_X(X)$ with $X_0$ and the range restriction $X\rightarrow \iota_X(X)$ of $\iota_X$ with $q$.

\end{rem}

\begin{prop}\label{Proposition 9.5}

Given a topological space $(X, \tau)$, let $\iota_X:X\rightarrow X^s$ be defined as in (\ref{e13}). Then:

\begin{enumerate}[label = (\alph*)]

    \item $\iota_X$ is injective if and only if $X$ is $T_0$.
    \item $\iota_X$ is surjective if and only if $X$ is quasi-sober.
    \item $\iota_X$ is bijective if and only if $X$ is sober. In this case $\iota_X$ is a homeomorphism.
    
\end{enumerate}
    
\end{prop}

\begin{proof}

The first statement follows by \textbf{Lemma \ref{Lemma 1.6}}. The second statement follows by the definition of quasi-sobriety. The third definition follows by the second statement and the fact that sobriety is equivalent to quasi-sobriety and $T_0$ (which is, again, thanks to \textbf{Lemma \ref{Lemma 1.6}}).
    
\end{proof}

\noindent
As we point out the relation between $\iota_X$ and $q$, together with the universal property of $q$ (proven in \textbf{Theorem \ref{Theorem 2.4}}, we could expect and next prove the universal property of $\iota_X$ in the next result.

\begin{theorem}\label{Theorem 9.6}

Given a continuous mapping $f:(X, \tau_X) \rightarrow (Y, \tau_Y)$ between topological spaces, there exists a unique continuous mapping $f^s: (X^s, (\tau_X)_s) \rightarrow (Y^s, (\tau_Y)_s )$ such that the following diagram commutes.

$$
\begin{tikzcd}
X^s \arrow{r}{f^s} & Y^s \\
X \arrow{u}{\iota_X}  \arrow{r}{f} & Y \arrow{u}{\iota_Y}
\end{tikzcd}
$$
    
\end{theorem}

\begin{proof}

There is a natural candidate of $f^s$. Clearly a continuous mapping preserve irreducibility by \textbf{Proposition \ref{Proposition 8.2}}. Hence for any $C\in X^s$, define $f^s(C) = \overline{f(C)}$. To check that $f^s$ is well-defined, note that for any $A\subseteq B\subseteq \overline{A}$, by \textbf{Proposition \ref{Proposition 8.2}}, $A$ is irreducible if and only if $B$ is. Hence $f^s$ is well-defined. To show that $f^s$ is continuous, fix $V\in \tau_Y$ and $C\in X^s$. Then observe that:

$$
\begin{aligned}
C\in (f^s)^{-1}(V_s)
& \hspace{0.5cm} \Longleftrightarrow \hspace{0.5cm}
f^s(C) = \overline{f(C)} \in V_s
\hspace{0.5cm} \Longleftrightarrow \hspace{0.5cm}
\overline{f(C)}\cap V\neq\emptyset\\
& \hspace{0.5cm} \Longleftrightarrow \hspace{0.5cm} f(C)\cap V\neq \emptyset
\hspace{0.5cm} \Longleftrightarrow \hspace{0.5cm} C\cap f^{-1}(V) \neq \emptyset\\
& \hspace{0.5cm} \Longleftrightarrow \hspace{0.5cm} C\in \left[ f^{-1}(V) \right]_s
\end{aligned}
$$
Hence $(f_s)^{-1}(V_s) = \left[ f^{-1}(V) \right]_s$. Since $f$ is continuous, for any $A\subseteq X$, we have $f\big( \overline{A} \big) \subseteq \overline{f(A)}$. Then for any $x\in X$:

$$
f^s\circ\iota_X(x) = \overline{f\big( \overline{\{x\}} \big)} = \overline{\{f(x)\}} = \iota_Y\circ f(x)
$$
As for the uniqueness of $f^s$, suppose $g_1, g_2: X^s\rightarrow Y^s$ are continuous mappings such that $g_k\circ \iota_X = \iota_Y\circ f$ for all $k=1, 2$. We will show that for any $C\in X^s$, $g_1(C) = g_2(C)$. By \textbf{Proposition \ref{Proposition 9.2}}, $Y^s$ is $T_0$. It suffices to show that for any $V\in \tau_Y$, $g_1(C) \in V_s$ if and only if $g_2(C)\in V_s$. If $g_1(C)\in V_s$, by continuous of $g_1$, there exists $U\in \tau_X$ such that $(g_1)^{-1}(V_s) = U_s$. Take $x\in C\cap U$. Then $\iota_X(x)\in U_s$, which yields $g_2\circ \iota_X(x) = \iota_Y\circ f(x) = g_1\circ\iota_X(x)\in V_s$. Since $\iota_X(x) \preccurlyeq_{X^s} C$ (by \textbf{Proposition \ref{Proposition 9.2}}, continuity of $g_2$ implies $g_2\circ \iota_X(x) \preccurlyeq_{Y^s} g_2(C)$. Hence $g_2(C)\in V_s$.
    
\end{proof}

\noindent
The next result \textbf{Theorem \ref{Theorem 9.7}} is an equivalent formulation of \textbf{Theorem \ref{Theorem 9.6}}. For a given topological space $(X, \tau)$, since $(X^s,\tau_s)$ is always sober by \textbf{Proposition \ref{Proposition 9.2}}, it is clear that one of \textbf{Theorem \ref{Theorem 9.6}} and \textbf{Theorem \ref{Theorem 9.7}} can deduce the other.

\begin{theorem}\label{Theorem 9.7}

Let $(X, \tau_X)$ be a topological space and $(Y, \tau_Y)$ be a sober topological space. Then for any continuous mapping $f:X\rightarrow Y$, there exists a unique $\overline{f}:X^s\rightarrow Y$ such that $\overline{f} \circ \iota_X = f$.
    
\end{theorem}

\begin{proof}

Since $Y$ is sober, $\iota_Y$ is homeomorphism according to \textbf{Proposition \ref{Proposition 9.5}}. Define $\overline{f} = \iota_Y^{-1} \circ f^s$ where $f^s$ is given by \textbf{Theorem \ref{Theorem 9.6}}. Then $\overline{f}$ is continuous and $\overline{f} \circ\iota_X = f$. If $h:X^s\rightarrow Y$ is another continuous mapping that satisfies $h\circ \iota_X = f$, we will have $(\iota_Y\circ h) \circ \iota_X = \iota_Y \circ f$. Again by \textbf{Theorem \ref{Theorem 9.6}}, $\iota_Y\circ h = f^s$ and hence $h=\overline{f}$.
    
\end{proof}

\begin{rem}

Let \textbf{Sob} be the full category of \textbf{Top} consisting of sober spaces. \textbf{Theorem \ref{Theorem 9.6}} shows that there is a functor \textbf{S}: \textbf{Top} $\rightarrow$ \textbf{Sob} which sends a topological space $X$ to its sobrification $X^s$, and a continuous mapping $f:X\rightarrow Y$ to $f^s: X^s \rightarrow Y^s$. If we let \textbf{U}: \textbf{Sob} $\rightarrow$ \textbf{Top} denote the inclusion functor, the above shows the following bijection:

$$
\operatorname{Hom}_{\textbf{Sob}} \big( \textbf{S}(X), Y \big) \cong \operatorname{Hom}_{\textbf{Top}} \big( X, \textbf{U}(Y) \big) 
$$
Using notations in \textbf{Theorem \ref{Theorem 9.6}}, more explicitly, for any $g\in \textbf{S}(X) = X^s\rightarrow Y$, define $g^{\flat}: X\rightarrow \textbf{U}(Y)$ by $g^{\flat} = g\circ \iota_X$. Given $f: X\rightarrow \textbf{U}(Y)$, define $f^{\sharp}:\textbf{S}(X) \rightarrow Y$ by $f^{\sharp} = \iota_Y^{-1} \circ f^s$. Then $g\mapsto g^{\flat}$ and $f\mapsto f^{\sharp}$ are mutual inverse. Therefore \textbf{S} is left adjoint to \textbf{U}, and hence \textbf{Sob} is a reflective subcategory of \textbf{Top}.
    
\end{rem}

\section{Locally Hausdorff spaces}

This is section is the last one for the set-up we need to introduce more of separation axioms between $T_1$ and Hausdorff. In \textbf{Proposition \ref{Proposition 4.5}} or \textbf{Proposition \ref{Proposition 4.6}}, we describe Hausdorff as a $T_0$-version of $R_1$, and has not discovered properties between $T_1$ and Hausdorff. We will start by the property \textbf{locally Hausdorff}, which is strictly stronger than $T_1$ but weaker than Hausdorff, and compare it with sobriety.

\begin{defn}

A topological space $(X, \tau)$ is \textbf{locally Hausdorff} if it has an open cover that consists of Hausdorff subspaces (equipped with relative topology). Equivalently, $(X, \tau)$ is locally Hausdorff if and only if each point has an open Hausdorff neighborhood (equipped with relative topology).
    
\end{defn}

\noindent
Clearly a Hausdorff space is locally Hausdorff. As a result, a locally Hausdorff space is $T_1$. We point out it is superfluous to define ``locally $T_0$" (resp. $T_D$, $T_1$) in a similar fashion, as such a space is automatically $T_0$ (resp. $T_D$, $T_1$). In addition, we will later show that if every point in a topological space has a closed Hausdorff neighborhood, then that space is Hausdorff. It is also easy to see that being locally Hausdorff is hereditary. Also, if $X$ has open cover that consists of locally Hausdorff subspaces, then $X$ is also locally Hausdorff. We shall apply these facts in results below.

\begin{prop}\label{Proposition 10.2}

A locally Hausdorff topological space $(X, \tau)$ is sober.
    
\end{prop}

\begin{proof}

Let $A\subseteq X$ be irreducible. We will show that $A$ is a singleton. Fix $x\in A$. Let $H_x\in \mathcal{U}(x)$ be an open Hausdorff neighborhood. Suppose there exists $y\in (A\cap H_x)\backslash \{x\}$. Then there exists $U\in \mathcal{U}(x)$, $V\in\mathcal{U}(y)$ such that $(V\cap U)\cap H_x=\emptyset$. We then have $V\cap H_x\cap A$ and $U\cap H_x\cap A$ are disjoint subsets that are open with respect to $\tau\Big|_A$, which contradicts the irreducibility of $A$ by \textbf{Proposition \ref{Proposition 8.2}}. Hence $A\cap H_x = \{x\}$, which means $\{x\}\in \tau\Big|_A$. If there exists $y\in A\backslash \{x\}$, then by the same token, we can show that $\{y\}\in \tau\Big|_A$, which again contradicts the irreducibility of $A$. Hence $A=\{x\}$.
    
\end{proof}

\noindent
Since a space $(X, \tau)$ is Hausdorff if and only if the diagonal line is closed with respect to the product topology, we can describe locally Hausdorff similarly, not using the diagonal line but $\sim_{\tau}$ as a subset of $X\times X$. The analogous result also holds for locally $R_1$ space since, according to \textbf{Proposition \ref{Proposition 4.5}}, Hausdorff can be viewed as a $T_0$-version of $R_1$.

\begin{prop}

Given a topological space $(X, \tau)$, without confusion we can view $\sim_{\tau}$ as a subset of $X\times X$. Then the following are equivalent:

\begin{enumerate}[label = (\alph*)]

    \item $X$ is locally $R_1$.
    \item $\sim_{\tau}$ is locally closed in $X\times X$.
    \item given another topological space $(Z, \tau_Z)$ and a continuous mapping $f:(Z, \tau_Z) \rightarrow (X, \tau)$, the following set:

    $$
    G = \big\{ (z, x)\in Z\times X: f(z)\sim_{\tau}x \big\}
    $$
    is locally closed (with respect to product topology) in $Z\times X$.
    
\end{enumerate}

\noindent
As a result, the analogous result holds for locally Hausdorff spaces if we, in $(b)$, replace $\sim_{\tau}$ by $\Delta_X$ the diagonal line and, in $(c)$, replace $G$ by the graph of $f$.
    
\end{prop}

\begin{proof}

We will first show $(b)\implies (a)$. Suppose $\sim_{\tau}$ is locally closed. Fix $x\in X$. By \textbf{Proposition \ref{Proposition 6.3}}, since $\sim_{\tau}$ is locally closed at $(x, x)$, there exists $O\in \mathcal{U}\big( (x, x) \big)$ such that $O\backslash \sim_{\tau}$ is relatively open in $O$. Note that for all small enough $U\in \mathcal{U}(x)$, we have $U\times U\subseteq O$. Since $U\times U$ is open in $X$, we have:

$$
(U\times U)\cap \big( O\backslash \sim_{\tau} \big) = (U\times U)\backslash \sim_{\tau}
$$
is relatively open in $O$. This shows that $U$ is an $R_1$ open neighborhood of $x$. We can immediate see $(c)\implies (b)$ when we set $(Z, \tau_Z) = (X, \tau)$ and $f=\operatorname{id}_X$. It remains to show $(a)\implies (c)$. When $(a)$ is true, fix $(z, x)\in G$. Let $V\in \mathcal{U}(x)$ be a open neighborhood that is $R_1$. Then one can observe that $f^{-1}(V)\times V$ is an open neighborhood of $(z, x)$ as $f(z)\in V$. Then $\big( f^{-1}(V) \times V \big) \backslash G$ is relatively open in $f^{-1}(V) \times V$ since $V$ is $R_1$. This shows that $G$ is locally closed at $(z, x)$. As $(z, x)\in G$ is arbitrarily picked, we can conclude that $G$ is locally closed by \textbf{Proposition \ref{Proposition 6.3}}.\\

\noindent
To show the analogous characterizations for locally Hausdorff space, according to \textbf{Proposition \ref{Proposition 4.5}}, it suffices to further assume $X$ is $T_0$ (also this is because we pointed out that it is superfluous to define ``locally $T_0$"). After assuming $X$ is $T_0$, then $\sim_{\tau}$ coincides with the identity relation. In this case, $\sim_{\tau}=\Delta_X$ and $G$ is the graph of $f$.
    
\end{proof}

\noindent
Based on the principle advocated in \textbf{Section 2}, we call a topological space \textbf{locally $R_1$} if its $T_0$-quotient is locally Hausdorff. It is routine to check a topological space is locally $R_1$ precisely when it has an open cover that consists of $R_1$ subspaces. It follows that a locally $R_1$ space is quasi-sober and $R_0$.\\

\noindent
We now present a general method of constructing locally Hausdorff spaces in a similar vein as the formulation of manifolds. This so-called \textbf{gluing construction}, which can be applied in other studies, may indeed have its origins in algebraic geometry (see {\cite[p.33]{1}}). The underlying idea is that is $X$ is a union of open subspaces $(X_{\alpha})_{\alpha\in \mathcal{A}}$. Then for any $\alpha, \beta\in \mathcal{A}$, define:

$$
U_{\alpha, \beta} = X_{\alpha}\cap X_{\beta}
$$
Then for any $\alpha, \beta, \gamma\in\mathcal{A}$, observe that:

$$
U_{\alpha, \alpha} = X_{\alpha} \hspace{1cm} U_{\alpha, \beta} = U_{\beta, \alpha} \hspace{1cm} U_{\alpha, \beta}\cap U_{\alpha, \gamma} = U_{\beta, \alpha} \cap U_{\beta, \gamma}
$$
It turns out that this observation can be reversed. A set of \textbf{gluing data} is a triple:

\begin{equation}\label{e14}
\Big( \big\{ (Y_{\alpha}, \tau_{\alpha}) \big\}_{\alpha\in \mathcal{A}} \,,\, \big\{ U_{\alpha, \beta} \big\}_{\alpha, \beta \in \mathcal{A}} \,,\, \big\{ h_{\alpha, \beta} \big\}_{\alpha, \beta \in \mathcal{A}} \Big)
\end{equation}
where for each $\alpha, \beta\in \mathcal{A}$, $(Y_{\alpha}, \tau_{\alpha})$ is a topological space, $U_{\alpha, \beta}$ is a (possibly empty) open set in $Y_{\alpha}$, and $h_{\alpha, \beta}: U_{\alpha, \beta} \rightarrow U_{\beta, \alpha}$ is a homeomorphism. In addition, the triple in (\ref{e14}) satisfies that for any $\alpha, \beta, \gamma\in \mathcal{A}$:

\begin{equation}\label{e15}
\begin{aligned}
& U_{\alpha, \alpha} = Y_{\alpha}, \hspace{1cm} h_{\alpha, \alpha} = \operatorname{id}_{Y_{\alpha}}\\
& h_{\alpha, \beta}\left( U_{\alpha, \beta}\cap U_{\alpha, \gamma} \right) = U_{\beta, \alpha} \cap U_{\beta, \gamma}\\
& \left( h_{\beta, \gamma}\big|_{U_{\beta, \alpha} \cap U_{\beta, \gamma}} \right) \circ \left( h_{\alpha, \beta}\big|_{U_{\alpha, \beta}\cap U_{\alpha, \gamma}} \right) = h_{\alpha, \gamma}\big|_{U_{\gamma, \alpha} \cap U_{\gamma, \beta}}
\end{aligned}
\end{equation}
As an immediate consequence of (\ref{e15}), we have for any $\alpha, \beta \in\mathcal{A}$:

\begin{equation}\label{e16}
h_{\beta, \alpha} = h_{\alpha, \beta}^{-1}
\end{equation}
Our intention is to form a space $X$ by patching up all $\big( Y_{\alpha}:\alpha \in \mathcal{A} \big)$ such that $U_{\alpha, \beta} \subseteq Y_{\alpha}$ is glued to $U_{\beta, \alpha} \subseteq Y_{\beta}$ according to $h_{\alpha, \beta}$. Let $Y$ be the sum of $(Y_{\alpha})_{\alpha\in \mathcal{A}}$ (see \textbf{Example \ref{Example 0.5}}). Recall that each inclusion mapping $j_{\alpha}: Y_{\alpha} \rightarrow Y$ is open and closed. We may identify each $Y_{\alpha}$ as $j_{\alpha}(Y_{\alpha})$, so that $(Y_{\alpha})_{\alpha\in \mathcal{A}}$ is an open cover of $Y$. Define the relation $R$ on $Y$ by the following: for any $x, y\in Y$:

$$
(x, y)\in R \hspace{0.5cm} \Longleftrightarrow \hspace{0.5cm} \exists\,\alpha, \beta\in\mathcal{A} \,\text{  such that  }\, x\in U_{\alpha, \beta},\, y\in U_{\beta, \alpha}\,\text{  and  }\,y=h_{\alpha, \beta}(x)
$$
By (\ref{e15}), we can see $R$ is an equivalence relation. Let $X=Y\slash R$ and let $p:Y\rightarrow X$ be the quotient mapping. By (\ref{e15}), for any $x, y\in Y_{\alpha}$, if $(x, y)\in R$, $x=y$. Hence $p_{\alpha} = p\circ j_{\alpha}$ is injective. It is worth noting that the quotient topology on $X$ coincides with the strong topology induced by $(p_{\alpha})_{\alpha\in \mathcal{A}}$.

\begin{lem}\label{Lemma 10.3}

In the set-up of the discussion above, the quotient mapping $p:Y\rightarrow X$ is open. Hence each $p_{\alpha}: Y_{\alpha} \rightarrow X$ is an open embedding.
    
\end{lem}

\begin{proof}

The second claim follows from the first since each $j_{\alpha}$ is open (see \textbf{Example \ref{Example 0.5}}). Fix an open set $W\subseteq Y$. As we have shown in \textbf{Example \ref{Example 0.5}}, $W=\bigcup_{\alpha\in \mathcal{A}} j_{\alpha}(W_{\alpha})$ where $W_{\alpha}\in \tau_{\alpha}$ for each $\alpha\in \mathcal{A}$. Fix $x\in p^{-1}\big( p(W) \big)$. Then there exists $y\in W$ such that $(x, y)\in W$. We then can find $\alpha, \beta\in \mathcal{A}$ such that $x\in Y_{\alpha}$ and $y\in W_{\beta}$. Since $y\in U_{\beta, \alpha}\cap W_{\beta}$, $x=h_{\beta, \alpha}(y)$. Since $h_{\beta, \alpha}$ is a homeomorphism, $V=h_{\beta, \alpha} \big( U_{\beta, \alpha}\cap W_{\beta} \big)$ is an open neighborhood of $x$ in $Y_{\alpha}$. By (\ref{e16}), a consequence of (\ref{e15}), we have $h_{\alpha, \beta}(V) \subseteq W_{\beta} \subseteq W$. It follows that $V\subseteq p^{-1}\big( p(W) \big)$. Hence $p^{-1}\big( p(W) \big)$ is an open neighborhood of $x$ in $Y$, which completes the proof.
    
\end{proof}

\noindent
As a result, $(X_{\alpha})_{\alpha\in \mathcal{A}}$ is an open cover of $X$ where for each $\alpha\in \mathcal{A}$, $X_{\alpha} = p(Y_{\alpha})$. For any $\alpha,\beta\in \mathcal{A}$, $X_{\alpha} \cap X_{\beta} = p(U_{\alpha, \beta})$. In addition, a base for the topology on $X$ can be obtained by combining each $\tau_{\alpha}$'s base. Namely, if, for each $\alpha\in \mathcal{A}$, $\mathcal{B}_{\alpha}$ is a base for $\tau_{\alpha}$, then:

$$
\mathcal{B} = \big\{ p_{\alpha}(B): \alpha\in\mathcal{A},\, B\in \mathcal{B}_{\alpha} \big\}
$$
is a base for $X$. Now the main result is within our reach.

\begin{prop}\label{Proposition 10.4}

In the set-up of the discussion above, $X$ is locally Hausdorff if and only if each $Y_{\alpha}$ is locally Hausdorff.
    
\end{prop}

\begin{proof}

Since being locally Hausdorff is hereditary, if $X$ is locally Hausdorff, then so is each $X_{\alpha}$. In this case, by \textbf{Lemma \ref{Lemma 10.3}}, each $Y_{\alpha}$ is locally Hausdorff. Conversely, again by \textbf{Lemma \ref{Lemma 10.3}}, each $X_{\alpha}$ is locally Hausdorff and hence $X$ as a union of open locally Hausdorff subspaces is locally Hausdorff.
    
\end{proof}

\noindent
Under this apparatus, we now obtain a special case where copies of a topological space $(Z, \tau_Z)$ are glued together along a fixed open subset $W\subseteq Z$. With the given index $\mathcal{A}$, set $Z_{\alpha} = Z$, $\tau_{\alpha} = \tau_Z$ for each $\alpha\in \mathcal{A}$. For any $\alpha, \beta\in \mathcal{A}$, define $U_{\alpha, \beta} = W$ whenever $\alpha \neq\beta$, and $U_{\alpha, \beta}=Y$ otherwise. For any $\alpha, \beta\in \mathcal{A}$, let $h_{\alpha, \beta}$ be the identity mapping of $W$ or $Z$. In this case, the sum of $(Z_{\alpha}, \tau_{\alpha})_{\alpha\in \mathcal{A}}$ is $\mathcal{A} \times Z$ and the associated equivalence relation $R_W$ is defined as follows: for any $(\alpha, z)$, $(\beta, z')\in \mathcal{A} \times Z$:

$$
\big( (\alpha, z), (\beta, z') \big)\in R_W \hspace{0.5cm} \Longleftrightarrow \hspace{0.5cm}
z=z', \,\text{  and either  }\, \alpha=\beta\,\text{  or  }\, z\in W
$$
In this case we use $X_Z$ to denote the quotient space $(\mathcal{A} \times Z)\slash R_W$, and $X_Z$, as a set, can be identified with:

$$
X_Z = W\cup \left( \bigcup_{\alpha\in \mathcal{A}} \{\alpha\} \times (Z\backslash W) \right)
$$
If, for each $\alpha\in \mathcal{A}$, we set $(X_Z)_{\alpha} = W\cup \big( \{\alpha\} \times (Z\backslash W) \big)$, then each $(X_Z)_{\alpha}$ is homeomorphic to $Z$ and hence $\big( (X_Z)_{\alpha} \big)_{\alpha\in \mathcal{A}}$ forms an open cover of $X$.

\begin{exmp}\label{Example 10.5}

With a fixed index set $\mathcal{A}$, we may duplicate the end point $0$ of the unit interval $Z=[0, 1]$ with the usual topology by gluing $\vert\, \mathcal{A} \,\vert$-many $Z$ together along with $W=(0, 1]$. The resulting $X_Z$ is then:

$$
X_Z = (0, 1]\cup \left( \bigcup_{\alpha\in \mathcal{A}} \{\alpha\} \times \{0\} \right)
$$
For each $\alpha\in\mathcal{A}$, we put $0_{\alpha} = \{\alpha\} \times \{0\}$. We call this $X_Z$ \textbf{the unit interval with} $\vert\,\mathcal{A} \,\vert$ \textbf{origins}. For any $\alpha\in \mathcal{A}$, $\delta\in (0, 1)$, use $[0, \delta)_{\alpha}$ to denote $\{0_{\alpha}\} \cup (0, \delta)$. For any $\alpha\in \mathcal{A}$, the family $\big( [0, \delta)_{\alpha}: \delta\in (0,1) \big)$ forms a neighborhood base of $0_{\alpha}$. If $\vert\, \mathcal{A} \,\vert > 1$, then for any two different $\alpha, \beta\in \mathcal{A}$ and $\delta, \epsilon\in (0, 1)$, we have $[0, \delta)_{\alpha} \cap [0, \epsilon)_{\beta} \neq\emptyset$. Thus $X_Z$ is not Hausdorff when $\vert\, \mathcal{A} \,\vert > 1$, but is locally Hausdorff by \textbf{Proposition \ref{Proposition 10.4}}. Also, one can check that $X_Z$ retain some properties of $[0, 1]$, such as first countability and locally compactness. Note that $X_Z$ is compact precisely when $\mathcal{A}$ is finite, and is second countable precisely when $\mathcal{A}$ is countable.
    
\end{exmp}

\begin{exmp}

This example is constructed the same way as \textbf{Example \ref{Example 10.5}} but obtained from the ordinal space $Z=\Omega = [0, \omega_1]$. Let $W=(0, \omega_1]$ and $\mathcal{A} = \{0, 1\}$. The resulting $X_Z$ can be written as $\Omega\cup \{\omega_1'\}$ where $\omega_1'$ is the duplicated $\omega_1$. Let $\Omega' = [0, \omega_1)\cup \{\omega_1'\}$. Then both $\Omega$ and $\Omega'$ are compact Hausdorff subspaces of $X_Z$. However, $\Omega\cap \Omega'$ is not compact (in fact, not even Lindelöf).
    
\end{exmp}

\noindent
We will close this section by showing that locally Hausdorff is not a productive property, while it is obvious that Hausdorff is.

\begin{prop}

Suppose $\big( X_{\alpha}, \tau_{\alpha} \big)_{\alpha\in \mathcal{A}}$ is an infinite family of non-Hausdorff spaces. Then $X=\prod_{\alpha\in \mathcal{A}} X_{\alpha}$, when equipped with the product topology $\tau$, is not locally Hausdorff.
    
\end{prop}

\begin{proof}

By assumption, for each $\alpha\in \mathcal{A}$, let $x_{\alpha}, y_{\alpha}\in X_{\alpha}$ be two different points that have no disjoint open neighborhoods. Define $x\in X$ by $\pi_{\alpha}(x) = x_{\alpha}$ for each $\alpha\in \mathcal{A}$. We claim that $x$ has no Hausdorff open neighborhood. Fix a finite subset $F\subseteq \mathcal{A}$ and define:

$$
B = \bigcap_{\alpha\in F}\pi_{\alpha}^{-1}(B_{\alpha})
$$
where $B_{\alpha}\in \tau_{\alpha}$ for all $\alpha\in F$. Since Hausdorff is a hereditary property, it suffices to show that all neighborhoods of $x$, which have the same form as $B$, are not Hausdorff. Fix $\beta\in \mathcal{A} \backslash F$. Define $y\in X$ by $\pi_{\alpha}(y) = x_{\alpha}$ whenever $\alpha\neq \beta$ and $\pi_{\beta}(y) = y_{\beta}$. Then $y\in B$. Let $C\in \mathcal{U}(x)$, $D\in \mathcal{U}(y)$ be such that $x\in C$ and $y\in D$. We can assume for some finite set $E\subseteq \mathcal{A}$ with $F\cup\{\beta\} \subseteq E$ such that:

$$
C = \bigcap_{\alpha\in E} \pi_{\alpha}^{-1}(U_{\alpha}) \hspace{1.5cm} D = \bigcap_{\alpha\in E} \pi_{\alpha}^{-1}(V_{\alpha})
$$
where $U_{\alpha}\in \mathcal{U} \big( \pi_{\alpha}(x) \big)$ and $V\in \mathcal{U}\big( \pi_{\alpha}(y) \big)$ for all $\alpha\in E$. For each $\alpha\in E\backslash \{\beta\}$, $x_{\alpha} \in U_{\alpha} \cap V_{\alpha}$. Also, $U_{\beta}$ and $V_{\beta}$ cannot separate $x_{\beta}$ and $y_{\beta}$. Hence $C\cap D\neq\emptyset$. This implies $B$ is not Hausdorff, and hence $x$ has no Hausdorff open neighborhoods.
    
\end{proof}

\section{Separation axioms between \texorpdfstring{$T_1$}{T1} and Hausdorff}

In this very last section, we will focus on introducing a list of properties that are strictly stronger than $T_1$ but strictly weaker than Hausdorff. We chose the following list of properties because each of them retain some properties of a Hausdorff space but, meanwhile, is more flexible in applications. We will also compare their strength. All implications are not reversible, and we will see some examples later. At the end, we will include diagrams that summarize the relation of all topological properties in this note.

\begin{defn}\label{Definition 11.1}

A topological space $(X, \tau)$ is called:

\begin{enumerate}[label = (\alph*)]

    \item \textbf{KC} if every compact subset is closed.
    \item $\textbf{KC}_{\omega}$ if every countable compact subset if closed.
    \item \textbf{SC} if, whenever $(s_n)_{n\in \mathbb{N}} \subseteq X$ is a sequence converging to $x\in X$, the set $\{s_n: n\in \mathbb{N}\} \cup \{x\}$ is closed.
    \item \textbf{SH} or \textbf{sequentially Hausdorff} if every sequence in $X$ has at most one limit.
    \item \textbf{WH} or \textbf{weakly Hausdorff} if the continuous image of any compact Hausdorff space into $X$ is closed.
    
\end{enumerate}
    
\end{defn}

\begin{prop}

In the set-up of \textbf{Definition \ref{Definition 11.1}}, the following implications hold:

\begin{enumerate}[label = (\alph*)]

    \item Hausdorff $\,\implies\,$ \textbf{KC} $\,\implies\,$ $\textbf{KC}_{\omega}$ $\,\implies\,$ \textbf{SC} $\,\implies\,$ \textbf{SH} $\,\implies\,T_1$. 

    \item \textbf{KC} $\,\implies\,$ \textbf{WH} $\,\implies\,$ \textbf{SC}.

    \item A first countable \textbf{SH} space is Hausdorff.
    
\end{enumerate}
    
\end{prop}

\begin{proof}

\begin{enumerate}[label = (\alph*)]

\item 

The implication Hausdorff $\,\implies\,$ \textbf{KC} $\,\implies\,$ $\textbf{KC}_{\omega}$ is immediate. To show $\textbf{KC}_{\omega}$, fix a convergent sequence $(s_n)_{n\in \mathbb{N}} \subseteq X$ converging to $x\in X$. Then the set $K = \{s_n:n\in \mathbb{N}\} \cup \{x\}$ is compact and hence closed by assumption. To show \textbf{SC} $\,\implies\,$ \textbf{SH}, first note that \textbf{SC} implies $T_1$ by considering constant sequences. Fix a sequence $(x_n)_{n\in \mathbb{N}} \subseteq X$ converging to $x\in X$ and fix $y\neq x$. Suppose for all $n\in\mathbb{N}$ with $n\geq m$, $x_n\in \{y\}^c$. By \textbf{SC}, the following set:

$$
\big( \{s_n: n\geq m\} \cup \{x\} \big)^c
$$
is an open neighborhood of $y$. Hence $(s_n)_{n\in \mathbb{N}}$ does not converge to $y$. To show \textbf{US} $\,\implies\,T_1$, fix two different points $x, y\in X$. \textbf{US} implies that the sequence that is constantly equal to $y$ does not converge to $x$, and hence $x$ has an open neighborhood that does not contain $y$. This shows that $X$ is $T_1$.

\item A \textbf{KC} space is \textbf{WH} since the continuous image of a compact set is always compact. If $(s_n)_{n\in \mathbb{N}} \subseteq X$ converges to $x\in X$, then the set $\{s_n:n\in \mathbb{N}\} \cup \{x\}$ is the continuous image of the ordinal space $[0, \omega]$, which is compact and Hausdorff. Hence a \textbf{WH} space is \textbf{SC}.

\item Suppose $X$ is first countable but non-Hausdorff. Let $x, y\in X$ be two different points with $U\cap V\neq\emptyset$ for all $U\in \mathcal{N}(x)$ and for all $V\in \mathcal{N}(y)$. Let $(B_n)_{n\in \mathbb{N}}$ ($(C_n)_{n\in \mathbb{N}}$ resp.) be a neighborhood base of $x$ ($y$ resp.), which is decreasing in inclusion. For each $n\in\mathbb{N}$, pick $s_n\in B_n\cap C_n$. Then $s_n\rightarrow x$ and $s_n\rightarrow y$.

\end{enumerate}
    
\end{proof}

\begin{prop}

Given a topological space $(X, \tau)$, the following are equivalent:

\begin{enumerate}[label = (\alph*)]

    \item $X$ is \textbf{SH}.
    \item the diagonal line $\Delta_X$ is sequentially closed. Namely, the set of limits of every convergent sequence in $\Delta_X$ is contained in $\Delta_X$.
    \item for any continuous mapping $f:(Y, \tau_Y) \rightarrow (X, \tau)$, the graph of $f$ is sequentially closed with respect to the product topology.
    
\end{enumerate}
    
\end{prop}

\begin{proof}

We will prove the implication $(a) \implies (c) \implies (b) \implies (a)$. When $(c)$ is true, to show $(b)$, note that $\Delta_X$ is the graph of the identity mapping. To show $(b) \implies (a)$, it suffices to show that every convergent sequence in $X$ has a unique limit. Suppose $(x_n)_{n\in \mathbb{N}} \subseteq X$ is a convergent sequence. If there exists two different $x, y\in X$ such that $x_n \rightarrow x$ and $x_n\rightarrow y$, we will have $(x_n, x_n)\rightarrow (x, y)\notin \Delta_X$ by definition of the product topology. When $(a)$ is true, suppose $\big( (y_n, f(y_n)) \big)_{n\in \mathbb{N}}$ is a convergent sequence in the graph of $f$. If $(y_n, f(y_n)) \rightarrow (y, y')\in Y\times X$, then we must have $f(y_n) \rightarrow f(y)$ by continuity, and $f(y)=y'$ by \textbf{SH}. Hence the graph of $f$ is sequentially closed.
    
\end{proof}

\begin{exmp}

Let $S$ be an infinite set equipped with the cofinite topology $\tau$. For any $x\in S$, $\{x\}^c$ is open, and hence this space is $T_1$. Since $S$ is infinite, there exists an infinite sequence of distinct elements $(s_n)_{n\in \mathbb{N}} \subseteq S$. Fix $x\in X$ and $U\in \mathcal{U}(x)$. Since $U^c$ is finite, $(s_n)_{n\in \mathbb{N}}$ will eventually be in $U$. This shows that $(s_n)_{n\in \mathbb{N}}$ converges to all points in $S$, and hence this space is not \textbf{SH}. Clearly $(S, \tau)$ is irreducible. Since $(S, \tau)$ is also $T_1$, $(S, \tau)$ is not quasi-sober and hence not locally Hausdorff by \textbf{Proposition \ref{Proposition 10.2}}.
     
\end{exmp}

\begin{exmp}

Consider $(X, \tau)$ an uncountable set equipped with the co-countable topology. Again for each $x\in X$, $\{x\}^c$ is open and hence $(X, \tau)$ is $T_1$. Given an infinite subset $S\subseteq X$, consider the open cover $\big( \{s\}^c:s\in S \big)$. Clearly this open cover has no subcover, and hence an infinite set cannot be compact with respect to $\tau$. This implies that $(X, \tau)$ is \textbf{KC}. Also, $(X, \tau)$ is irreducible, $T_1$, and hence not locally Hausdorff by \textbf{Proposition \ref{Proposition 10.2}} and not quasi-sober.
    
\end{exmp}

\begin{exmp}

Let $X=[0, 1]\cup\{\ast\}$ where $\ast\notin [0, 1]$, and let $\mathcal{C}$ denote the family of all closed countable subsets of $[0, 1]$. Let $\mathcal{F} = \big\{ X\backslash E:E\in \mathcal{C} \big\}$ and $\tau$ be the usual topology on $[0, 1]$. Define $\tau^* = \tau\cup \mathcal{F}$. For any $x\in [0, 1]$ there exists $U\in \mathcal{U}(x)$ such that $\ast\notin U$. This implies $\{\ast\}$ is closed. Also, it is clear that $\{x\}$ is closed for each $x\in X$. Hence $(X, \tau^*)$ is $T_1$. It is also clear that both $[0, 1]$ and $(X, \tau^*)$ are compact with respect to $\tau^*$. However, $[0,1]$ is not closed with respect to $\tau^*$ and hence $(X, \tau^*)$ is not \textbf{KC} but $\textbf{KC}_{\omega}$ by definition of $\mathcal{F}$.
    
\end{exmp}

\noindent
One of the main focus in general topology is the compact Hausdorff space. We will show next that the reason is not only because of compactness, but also because there is no strictly finer topology that makes the given space compact. Such property for a topology is called \textbf{maximal compact}. We will see that a topology that makes a given set compact and Hausdorff is maximal compact. However, the converse is not true. We will prove the relation between maximal compactness and \textbf{KC}.

\begin{defn}

Given a set $X$, a topology $\tau$ defined on $X$ is a \textbf{compact topology of} $X$ if $X$ is compact with respect to $\tau$, and is called \textbf{maximal compact} if $\tau$ is a compact topology of $X$ and $X$ has no other compact topologies that are strictly finer than $\tau$.
    
\end{defn}

\begin{prop}

Given a topological space $(X, \tau)$:

\begin{enumerate}[label = (\alph*)]

    \item if both $\tau$, $\tau'$ are compact topologies of $X$ such that $\tau \subseteq \tau'$ and $(X, \tau)$ is \textbf{KC}, then $\tau=\tau'$.

    \item for any $A\subseteq X$, define:

    $$
    \tau^A = \big\{ U\cup (V\cap A): U, V\in \tau \big\}
    $$
    Show that $\tau^A$ is the smallest topology that is finer than $\tau$ and contains $A$. We call $\tau^A$ the \textbf{simple extension of} $\tau$ by $A$.

    \item if $(X, \tau)$ is compact but not \textbf{KC}, then there exists $\tau'$ a compact topology of $X$ such that $\tau\subsetneq \tau'$.
    
\end{enumerate}

\noindent
As a result, we can conclude that $\tau$ is maximally compact if and only if it is a compact topology of $X$ and $(X, \tau)$ is \textbf{KC}.
    
\end{prop}

\begin{proof}

\begin{enumerate}[label = (\alph*)]

    \item Suppose $O\in \tau'\backslash \tau$. Since $O^c$ is closed in $(X, \tau')$, it is compact with respect to $\tau'$, and hence also compact with respect to $\tau$. However, $O^c$ is not closed in $\tau$, which leads to a contradiction.

    \item Clearly $A\in \tau^A$ and $\tau\subseteq \tau^A$. Suppose $\tau'$ is another topology that is finer than $\tau$ and contains $A$. Then for any $U, V\in\tau$, we have $U\cup (V\cap A) \in \tau'$, which shows that $\tau^A \subseteq \tau'$.

    \item Suppose $K\subseteq X$ is a compact set that is not closed with respect to $\tau$. Consider $\tau'$ the simple extension of $\tau$ by $K^c$. Since $K^c\notin \tau$ we have $\tau\subsetneq \tau'$. Fix $\big( O_i\cup (V_i\cap K^c) \big)_{i\in I}$ an $\tau'$-open cover of $X$. Then we can find $I_1, I_2\subseteq I$ such that:

    $$
    K \subseteq \bigcup_{i\in I_1}O_i, \hspace{1cm} K^c\subseteq \bigcup_{i\in I_2} \big( O_i\cup V_i \big)
    $$
    Since both $X$ and $K$ are compact with respect to $\tau$, there exists two finite sets $F_1\subseteq I_1$, $F_2\subseteq I_2$ such that:

    $$
    K_1\subseteq \bigcup_{i\in F_1}O_i, \hspace{1cm} K^c\subseteq \bigcup_{i\in F_2} (O_i\cup V_i)
    $$
    For any $x\in K^c$, if $x\in O_i$ for some $i\in F_2$, then we are good; otherwise, we will have $x\in V_i\cap K^c$ for some $i\in F_2$. We can now conclude that $\big( O_i\cup (V_i\cap K^c): i\in F_1\cup F_2 \big)$ is a finite subcover of $\big( O_i\cup (V_i\cap K^c) \big)_{i\in I}$. Hence $X$ is compact with respect to $\tau'$. With $(a)$ and $(c)$, we can now conclude $\tau$ is maximally compact if and only if $\tau$ is a compact topology of $X$ and $(X, \tau)$ is \textbf{KC}.
    
\end{enumerate}
    
\end{proof}

\section{Summary}

\begin{center}
\begin{tikzpicture}
    \node (A) at (1,0) {$T_0$};
    \node (B) at (0,1) {sober};
    \node (C) at (2,1) {$T_D$};
    \node (D) at (1,2) {$T_D$ $+$ sober};
    \node (E) at (-3, 2){hereditarily sober};
    \node (F) at (3,2) {$T_1$};
    \node (G) at (2,3) {$T_1$ $+$ sober};
    \node (H) at (4,3) {\textbf{SH}};
    \node (I) at (5,4) {\textbf{SC}};
    \node (J) at (6,5) {\textbf{WH}};
    \node (K) at (4,5) {$\textbf{KC}_{\omega}$};
    \node (L) at (5,6) {\textbf{KC}};
    \node (M) at (4,7) {Hausdorff};
    \node (N) at (2,5) {locally Hausdorff};

    \draw[->] (M) -- (L);
    \draw[->] (M) -- (N);
    \draw[->] (L) -- (K);
    \draw[->] (L) -- (J);
    \draw[->] (K) -- (I);
    \draw[->] (J) -- (I);
    \draw[->] (N) -- (G);
    \draw[->] (I) -- (H);
    \draw[->] (G) -- (D);
    \draw[->] (H) -- (F);
    \draw[->] (G) -- (F);
	\draw[->] (D) -- (C);
	\draw[->] (F) -- (C);
	\draw[->] (D) -- (B);
	\draw[->] (B) -- (A);
	\draw[->] (C) -- (A);
	\draw[<->](D) -- (E);
    
\end{tikzpicture}
\end{center}

\begin{center}
\begin{tikzpicture}
    \node (A) at (1,0) {$R_d$};
    \node (B) at (0,1) {$R_0$};
    \node (C) at (2,1) {$R_d$ $+$ quai-sober};
    \node (D) at (3,0) {quasi-sober};
    \node (E) at (7,1) {hereditarily quasi-sober};
    \node (F) at (1,2) {$R_0$ $+$ quasi-sober};
    \node (G) at (1,3) {locally $R_1$};
    \node (H) at (1,4) {$R_1$};
    \node (I) at (1,5) {weakly Urysohn};
    \node (J) at (1,6) {regular};
    \node (K) at (1,7) {completely regular};
    \node (L) at (-1,8) {$R_0$ $+$ normal};
    \node (M) at (3,8) {zero-dimensional};
    \node (N) at (5,9) {almost discrete};
	\node (O) at (5,10) {Alexandroff};
	\node (Q) at (-4,8) {$R_1$ $+$ locally compact};
	\node (R) at (-3,10) {$R_1$ $+$ compact};

    \draw[->] (R) -- (Q);
    \draw[->] (R) -- (L);
    \draw[->] (Q) -- (K);
    \draw[->] (L) -- (K);
    \draw[->] (N) -- (M);
    \draw[->] (O) -- (N);
    \draw[->] (M) -- (K);
    \draw[->] (K) -- (J);
    \draw[->] (J) -- (I);
    \draw[->] (I) -- (H);
    \draw[->] (H) -- (G);
	\draw[->] (G) -- (F);
	\draw[->] (F) -- (B);
	\draw[->] (F) -- (C);
	\draw[<->] (C) -- (E);
	\draw[->] (B) -- (A);
	\draw[->](C) -- (A);
	\draw[->] (C) -- (D);
    
\end{tikzpicture}
\end{center}

\bibliographystyle{amsalpha}
\bibliography{Reference_List}

\end{document}